\documentclass{amsart}
\usepackage{mathrsfs,amssymb,amsmath,amsthm,stmaryrd}
\usepackage[all,cmtip]{xy}
\usepackage{cite}

\usepackage[T1]{fontenc}
\usepackage[french, english]{babel}

\usepackage{fancyhdr}
\newtheorem{theorem}{Theorem}[section]
\newtheorem*{theorem*}{Theorem}
\newtheorem{lemma}[theorem]{Lemma}
\newtheorem{proposition}[theorem]{Proposition}
\newtheorem{corollary}[theorem]{Corollary}

\theoremstyle{definition}
\newtheorem{definition}[theorem]{Definition}
\newtheorem{example}[theorem]{Example}

\theoremstyle{remark}
\newtheorem{remark}[theorem]{Remark}

\numberwithin{equation}{subsection}

\newcommand{\op}[1]{\operatorname{#1}}
\newcommand{\p}{\prime}

\DeclareMathOperator\SU{SU}
\DeclareMathOperator\U{U}
\DeclareMathOperator\SO{SO}
\DeclareMathOperator\Sp{Sp}

\begin{document}

\title{Construction of higher-dimensional ALF Calabi-Yau metrics}

\author{Daheng Min}
\address{Sorbonne Universit\'{e} and Universit\'{e} Paris Cit\'{e}, CNRS, IMJ-PRG, F-75005 Paris, France.}
\email{min@imj-prg.fr}
\date{June 2, 2023}

\selectlanguage{english}
\begin{abstract}
Roughly speaking, an ALF metric of real dimension $4n$ should be a metric such that it has a ($4n-1$)-dimensional asymptotic cone, the volume growth of this metric is of order $4n-1$ and its sectional curvature tends to 0 at infinity.

In this paper, we first show that the Taub-NUT deformation of a hyperk\"ahler cone with respect to a locally free $\mathbb{S}^1$-symmetry is ALF hyperk\"ahler. Using this metric at infinity, we establish the existence of ALF Calabi-Yau metric on certain crepant resolutions. In particular, we prove that there exist ALF Calabi-Yau metrics on canonical bundles of classical homogeneous Fano contact manifolds.
\end{abstract}


\maketitle

\section{Introduction}

\subsection{Motivation}
After Yau's celebrated confirmation of Calabi conjecture \cite{Yau78} which implies the existence of a Ricci-flat K\"ahler metric on a compact K\"ahler manifold satisfying certain
condition, there have been many works on non-compact Calabi-Yau metrics. Among those works, a particular class introduced by Hawking \cite{Hawking:1976jb}, namely gravitational
instanton, is of special interest. They are complete hyperk\"ahler 4-manifolds with a decaying curvature at infinity.

Under the assumption of finite energy, the gravitational instantons are classified into four types according to their `dimension $m$ at infinity' (see
\cite[Section 6.4]{sun2021collapsing} for the precise definition of the following spaces). We have ALE ($m=4$), ALF ($m=3$), ALG and ALG* ($m=2$), ALH and ALH* ($m=1$) gravitational instantons. (See also \cite{Chen-Chen1} for a classification under the assumption of a faster than quadratic curvature decay. In this case, it means that the metric is locally asymptotic to a $\mathbb{T}^{4-m}$-fibration over $\mathbb{R}^m$ near the infinity).

The gravitational instantons have been very well understood. Examples of ALE gravitational instantons are found by Eguchi and Hanson \cite{Eguchi:1978xp}, Gibbons and Hawking
\cite{Gibbons-Hawking78} and Hitchin \cite{Polygonsandgravitons}. A complete classification of ALE gravitational instantons is given by Kronheimer
\cite{Kronheimer1, Kronheimer2}, showing that ALE gravitational instantons are crepant resolutions of $\mathbb{C}^2/G$ where $G$ is a finite subgroup of $\SU(2)$. For ALF
gravitational instantons, examples of cyclic type, known as multi-Taub-NUT metrics, are constructed by Hawking \cite{Hawking:1976jb}, while Cherkis and Kapustin
\cite{Cherkis-Kaupstin1, Cherkis-Kaupstin2} find the dihedral type. The cyclic type is classified by Minerbe \cite{Minerbe} and then a complete classification of ALF instantons
is given by Chen and Chen \cite{Chen-Chen2}.

Certain aspects of ALE gravitational instantons can be generalized to higher dimensions. Bando, Kasue and Nakajima \cite{BKN} show that faster than quadratic curvature decay and
Euclidean volume growth imply ALE property. Joyce \cite{joyce2000compact} systematically studies certain resolutions of $\mathbb{C}^m/G$ where $G$ is a finite subgroup of $\U(m)$
that may not act freely on $\mathbb{C}^m\setminus\{0\}$. More generally, if we allow for the quadratic decay of curvature, or arbitrary asymptotic cone, then asymptotically conic
Calabi-Yau manifolds can be thought as a generalization of ALE instantons to higher dimensions. To this end, in addition to Joyce's works, we have works by Goto \cite{Goto} and Van
Coevering \cite{VanCoevering} which study the crepant resolution of a Calabi-Yau cone. We also have works by Conlon and Hein \cite{Conlon-Hein1} which study both resolutions and deformations of Calabi-Yau cones.

Compared to ALE instantons, less is known about how to generalize the theory of ALF instantons to higher dimensions. Chen and Li \cite{Chen-Li} show that under a faster than
quadratic curvature decay and some holonomy control, the metric is asymptotically a torus fibration over an ALE space. However, there is no known example of a higher-dimensional ALF Calabi-Yau manifold satisfying their assumptions.

In analogy to ALF instantons, we propose the following notion of a higher-dimensional ALF metric. An ALF metric of real dimension $4n$ should be
a complete metric such that it has a ($4n-1$)-dimensional asymptotic cone, the volume growth of this metric is of the order $4n-1$ and its sectional curvature tends to 0 at infinity. In
this article, we give an affirmative response to the following question: Does there exist an ALF Calabi-Yau metric of dimension strictly larger than $4$?

\subsection{Statement of results}

The first main result of this article is as follows.
\begin{theorem}\label{theorem introduction main theorem part 1}
Let $S$ be a compact connected 3-Sasakian manifold of dimension $4n-1$ ($n\geq 1$) admitting a locally free $\mathbb{S}^1$-symmetry. Let $(M,g_0, I_i, \omega_i)$ be the hyperk\"ahler cone over $S$, then for any $a > 0$, a certain deformation (known as the Taub-NUT deformation) $(M, g_a, I_i^a,\omega_i^a)$ with respect to this $\mathbb{S}^1$-symmetry is an ALF hyperk\"ahler metric on $M$ in the sense that the sectional curvature $K_{g_a}$ is bounded by $\frac{C}{\rho_a}$ for some $C>0$, the volume growth of $g_a$ is of order $4n-1$ and its asymptotic cone is $M_0\times \mathbb{R}^3$, which is a metric cone of dimension $4n-1$. Here $\rho_a$ is the distance function measured by the deformed metric $g_a$, and $M_0$ is the hyperk\"ahler quotient of $M$ with respect to the $\mathbb{S}^1$-symmetry and hyperk\"ahler moment $0\in \mathbb{R}^3$.
\end{theorem}
We will explain the Taub-NUT deformation in Section \ref{section hk metrics with S1 symmetry}, here we point out that in the special case of $S = \mathbb{S}^3$, the resulting metric in Theorem \ref{theorem introduction main theorem part 1} is the Taub-NUT metric on $\mathbb{C}^2$, which is known to be ALF. The passage from the Euclidean metric on $\mathbb{C}^2$ to the Taub-NUT metric is known as the Gibbons-Hawking deformation, and the Taub-NUT deformation is a generalization of the Gibbons-Hawking deformation to higher dimensions. In this way, we think of Theorem \ref{theorem introduction main theorem part 1} as a generalization of the Taub-NUT metric. A special case of $S = \mathbb{S}^{4n-1}$ of the theorem produces the so-called Taubian-Calabi metric; in this case $M_0$ is a toric hyperk\"ahler cone depending on the choice of the $\mathbb{S}^1$-symmetry. According to \cite{gibbons1997hyperkahler}, the name Taubian-Calabi is due to Ro\v{c}ek \cite{Rocek:1983ja}. It follows that Theorem \ref{theorem introduction main theorem part 1} implies that the Taubian-Calabi metric has an ALF property. Besides the case $S = \mathbb{S}^{4n-1}$, we also give other applications of Theorem \ref{theorem introduction main theorem part 1} in Section \ref{section Examples}.

One feature of Theorem \ref{theorem introduction main theorem part 1} is that it allows us to produce many ALF hyperk\"ahler metrics on the same complex manifold. As an example, we will show that, in addition to the Taubian-Calabi metric, there exist infinitely many ALF hyperk\"ahler metrics on $\mathbb{C}^{2n}$ with different asymptotic cones (see Example \ref{example C^2n generalized}). This phenomenon also appears in the context of asymptotically conic Calabi-Yau metric, we mention the works \cite{Yangli2017}, \cite{ConlonRochon2021} and \cite{zbMATH07131296} which give counterexamples to a conjecture of Tian \cite[Remark 5.3]{zbMATH05234298}.

Another feature of Theorem \ref{theorem introduction main theorem part 1} is that the asymptotic cone $M_0\times \mathbb{R}^3$ is generally not a smooth cone. For example, $\{0\}\times \mathbb{R}^3$ may be a singular locus. We will see in Section \ref{section hk metrics with locally free S1 action} and Section \ref{section Twist coordinates} how this singularity in the asymptotic cone is resolved in the ALF metric. Examples of asymptotically conic Calabi-Yau metrics with singular asymptotic cone are studied by Joyce \cite{joyce2000compact} (QALE manifold), Sz\'{e}kelyhidi \cite{zbMATH07131296}, Yang Li \cite{Yangli2017}, Conlon, Degeratu and Rochon \cite{ConlonRochon2019, ConlonRochon2021}.

The proof of Theorem \ref{theorem introduction main theorem part 1} is based on a careful study of the Taub-NUT deformation from Section \ref{section hk metrics with S1 symmetry} to Section \ref{section Twist coordinates}.

Observe that as long as the 3-Sasakian manifold is not isometrically the unit sphere, its cone $M$ is not smooth at the vertex. Thus, we are led to consider resolution of quotients of $M$. This is the motivation of the second main result:
\begin{theorem}\label{theorem introduction main theorem part 2}
With the assumptions of Theorem \ref{theorem introduction main theorem part 1}, suppose furthermore that $n\geq 2$ and there is a finite group $\Gamma$ acting on $S$ whose extension to $M$ preserves one of the complex structures $I_1$ and its deformed K\"ahler potential $K_{\mathrm{ALF}}^a$.
Let $\pi: Y\rightarrow M/\Gamma$ be an $I_1-$holomorphic crepant resolution, then for any compactly supported K\"ahler class of $Y$ and any $c > 0$, there exists an ALF Calabi-Yau metric $\omega$ in this class which is asymptotic to $c\omega_1^a$ near the infinity. More precisely, we have
\begin{align}
|\nabla^k(\omega - c\pi^*(i\partial\bar\partial K_{\mathrm{ALF}}^a))|_\omega \leq C(k,\varepsilon)(1 + \rho_\omega)^{-4n + 3 + \varepsilon},
\end{align}
where $\varepsilon > 0$ is sufficiently small, $\rho_\omega$ is the distance from a point in $Y$ measured by $\omega$ and $k\geq 0$. Here $K_{\mathrm{ALF}}^a$ is invariant by $\Gamma$ so we think of it as a function on $M/\Gamma$.
\end{theorem}
In the work of Van Coevering \cite{VanCoevering}, asymptotically conic Calabi-Yau metrics are constructed in each compactly supported K\"ahler class of crepant resolution of Ricci-flat K\"ahler cone. So, Theorem \ref{theorem introduction main theorem part 2} should be understood as an ALF analogue of the result of \cite{VanCoevering}. We also have Theorem \ref{theorem isolated singularity} as an ALF analogue of the work of Joyce \cite{joyce2000compact} and Theorem \ref{theorem twistor space} as an ALF analogue of the work of Calabi \cite{calabi1979metriques}. Both of the two theorems are deduced from Theorem \ref{theorem introduction main theorem part 2}.

We assume $n\geq 2$ in Theorem \ref{theorem introduction main theorem part 2}. In most applications of the theorem, the underlying manifold $Y$ is generally not hyperk\"ahler, and in this case the Taub-NUT deformation is not applicable to $Y$. This turns out to be quite different from the case $n=1$, where the ALF-$A_k$ instantons are obtained by a Gibbons-Hawking deformation of the ALE instantons. This partially justifies our indirect approach. On the other hand, in the case $n=1$, there is a ``Kummer construction'' of ALF-$D_k$ instantons discussed in the work of Biquard and Minerbe \cite{biquard2011kummer}. So we may also think of Theorem \ref{theorem introduction main theorem part 2} as a higher-dimensional analogue of their work.

As for applications of Theorem \ref{theorem introduction main theorem part 2}, we have
\begin{corollary}\label{corollary introduction}
Let $S$ be one of the following homogeneous 3-Sasakian manifolds: $\Sp(n)/\Sp(n-1)$, ($n\geq 2$), $\SU(m)/\op{S}(\U(m-2)\times \U(1))$, ($m\geq 3$), $\SO(l)/(\SO(l-4)\times \Sp(1))$, ($l\geq 5$), $G_2/\Sp(1)$, then there exist infinitely many different 3-Sasakian $\mathbb{S}^1-$symmetries of $S$. Let $(M,g_0,\omega_i,I_i)$ be the hyperk\"ahler cone over $S$, fix any $\mathbb{S}^1$-symmetry of $S$ and let $(M, g_a, \omega_i^a, I_i^a)$ be the Taub-NUT deformation ($a>0$). Then for any $\varepsilon >0$, $c>0$, there exists an ALF Calabi-Yau metric on $\mathcal{K}_Z$ in the class $\varepsilon c_1(Z)$ which is asymptotic to $c\omega_1^a$. Here $Z$ is the twistor space of $S$ which is a homogeneous Fano contact manifold, $\mathcal{K}_Z$ is the canonical bundle of $Z$, and $c_1(Z)$ denotes the first Chern class of $Z$.
\end{corollary}

For example, if $S = \mathbb{S}^{4n-1} = \Sp(n)/\Sp(n-1)$, then Corollary \ref{corollary introduction} implies that there are ALF Calabi-Yau metrics on the total space $\mathcal{K}_{\mathbb{CP}^{2n-1}}$ of the canonical bundle of $\mathbb{CP}^{2n-1}$ and one of them is asymptotic to the Taubian-Calabi metric.

The proof of Theorem \ref{theorem introduction main theorem part 2} follows the approach of Tian-Yau's work \cite{zbMATH04186562, zbMATH00059791}, which is a non-compact version of the classical Calabi-Yau theorem.

Finally, let us discuss some aspects that are still open.

It is worth pointing out that Theorem \ref{theorem introduction main theorem part 2} is an existence result. Since we have already found many different ALF Calabi-Yau metrics on the same complex space, the uniqueness problem should be proposed more or less as follows: Given a prescribed asymptotic cone and K\"ahler class, is the ALF Calabi-Yau metric unique? We mention \cite{sze2020uniqueness} and \cite{conlonhein2024} as results of this type for asymptotically conic Calabi-Yau metrics.

Both Theorem \ref{theorem introduction main theorem part 1} and Theorem \ref{theorem introduction main theorem part 2} produce ALF Calabi-Yau metrics of even complex dimension. So, it is open whether there exists an ALF Calabi-Yau metric of odd complex dimension.

As we shall see in Proposition \ref{estimate of K, generalized}, the curvature decay of the Taub-NUT deformation given by Theorem \ref{theorem introduction main theorem part 1} is of order $O(\frac{1}{\rho})$. In \cite{zbMATH07698520}, it is proved that for a complete Calabi-Yau metric of maximal voulume growth, the quadratic curvature decay ($|\op{Rm}|\leq \frac{C}{\rho^2}$) is equivalent to being asymptotic to a Calabi-Yau cone with smooth link. Then it is natural to ask whether there is a similar result for non-maximal volume growth. For example, for ALF Calabi-Yau metrics, is there a relation between the quadratic curvature decay and the smoothness of the link of asymptotic cone? In the four-dimensional case, we know that ALF gravitational instantons have faster than quadratic curvature decay and smooth asymptotic cone. But in higher dimensions, we have no examples of (non-trivial) ALF Calabi-Yau metrics with quadratic curvature decay or smooth asymptotic cone, and it will be interesting to find such examples. Intuitively, guided by the case of maximal volume growth mentioned above, the author believes that the simultaneous occurrence of $|Rm| = O(\frac{1}{\rho})$ and the singular link of asymptotic cone in his construction is not a coincidence (see also Example \ref{example EH blown down} for a concrete example).

In a recent work by Apostolov and Cifarelli \cite{apostolov2023hamiltonian}, complete Ricci-flat K\"ahler metrics on $\mathbb{C}^n$ ($n\geq 2$) whose volume growth are of order $2n-1$ are constructed (\cite[Theorem 1.4]{apostolov2023hamiltonian}). It will be interesting to understand whether there is a link between their construction and Theorem \ref{theorem introduction main theorem part 1} when $M=\mathbb{C}^{2m}$.

\subsection{Organization of the article}
Let $(M,g,I_i,\omega_i)$ be a hyperk\"ahler manifold, which means that for each $i = 1,2,3$, the triple $(g,I_i,\omega_i)$ is a K\"ahler structure on $M$, and the three complex structures satisfy the relation $I_1I_2I_3 = -1$. Assume that there is an $\mathbb{R}$-action on $M$ which preserves the hyperk\"ahler structure.

Consider the hyperk\"ahler manifold $\mathbb{H}$. The $\mathbb{R}$-translation on its first real coordinate is tri-holomorphic and isometric. The hyperk\"ahler quotient of $M\times\mathbb{H}$ by $\mathbb{R}$ is known to be a Taub-NUT deformation of $M$. In Section \ref{section hk metrics with S1 symmetry}, we will give a description of the deformed hyperk\"ahler structure in Proposition \ref{deformation}, which is a natural generalization of the Gibbons-Hawking ansatz. In particular, the underlying smooth manifold of the hyperk\"ahler quotient is diffeomorphic to $M$, and we will denote by $(M,g_a,I_1^a,I_2^a,I_3^a,\omega_1^a,\omega_2^a,\omega_3^a)$ the deformed hyperk\"ahler structure.

In Section \ref{section Basic properties of the main construction}, we deduce some properties of the Taub-NUT deformation of hyperk\"ahler cone. We show in Proposition \ref{estimate of K, generalized} that the sectional curvature of $g_a$ is bounded by $\frac{C}{\rho_a}$, where $\rho_a$ is the distance function of $g_a$. In Proposition \ref{complex symplectic
structure} we show that there is a complex symplectomorphism $\Phi_a$ from $(M,I_1^a,\omega_2^a + i\omega_3^a)$ to $(M,I_1,\omega_2 + i\omega_3)$ and in Proposition \ref{deformed
kahler potential} we give a formula of $I_i^a-$K\"ahler potential of $\omega_i^a$. In Proposition \ref{Sp(1) acts by isometry} and \ref{Inner left Sp(1) action} we show that the
$\Sp(1)$-action generated by the Reeb fields on $M$ preserves the metric $g_a$ but rotates the three complex structures $I_i^a$.

In Section \ref{section hk metrics with locally free S1 action}, we assume moreover that the $\mathbb{R}$-action on the hyperk\"ahler cone $M=C(S)$ is a locally free $\mathbb{S}^1$-symmetry. For $x\in \mathbb{R}^3$, denote by $M_x = \mu^{-1}(x)/\mathbb{S}^1$ the hyperk\"ahler quotient of $M$ with respect to the $\mathbb{S}^1$-action at the moment $x$, here $\mu$ denotes the hyperk\"ahler moment map. With the help of the Morse-Bott theory, in Corollary \ref{corollary S_0 is connected} and Lemma \ref{lemma P is diffeomorphism} we will show that $M_x$ is connected. And we will show in Subsection \ref{subsection the hk quotients of M} that for $x\neq 0$, $M_x$ is a resolution of $M_0$ in a certain sense.

In Section \ref{section Twist coordinates}, we choose a coordinate system and calculate the metric tensor $g_a$ in terms of these coordinates. This enables us to deduce the
asymptotic behavior of $g_a$ in Subsection \ref{subsection the asymptotic behavior}. Moreover, in Proposition \ref{volume growth} we show that the volume growth of $g_a$ is of order
$4n-1$ where $\dim_\mathbb{R}M = 4n$ and in Proposition \ref{asymptotic cone} we prove that the asymptotic cone of $g_a$ is the product $M_0\times\mathbb{R}^3$, finishing the proof of Theorem \ref{theorem introduction main theorem part 1}.

Section \ref{section ALF CY metrics on crepant resolutions} is devoted to the proof of Theorem \ref{theorem introduction main theorem part 2} using the approach of Tian-Yau \cite{zbMATH04186562, zbMATH00059791}.
More precisely, in Subsection \ref{subsection asymptotic ALF CY metric} we construct an approximately Calabi-Yau metric $\hat\omega$ on $Y$ by gluing $\omega_Y$ and
$\omega_1^a$. After that, we apply a result of Hein \cite{Heinthesis} to show the existence of the solution of the Monge-Amp\`{e}re equation, and this in turn produces a
genuine Calabi-Yau metric.

Finally, applications of Theorem \ref{theorem introduction main theorem part 1} and Theorem \ref{theorem introduction main theorem part 2} are discussed in Section \ref{section Examples}. In addition to Corollary \ref{corollary introduction}, we will also prove an existence result for the ALF Calabi-Yau metrics on the crepant resolution of isolated singularity (see Theorem \ref{theorem isolated singularity}).

\section{Hyperk\"ahler metrics with $\mathbb{R}$-symmetry} \label{section hk metrics with S1 symmetry}

\subsection{Decomposition of the metric} \label{subsection decomposition of the metric}

Let $(M,g,I_i,\omega_i)$ be a hyperk\"ahler manifold. Suppose that there is a non-vanishing vector field $T$
on $M$ which preserves the hyperk\"ahler structure, i.e. $L_Tg=0,L_TI_i=0,L_T\omega_i=0$, for $i=1,2,3$, where $L_T$ stands for the Lie derivative.
So $T$ is Killing with respect to $g$, holomorphic with respect to $I_i$. For example, if there is a locally free $\mathbb{S}^1$-action on $M$ which preserves its hyperk\"ahler
structure, then the generator $T$ of this $\mathbb{S}^1$-action will satisfy the above conditions.

Define $\theta$ as the 1-form dual to $T$ with respect to $g$, that is, $\theta(Y)=g(T,Y)$ for any tangent vector $Y$. For $i=1,2,3$, define
$T_i=I_iT$, and define $\theta_i$ as the dual of $T_i$ with respect to $g$. It follows that $\theta_i=I_i\theta$. Since $g$ is invariant by $I_i$, it follows that $T,T_1,T_2,T_3$ are mutually orthogonal and consequently $\theta,\theta_1,\theta_2,\theta_3$ are mutually orthogonal with respect to the induced metric on $T^*M$. Since
the hyperk\"ahler structure and $T$ itself are invariant by $T$, it follows that $T_i$, $\theta,\theta_i$ are invariant by $T$.

Define $V=\frac{1}{g(T,T)}=|T|^{-2}_g$, $\eta=V\theta$, then $\eta(T)=1$. Here $V$ and $\eta$ are well defined, since we have assumed that $T$ is non-vanishing. And they are also invariant by $T$. Now
$\sqrt{V}T,\sqrt{V}T_1,\sqrt{V}T_2,\sqrt{V}T_3$ form an orthonormal set of vector fields, hence
$\sqrt{V}\theta=\frac{1}{\sqrt{V}}\eta,\sqrt{V}\theta_1,\sqrt{V}\theta_2,\sqrt{V}\theta_3$ form an orthonormal set of 1-forms.

Denote by $L\subset TM$ the vector subbundle of $TM$ generated by $T,T_1,T_2,T_3$, and denote by $L^{\perp}$ the orthogonal complement of $L$ in $TM$
with respect to $g$ so that $TM=L^{\perp}\oplus L$ and $g(L^\perp,L)=0$. Since $T_i=I_iT$, it is clear that $L$ is invariant by $I_i$, that is to say $I_iL=L$. Note
that $g$ is $I_i$-invariant, so $L^\perp$ is also invariant by $I_i$. In other words, we have $(I_iX)^\perp=I_iX^\perp$, where $X^\perp$ is the
orthogonal projection of $X$ onto $L^\perp$. In this way, we decompose at each point $m\in M$ the tangent space $T_mM$ into a direct sum of two
orthogonal $\mathbb{H}$-linear subspaces.

The above discussion shows that:

\begin{align}
g|_L=V\left(\sum_{i=1}^3\theta_i^2 + \theta^2\right)=V\sum_{i=1}^3\theta_i^2 + \frac{1}{V}\eta^2.
\end{align}
Denote by $\bar{g}$ the ``restriction'' of $g$ on $L^\perp$, that is to say $\bar{g}(X,Y)=g(X^\perp,Y^\perp)$. Then we have
\begin{align}
g=\bar{g} + V\sum_{i=1}^3\theta_i^2 + \frac{1}{V}\eta^2.
\end{align}

Locally we may choose $x_i$ to be a moment map of $T$ with respect to $\omega_i$:
\begin{align}
dx_i=-\iota_T\omega_i.
\end{align}

Then it follows that $dx_i = -\theta_i$, so by the calculations above we have
\begin{align}\label{g}
g=\bar{g} + V\sum_{i=1}^3 (dx_i)^2 + \frac{1}{V}\eta^2.
\end{align}

One verifies that $I_1dx_1=\frac{1}{V}\eta$, $I_1dx_2=dx_3$, so by formula \eqref{g} we get
\begin{align}
\omega_1=\bar{\omega}_1+dx_1\wedge\eta + Vdx_2\wedge dx_3,
\end{align}
where $\bar{\omega}_i(X,Y)=\omega_i(X^\perp,Y^\perp)$ is the ``restriction'' of $\omega_i$ to $L^\perp$. Similarly for any cyclic permutation $(i,j,k)$ of $(1,2,3)$ we have
\begin{align} \label{omegai}
\omega_i=\bar{\omega}_i+dx_i\wedge\eta + Vdx_j\wedge dx_k.
\end{align}

\begin{remark}
When $\op{dim}_\mathbb{R}M=4$, the function $V$ is a harmonic function of $x_1,x_2,x_3$ and $d\eta$ is a 2-form on $\mathbb{R}^3$, see for example \cite{Gibbons-Hawking78} for the
Gibbons-Hawking ansatz. However, in general, the function $V$ is not a function of
$x_1,x_2,x_3$ and $d\eta$ cannot be considered a 2-form in $\mathbb{R}^3$.
\end{remark}

If the functions $x_1,x_2,x_3$ are well defined globally in $M$, then the map $\mu:M\rightarrow \mathbb{R}^3$ given by
$\mu(m)=(x_1(m),x_2(m),x_3(m))$ is the hyperk\"ahler moment map. Suppose that there is a $\lambda\in \mathbb{R}^3$ such that $T$ generates a free
$\mathbb{S}^1$-action on $\mu^{-1}(\lambda)$, then one can form the hyperk\"ahler quotient $Q_\lambda=\mu^{-1}(\lambda)/\mathbb{S}^1$ which has a natural hyperk\"ahler
structure. It is the space of $\mathbb{S}^1$-orbits in $\mu^{-1}(\lambda)$. Let $q\in Q_\lambda$ and let $m\in \mu^{-1}(\lambda)$ be any point in the
orbit associated with $q$. Then the tangent space $T_qQ_\lambda$ can be identified with $L^\perp_m\subset T_mM$. Under this identification, the
hyperk\"ahler structure of $Q_\lambda$ is given by the restrictions $g|_{L^\perp}, \omega_i|_{L^\perp},I_i|_{L^\perp}$. From this point of view, we can think of
$\bar{g}$ (resp. $\bar{\omega}_i$) as the Riemannian metric (resp. K\"ahler form) of the hyperk\"ahler quotient.

\begin{example} \label{H}
We denote by $\mathbb{H}$ the (skew) field of quaternions. One can identify $\mathbb{H}$ with $\mathbb{R}^4$ by identifying each quaternion
$u=q_0+q_1i+q_2j+q_3k$ with $(q_0,q_1,q_2,q_3)\in \mathbb{R}^4$. The left multiplication by $i,j,k$ in $\mathbb{H}$ defines three
integrable almost complex structures on $\mathbb{H}$, denoted by $I_1,I_2,I_3$.

The standard Euclidean metric on $\mathbb{H}$ is given by
\begin{align} \label{gH}
g_\mathbb{H}=\sum_{\alpha=0}^3(dq_\alpha)^2.
\end{align}
Let $\kappa_\alpha$ be the corresponding K\"ahler form given by $\kappa_\alpha(X,Y)=g_\mathbb{H}(I_\alpha X,Y)$, then by direct computation we have
\begin{align}\label{kappaalpha}
\kappa_\alpha&=dq_0 \wedge dq_\alpha + dq_\beta\wedge dq_\gamma,
\end{align}
where $(\alpha,\beta,\gamma)$ is any cyclic permutation of $(1,2,3)$. It follows that $d\kappa_\alpha=0$, so $(g_\mathbb{H},I_1,I_2,I_3,\kappa_1,\kappa_2,\kappa_3)$ is a
hyperk\"ahler structure on $\mathbb{H}$.

Let $T=\frac{\partial}{\partial q_0}$, then $T$ is a nonvanishing vector field on $\mathbb{H}$ which preserves its hyperk\"ahler structure. In fact,
$T$ is the generator of the $\mathbb{R}$-action of translation on $\mathbb{H}$ defined by $s\cdot q=q+s$, for $s\in \mathbb{R}, q\in \mathbb{H}$. Then we
have $V=\frac{1}{g_\mathbb{H}(T,T)}=1$, and $-q_\alpha$ is the moment map of $T$ with respect to $\kappa_\alpha$.

It follows that if we apply the decomposition \eqref{g} to this example, we will get \eqref{gH}, and if we apply the decomposition \eqref{omegai} to this example, then we will
get \eqref{kappaalpha}.

Sometimes we want an action of $\mathbb{S}^1$ rather than $\mathbb{R}$. The map $\mathbb{H}\rightarrow \mathbb{S}^1\times\mathbb{R}^3$ defined by
$(q_0,q_1,q_2,q_3)\mapsto(e^{iq_0},q_1,q_2,q_3)$ gives $\mathbb{S}^1\times\mathbb{R}^3$ the quotient hyperk\"ahler structure, and the $\mathbb{R}$-action on $\mathbb{H}$ descends to
an $\mathbb{S}^1$-action on $\mathbb{S}^1\times\mathbb{R}^3$ which preserves its hyperk\"ahler structure.
\end{example}

\begin{remark}
In \cite{kronheimer2004hyperkahler}, it is proved that for any compact Lie group $G$, there is a hyperk\"ahler structure on the cotangent bundle
$T^*G^c$ of its complexification $G^c$ which is invariant by the left and right action of $G$, and that if we take $G=\mathbb{S}^1$, then the hyperk\"ahler
manifold $T^*(\mathbb{S}^1)^c$ is exactly the hyperk\"ahler manifold $\mathbb{S}^1\times \mathbb{R}^3$ that we just considered. Moreover, the left or right
action of $\mathbb{S}^1$ on $T^*(\mathbb{S}^1)^c$ is exactly the action we have pointed out above.
\end{remark}

\begin{example} \label{C^2n}
Consider $\mathbb{C}^{2n}=\mathbb{C}^2\times\dots\times\mathbb{C}^2$ for $n\geq 1$, with complex coordinates $(z_1,w_1,\dots,z_n,w_n)$. Its hyperk\"ahler structure is obtained by
identifying each $\mathbb{C}^2$ component with $\mathbb{H}$ described in Example \ref{H} as follows: Fix $1\leq a \leq n$, the pair $(z_a,w_a)\in \mathbb{C}^2$ is identified with
$z_a+w_aj\in \mathbb{H}$.

So the Euclidean metric $g_0$ is defined by
\begin{align}
g_0 = \sum_{a=1}^n|dz_a|^2 + \sum_{a=1}^n|dw_a|^2.
\end{align}

Define the $\mathbb{S}^1$-action on $\mathbb{C}^{2n}$ by
\begin{align}
e^{it}\cdot(z_1,w_1,\dots,z_n,w_n)=(e^{it}z_1,e^{-it}w_1,\dots,e^{it}z_n,e^{-it}w_n).
\end{align}
Here we note that since $(z_a+w_aj)e^{it} = e^{it}z_a + e^{-it}w_aj$, we know that after identifying $\mathbb{C}^{2n}$ with $\mathbb{H}^n$, the $\mathbb{S}^1$-action can be viewed
as the right multiplication with $e^{it}$. Since $e^{it}$ is of the norm $1$, this action preserves the metric. And since left multiplication commutes with right multiplication, we know
that this action preserves the complex structures, hence the three K\"ahler forms.

Then we have $V=\frac{1}{\rho^2}$, where $\rho$ is the radius defined by $\rho^2=\sum_{a=1}^n|z_a|^2 + \sum_{a=1}^n|w_a|^2$.

Then we may define the moment maps $x_1,x_2,x_3$ by
\begin{align}\label{x_1 of C2n}
x_1&= \frac{1}{2}\sum_{a=1}^n|z_a|^2 - \frac{1}{2}\sum_{a=1}^n|w_a|^2,\\
x_2+ix_3 &= -i\sum_{a=1}^n z_aw_a.\label{x_2 x_3 of C2n}
\end{align}
This can be verified directly, or one can deduce this formula by Proposition \ref{moment map of cone}.

If one applies the decomposition \eqref{g} to this example, one will get
\begin{align}
g_0=\bar{g} + \frac{1}{\rho^2}\sum_{i=1}^3 (dx_i)^2 + \rho^2\eta^2,
\end{align}
here $\eta = Vg_0(T,-)$ is the connection $1$-form of the $\mathbb{S}^1$-action.
\end{example}

\subsection{Hyperk\"ahler cone with $\mathbb{R}$-symmetry}\label{subsection hk cone with S1 action}

The previous example $\mathbb{C}^{2n}$ is a metric cone, this suggests that one may consider a hyperk\"ahler cone admitting an $\mathbb{R}$-symmetry. Compared with the previous
subsection, now we have globally defined moment maps and a $\Sp(1)$-action on the cone, as we shall explain.

Let $(S,g_S)$ be a compact Riemannian manifold, then $(C(S)=\mathbb{R}_+\times S,g_{C(S)} = d\rho^2 + \rho^2g_S)$ is called the Riemannian cone of $S$, here $\rho$ is the coordinate
of $\mathbb{R}_+$ and in fact it is the distance to the vertex of the cone. It is clear that $S$ can be identified with $\{\rho=1\}\subset C(S)$ as a Riemannian submanifold of
$C(S)$. The vector field $\rho\frac{\partial}{\partial \rho}$ is called the Euler field, it generates the homothety of the cone $\lambda\cdot(\rho,s)=(\lambda\rho,s)$ for
$\lambda,\rho \in \mathbb{R}_+, s\in S$. If there is a $\mathbb{R}$-isometry of $S$, then this action can be extended to an isometry of $C(S)$ in a unique way such that the action
commutes with the homothety. In particular, let $T$ be the generator of the $\mathbb{R}$-action, then $[T,\rho\frac{\partial}{\partial\rho}]=0$.

Let $M=C(S)$ be a Riemannian cone. If $M$ is also equipped with a K\"ahler structure $(M,g,\omega,I)$ such that $g$ coincides with the conic metric, then $S$ is called a Sasakian
manifold. Let $\xi=I(\rho\frac{\partial}{\partial\rho})$ be the Reeb vector field; then it is orthogonal to the Euler field, and hence it can be viewed as a vector field on $S$. It is known that
the Reeb field is Killing with respect to the metric of $M$ and $S$, and it is also holomorphic with respect to $I$. It can be verified that $L_{\rho\frac{\partial}{\partial
\rho}}I=0$ and consequently $[\rho\frac{\partial}{\partial\rho},\xi]=0$. Regarding the K\"ahler potential, it is known that a K\"ahler potential of $M$ is $\frac{1}{2}\rho^2$, that is, $\omega = i\partial\bar\partial (\frac{1}{2}\rho^2)$.

Suppose now that there is an $\mathbb{R}$-isometry on $S$, then it can be extended to $M$ with generator $T$. It is now clear that $\rho$ is invariant by this action. The next
proposition will give a formula of the moment map.
\begin{proposition}\label{moment map of cone}
Let $K$ be an $\mathbb{R}$-invariant K\"ahler potential of the K\"ahler manifold $M$, then $x=\frac{1}{2}d^cK(T)$ defines a moment map of this $\mathbb{R}$-action.
\end{proposition}
\begin{proof}
First we note that if $\alpha$ is a 1-form and $X,Y$ are two vector fields, then
\begin{align}
d(\alpha(X))(Y)=(L_X\alpha)(Y)-(d\alpha)(X,Y).
\end{align}
Taking $\alpha=d^cK$ and $X=T$, since $I$ and $K$ are $T$-invariant, we have $L_X\alpha=0$, hence
\begin{align}
d(d^cK(T))(Y)=-(dd^cK)(T,Y)=-2\omega(T,Y)=-2(\iota_T\omega)(Y).\qed\qedhere
\end{align}
\end{proof}

From the above proposition we know that
\begin{align}
x=\frac{1}{4}d^c(\rho^2)(T) = -\frac{1}{4}d(\rho^2)(IT)
\end{align}
is a moment map with respect to the $\mathbb{R}$-action and $\omega$. It is the best choice of the moment map since it is homogeneous with respect to the homothety: we have
$L_{\rho\frac{\partial}{\partial \rho}}x=2x$.

Now assume that $M$ admits a hyperk\"ahler structure $(M,g_0,I_i,\omega_i)$, where $g_0$ coincides with the conic metric. In this case $S$ is called a
3-Sasakian manifold. Let $\xi_i=I_i\rho\frac{\partial}{\partial \rho},(i=1,2,3)$ be the Reeb fields, then we have $[\xi_i,\xi_j]=-2\xi_k$ when $(i,j,k)$ is a cyclic permutation of
$(1,2,3)$. So, it follows that $\xi_1,\xi_2,\xi_3$ generate an $\Sp(1)$-action on $M$. For the K\"ahler potential, we have $\omega_j = i\partial\bar\partial_{I_j}
(\frac{1}{2}\rho^2)$ for $j=1,2,3$. For more details, we refer to \cite{boyer2008sasakian} as a reference on Sasakian and 3-Sasakian geometry.

Suppose that there is an $\mathbb{R}$-action on $S$ whose extension to $M$ preserves its hyperk\"ahler structure. Let $T$ be its generator, then by the discussion above, we know
that for $j=1,2,3$
\begin{align}\label{xj}
x_j=\frac{1}{4}d^c_{I_j}(\rho^2)(T) = -\frac{1}{4}d(\rho^2)(I_jT)
\end{align}
are the hyperk\"ahler moment maps. As before, we will use $\mu=(x_1,x_2,x_3)$ to denote the hyperk\"ahler moment map from $M$ to $\mathbb{R}^3$. The following proposition describes
the effect of the $\Sp(1)$-action to the hyperk\"ahler moment maps.
\begin{proposition}\label{Sp(1)-equivariance, generalized}
For $i,j,k\in\{1,2,3\}$, we have $L_{\xi_i}x_j=-2\varepsilon_{ijk}x_k$.
\end{proposition}
Here $\varepsilon_{ijk}=0$ if some indices coincide, $\varepsilon_{ijk}=1$ if $(i,j,k)$ is a cyclic permutation of $(1,2,3)$ and $\varepsilon_{ijk}=-1$ if $(i,j,k)$ is a noncyclic
permutation of $(1,2,3)$.
\begin{proof}
First, we note that $I_iI_j=\varepsilon_{ijk}I_k-\delta_{ij}1$ and $g_0(X,\rho\frac{\partial}{\partial\rho}) = \frac{1}{2}d(\rho^2)(X)$ for any tangent vector $X$ of $M$. Next, we
observe that since $\rho$ is invariant by $T$, we have $d(\rho^2)(T)=0$. Then we calculate
\begin{align*}
L_{\xi_i}x_j &= dx_j(\xi_i) = -(\iota_T\omega_j)(\xi_i) = -\omega_j(T,\xi_i) = -g_0(I_jT,\xi_i) = -g_0(I_jT,I_i(\rho\frac{\partial}{\partial \rho})) \\
             &= g_0(I_iI_jT,\rho\frac{\partial}{\partial \rho}) = g_0(\varepsilon_{ijk}I_kT-\delta_{ij}T,\rho\frac{\partial}{\partial \rho}) =
             \frac{1}{2}d(\rho^2)(\varepsilon_{ijk}I_kT-\delta_{ij}T)  \\
             &= \frac{1}{2}\varepsilon_{ijk}d(\rho^2)(I_kT) = -2\varepsilon_{ijk}x_k.\qedhere
\end{align*}
\end{proof}
Passing from infinitesimal generator to the level of group action, we have
\begin{proposition} \label{2-fold covering}
Denote by $\mu = (x_1,x_2,x_3): M\rightarrow \mathbb{R}^3$ the hyperk\"ahler moment map. For $q\in \Sp(1)$, $m\in M$, we have
\begin{align}
\mu(q\cdot m) = \phi(q)\mu(m),
\end{align}
where $\phi: \Sp(1) \rightarrow \SO(3)$ is the standard 2-fold covering.
\end{proposition}

\begin{remark}
In the special case of $\mathbb{C}^{2n}$ described in Example \ref{C^2n}, this $\Sp(1)$-action is quite easy to understand. In fact, as we have identified $\mathbb{C}^{2n}$ with
$\mathbb{H}^n$, this $\Sp(1)$-action is the left multiplication on $\mathbb{H}^n$.
\end{remark}

Similarly, one may ask the effect of the $\Sp(1)$-action on $g_0$, $I_j$ and $\omega_j$.
\begin{proposition}[{\cite[Chapter 13]{boyer2008sasakian}}] \label{Sp(1)-equivariance on hk sructures}
We have $L_{\xi_i}g_0 = 0$, $L_{\xi_i}I_j = -2\varepsilon_{ijk}I_k$ and $L_{\xi_i}\omega_j = -2\varepsilon_{ijk}\omega_k$.
\end{proposition}

\subsection{The Taub-NUT deformation}\label{subsection the Taub-NUT deformation}

Let $(M,g,I_1,I_2,I_3,\omega_1,\omega_2,\omega_3)$ be a hyperk\"ahler manifold. Suppose $T$ is a vector field on $M$ which preserves
its hyperk\"ahler structure. Unlike Subsection \ref{subsection decomposition of the metric}, we do not suppose that $T$ is non-vanishing. Instead, we assume that $T$ is a complete
vector field and that the $\mathbb{R}$-action generated by $T$ is tri-Hamiltonian (that is, admitting a global hyperk\"ahler moment map). Denote by
$M^\p\subset M$ the open subset where $T\neq 0$. Clearly $M^\p$ is invariant by $T$. All the discussion in Subsection \ref{subsection decomposition of the metric} applies to $M^\p$.

According to \cite{boyer2008sasakian}, the procedure sending $M$ to the hyperk\"ahler quotient of $M\times\mathbb{H}$ by the $\mathbb{R}$-action generated by
$aT+\frac{\partial}{\partial q_0}$ is called a Taub-NUT deformation (of order 1) of $M$. Such deformations are parameterised by $a > 0$. When $\dim_\mathbb{R}M = 4$, this
deformation is also known as the standard Gibbons-Hawking deformation and is used in \cite{gibbons1997hyperkahler} to produce ALF instantons. In the following proposition, we give
a formula of the deformed structures, generalizing the formulas given in \cite{gibbons1997hyperkahler} to higher dimensions.

\begin{proposition}\label{deformation}
The hyperk\"ahler quotient of $M\times \mathbb{H}$ with respect to the $\mathbb{R}$-action generated by $Z=aT+\frac{\partial}{\partial q_0}~ (a>0)$
is diffeomorphic to $M$ as a smooth manifold. Under a certain identification, its hyperk\"ahler structure $(g_a,I_1^a,I_2^a,I_3^a,\omega_1^a,\omega_2^a,\omega_3^a)$ is $T$-invariant
on $M$. And on $M^\p$ we have
\begin{align}\label{g1}
g_a&=\bar{g} + (V+a^2)\sum_{i=1}^3(dx_i)^2 + \frac{1}{V+a^2}\eta^2,\\
I_i^a&=I_i-\frac{a^2V}{a^2+V}\theta\otimes I_iT + a^2dx_i\otimes T,\\
\omega_i^a&=\omega_i + a^2dx_j\wedge dx_k,
\end{align}
where $(i,j,k)$ is any cyclic permutation of $(1,2,3)$. In particular, the hyperk\"ahler structure is independent of the choice of the value of the moment map when forming the
hyperk\"ahler quotient.
\end{proposition}

\begin{remark}
If $T$ is the zero vector field, the proposition still holds. As long as $T$ is not identically $0$, the new hyperk\"ahler structure is different
from the original one.
\end{remark}

\begin{remark}
It it clear from the above proposition that $x_j$ is the moment map of $\omega_j^a$, for $j=1,2,3$.
\end{remark}

\begin{proof}
First, we note that the product $M\times \mathbb{H}$ is hyperk\"ahler, equipped with product hyperk\"ahler structure $(g_{M\times
\mathbb{H}}=g+g_\mathbb{H},I_{M\times\mathbb{H},i}=I_{i}+I_{\mathbb{H},i},\omega_{M\times \mathbb{H},i}=\omega_i+\kappa_i)$.
The vector field $Z=(aT,\frac{\partial}{\partial q_0})$ is complete in $M\times \mathbb{H}$ preserving the hyperk\"ahler structure, with moment map
$\mu=(ax_1-q_1,ax_2-q_2,ax_3-q_3)$. In particular, the range of the hyperk\"ahler moment map is the whole $\mathbb{R}^3$. The corresponding
$\mathbb{R}$-action is clearly free, since it acts by translations on the $\mathbb{H}$-component. Fix any $\lambda=(\lambda_1,\lambda_2,\lambda_3)\in
\mathbb{R}^3$, $\mu^{-1}(\lambda)$ is a smooth submanifold of $M\times \mathbb{H}$ and we have
\begin{align}
\mu^{-1}(\lambda)=\{(m,q)\in M\times \mathbb{H}\mid q=(q_0,ax_1(m)-\lambda_1,ax_2(m)-\lambda_2,ax_3(m)-\lambda_3)\}.
\end{align}

Define a map $\sigma: M\rightarrow \mu^{-1}(\lambda)$ by $\sigma(m)=(m,0,ax_1-\lambda_1,ax_2-\lambda_2,ax_3-\lambda_3)$. Then each $\mathbb{R}$-orbit
in $\mu^{-1}(\lambda)$ contains exactly one point of $\sigma(M)$ which is the point with $q_0=0$. Let $\pi: \mu^{-1}(\lambda) \rightarrow Q_\lambda=\mu^{-1}(\lambda)/\mathbb{R}$ be
the quotient map, then $\pi \circ\sigma$ identifies $M$ with the orbit space $Q_\lambda$.

Let $\iota:\mu^{-1}(\lambda) \rightarrow M\times\mathbb{H}$ be the inclusion map. Denote by $\omega_i^\p$ the hyperk\"ahler quotient K\"ahler form on $Q_\lambda$, then $\omega_i^\p$
is determined by $\pi^*\omega_i^\p = \iota^*(\omega_i + \kappa_i)$. By identification $\pi \circ\sigma: M \rightarrow Q_\lambda$, we get $\omega_i^a = \sigma^*\pi^*\omega_i^\p =
\sigma^*\iota^*(\omega_i + \kappa_i) = \omega_i + a^2dx_j \wedge dx_k$, where $(i,j,k)$ is any cyclic permutation of $(1,2,3)$. This proves the formula of $\omega_i^a$.

Note that the triple $(\omega_q^a, \omega_2^a, \omega_3^a)$ determines the hyperk\"ahler structure $(g_a,I_i^a,\omega_i^a)$, so to prove the rest
of the formulas, it suffices to verify that $I_1^a,I_2^2,I_3^a$ satisfy the quaternionic relations and that $\omega_i^a(X,Y) = g_a(I_i^aX,Y)$. We will omit the details.
\end{proof}

If $M=\mathbb{C}^2$ equipped with the hyperk\"ahler structure described in Example \ref{H} and the $\mathbb{S}^1$-action $e^{it}\cdot(z_1,z_2)=(e^{it}z_1,
e^{-it}z_2)$, then the above deformation of it is the Taub-NUT metric.

If $T$ is the generator of an $\mathbb{S}^1$-action on $M$, then we can replace $\mathbb{H}$ by $\mathbb{S}^1\times \mathbb{R}^3$ as described in Example \ref{H} and fix $a = 1$.
That is to say, the hyperk\"ahler quotient of $M\times \mathbb{H}$ with respect to the $\mathbb{R}$-action generated by $(T,\frac{\partial}{\partial q_0})$ is the
same as the hyperk\"ahler quotient of $M\times S^1\times \mathbb{R}^3$ with respect to the diagonal $S^1$-action in the product.

\section{Taub-NUT deformation of hyperk\"ahler cones}\label{section Basic properties of the main construction}

Let $(M,g_0,I_i,\omega_i)$ be a hyperk\"ahler cone with link $S$. Assume that $S$ admits an $\mathbb{R}$-symmetry which is automatically extended to $M$. Denote by $T$ its
generator. Let $(g_a,I_i^a,\omega_i^a)$ be the deformation described in Proposition \ref{deformation} of $M$ with $a>0$.

In this section, we will deduce some basic properties of this deformed hyperk\"ahler structure $(M,g_a,I_i^a,\omega_i^a)$.

\subsection{An estimation of the distance function}

Denote by $\rho_a$ the distance to the origin of $M$ measured by $g_a$. Recall that $\rho$ is the
distance for $g_0$. In this section, we shall compare $\rho_a$ and $\rho$.

Comparing formulas \eqref{g} and \eqref{g1}, one sees that
\begin{align}
g_a\leq g_0+ a^2 \sum_{i=1}^3(dx_i)^2.
\end{align}
Fixing any point $m\in M$, let $c(t)=tm,t\in[0,1]$ be the straight line connecting $0$ and $m$. Denote
by $l(g_a,c)$ the length of $c$ measured by $g_a$, then $\rho_a(m)\leq l(g_a,c)$.

Note that $c^\p(t)=m$ and $g_0(c^\p(t),c^\p(t))=\rho^2(m)$. Since the hyperk\"ahler moments $x_1,x_2,x_3$ are quadratic, that is,
$L_{\rho\frac{\partial}{\partial\rho}}x_i=2x_i$, we have $x(c(t))=t^2x(m)$ so $dx_i(c^\p(t))=2tx_i(m)$, hence
$\sum_{i=1}^3dx_i(c^\p(t))^2=4t^2|x(m)|^2$. Note that since the link $S$ of $M$ is compact, we have $|x(m)|\leq C\rho(m)^2$ for some constant $C>0$.

Now we can estimate $l(g_a,c)$ as follows:
\begin{align*}
l(g_a,c)&=\int_0^1\sqrt{g_a(c^\p(t),c^\p(t))}dt \\
        &\leq \int_0^1 \sqrt{\rho^2(m)+4t^2a^2|x(m)|^2}dt \\
        &\leq \sqrt{\rho^2(m)+4C^2a^2\rho^4(m)}\\
        &= \rho(m)\sqrt{1+4C^2a^2\rho^2(m)}
\end{align*}

So we conclude that
\begin{proposition}\label{estimate of rho1, generalized}
There exists $C>0$ such that $\rho_a(m)\leq C \rho^2(m)$ for any $m\in M$ with $\rho(m)\geq 1$.
\end{proposition}

It follows that by adding the vertex $0$ to $M$, the metric space $(M,d_{g_a})$ will be completed.

\subsection{An estimation of the sectional curvature}

In this section, we give an estimate of the sectional curvature of $g_a$.

\begin{proposition}[{\cite[Proposition 2.1]{zbMATH05292970}}]\label{lemma of bielawski}
Suppose that we have a hyperk\"ahler manifold $(M,g,I_1,I_2,I_3)$ with a tri-Hamiltonian action of a Lie group $G$ that is free, proper, isometric on the $c$-level set of the
hyperk\"ahler moment map $\mu=(\mu_1,\mu_2,\mu_3)$. Denote by $Q=\mu^{-1}(c)/G$ the hyperk\"ahler quotient. Define $\breve{\mathfrak{g}}$ to be the subbundle of the tangent bundle
of $M$ generated by the action of $G$ and $L=\breve{\mathfrak{g}}\oplus I_1\breve{\mathfrak{g}}\oplus I_2\breve{\mathfrak{g}}\oplus I_3\breve{\mathfrak{g}}$. For two tangent vectors
$X,Y$ at $m\in M$ contained in $L^\perp$, denote by $A(X,Y)$ the $\breve{\mathfrak{g}}$-part of $\nabla_{\tilde{X}}\tilde{Y}$, where $\nabla$ is the Levi-Civita connection and
$\tilde{X},\tilde{Y}$ are any extensions of $X,Y$ to a neighbourhood of $m$. So $A$ is the restriction to $L^\perp\times L^\perp$ of the O'Neil tensor of the submersion $M\rightarrow M/\mathbb{S}^1$.

Then the sectional curvature $K_Q$ of the hyperk\"ahler quotient $Q$ of $M$ by $G$ satisfies the pointwise estimate
\begin{align}
|K_Q(p)(\pi)-K_M(m)(\tilde{\pi})|\leq 9|A(m)|^2.
\end{align}
where $m$ is any point of $M$ projected to $p$ and $\tilde{\pi}$ is the horizontal lift of a plane $\pi\subset T_pQ$ to $T_mM$.
\end{proposition}

\begin{proposition}\label{estimate of K, generalized}
Assume that the $\mathbb{R}$-action is locally free on $S$, that is, $T\neq 0$ on $S$. Let $K_{g_a}$ be the sectional curvature of $g_a$, then there exits $C>0$ such that $|K_{g_a}|\leq
\frac{C}{\rho_a}$.
\end{proposition}

\begin{proof}
Recall that $M=C(S)$, so any vector field $X$ in $M$ can be decomposed as $X=X_S+x_r\rho\frac{\partial}{\partial \rho}$, here at each point $X_S\in TS$ and $x_r$ is a function. If
$X_S$ and $x_r$, viewed as objects in $S$ are independent of $\rho$, then the field $X$ is invariant by homothety. Such vector fields are completely determined by their value in
$S$. Moreover, if such a field is bounded on $S$ by some constant $C$, then it is bounded by $C\rho$ on $M$.

Let $X=X_S+x_r\rho\frac{\partial}{\partial \rho}$ and $Y=Y_S+y_r\rho\frac{\partial}{\partial \rho}$ be two homothety invariant field on $M$ and let $\nabla$ denote the Levi-Civita
connection on $M$ with respect to $g_0$, then we have
\begin{align}\label{nablaXY}
\nabla_XY=(\nabla^S_{X_S}Y_S+y_rX_S+x_rY_S) + (X_S(y_r) + x_ry_r-g_S(X_S,Y_S))\rho\frac{\partial}{\partial\rho}.
\end{align}
Here, $\nabla^S$ is the Levi-Civita connection of $g_S$. So, it follows that their covariant derivative is also homothety-invariant.

We want to apply Proposition \ref{lemma of bielawski}. Let $\dim_\mathbb{R}M=4n$. To begin with, we will choose a local frame of $T(M\times\mathbb{H})|_{S\times\mathbb{H}}$ in an
open subset $U$ of $S$. Let $Z=aT+\frac{\partial}{\partial q_0}$, $Z_i=I_iZ$ for $i=1,2,3$. Next, we choose $Y_1,\dots,Y_{4n-4}$ in $TM$ with $|Y_j|=1$ such that
$Y_1,\dots,Y_{4n-4},T,T_1,T_2,T_3$ form an orthogonal frame of $M$. Then we can choose $X_0$ in $T(M\times\mathbb{H})$ which is orthogonal to $Z,Z_i,Y_j$ with $|X_0|=1$. Finally, let
$X_i=I_iX_0$ where $i=1,2,3$. It is clear that $\{X_0,\dots,X_3,Y_1,\dots,Y_{4n-4}\}$ form an orthonormal frame of the orthogonal complement of $\{Z,Z_1,Z_2,Z_3\}$.

In the next step, we extend $X_i,Y_j$ to $\mathbb{R}_+\times U\times\mathbb{H}$ as follows. We extend $Y_i$ by homothety-invariance on $M$. As for $X_i$, we may decompose $X_i$
as $X_i=X^{(i)}_S + x^{(i)}_r\rho\frac{\partial}{\partial\rho} + q^{(i)}$. Here by $q^{(i)}$ we mean the $T\mathbb{H}$ part of $X_i$. It is important to note that we may choose
$X_0$ such that all the $q^{(i)}$ depend only on the coordinates of $S$, not on the coordinates of $\mathbb{H}$. We extend the $TM$-part of $X_i$ by homothety-invariance, and we
extend $q^{(i)}$ by $L_{\rho\frac{\partial}{\partial\rho}}q^{(i)} = 2q^{(i)}$. It is clear that $\{X_0,\dots,X_3,Y_1,\dots,Y_{4n-4}\}$ form an orthogonal frame of the orthogonal
complement of $\{Z,Z_1,Z_2,Z_3\}$.

By construction, we have $|Y_i| = \rho$ and there exists $C>0$ such that $|X_i|\geq C\rho^2$ (on a perhaps smaller neighbourhood $U^\p$ of $U$). We claim that $|\nabla_{X_i}X_j| \leq
C\rho^2$. Note that $X_j$ is independent of the coordinates of $\mathbb{H}$, so we may drop the component $q^{(i)}$ in $X_i$. By equation \eqref{nablaXY}, the $TM$-part of
$\nabla_{X_i}X_j$ is of order $O(\rho)$. By the construction of $q^{(j)}$, the $T\mathbb{H}$ part of $\nabla_{X_i}X_j$ is of the order $O(\rho^2)$. By similar argument, we have
$|\nabla_{X_i}Y_j|\leq C\rho$, $|\nabla_{Y_i}X_j|\leq C\rho^2$, $|\nabla_{Y_i}Y_j|\leq C\rho$.

Recall the definition of the tensor $A$ in Proposition \ref{lemma of bielawski}, we have
\begin{align*}
\frac{|A(X_i,X_j)|}{|X_i||X_j|}\leq \frac{C}{\rho^2}, \frac{|A(X_i,Y_j)|}{|X_i||Y_j|}\leq \frac{C}{\rho^2},\\
\frac{|A(Y_i,X_j)|}{|Y_i||X_j|}\leq \frac{C}{\rho}, \frac{|A(Y_i,Y_j)|}{|Y_i||Y_j|}\leq \frac{C}{\rho},
\end{align*}
It follows that $|A|\leq \frac{C}{\rho}$. So by Proposition \ref{lemma of bielawski} we have $|K_{g_a}-K_0|\leq \frac{C}{\rho^2}$, here $K_0$ is the sectional curvature of $g_0$. On the other hand, since $(M,g_0)$ is a Riemannian cone, we know that $|K_0|\leq C\frac{1}{\rho^2}$, so it follows that $|K_{g_a}|\leq\frac{C}{\rho^2}$. Finally, the result follows from
Proposition \ref{estimate of rho1, generalized}.
\end{proof}

\subsection{The deformed complex symplectic structure}\label{subsection of complex structure}

It is proved in \cite{lebrun1991complete} that the complex symplectic structure of the Taub-NUT metric is that of a standard
$\mathbb{C}^2$. In this subsection we will prove a similar result for the Taub-NUT deformation of hyperk\"ahler cones, thus generalizing the result of \cite{lebrun1991complete} (see also \cite[Chapter 7]{hein2012gravitational}).

We start from a general proposition concerning the deformation of $(1,0)$ form.

\begin{lemma}\label{deformed (1,0) form}
Under the same assumption of Proposition \ref{deformation}, fix $j=1,2,3$, let $\alpha$ be any $(1,0)$-form with respect to the complex structure
$I_j$. Then for any $a\geq 0$, the complex 1-form $\Psi_j=\alpha - ia^2\alpha(T)dx_j$ is of type $(1,0)$ with respect to $I_j^a$.
\end{lemma}

\begin{proof}
Applying the formula for $I_j^a$ in Proposition \ref{deformation}, for any tangent vector $X$ we have
\begin{align}
dx_j(I_j^aX)&=dx_j(I_jX-\frac{a^2V}{a^2+V}\theta(X)I_jT+a^2dx_j(X)T)\\
            &=-\theta(X)+\frac{a^2}{a^2+V}\theta(X)+0\\
            &=-\frac{V}{a^2+V}\theta(X),
\end{align}
and we also have
\begin{align}
\alpha(I_j^aX)&=\alpha(I_jX-\frac{a^2V}{a^2+V}\theta(X)I_jT+a^2dx_j(X)T)\\
          &=i\alpha(X)-\frac{a^2V}{a^2+V}\theta(X)i\alpha(T)+a^2dx_j(X)\alpha(T).
\end{align}
So it follows that
\begin{align}
(\alpha-ia^2\alpha(T)dx_j)(I_j^aX)=i(\alpha-ia^2\alpha(T)dx_j)X.\qed\qedhere
\end{align}
\end{proof}

We complexify the $\mathbb{R}$-action with respect to $I_1$ to get a $\mathbb{C}$-action on $M$. More precisely, for any point $m\in M$, and for any $it\in i\mathbb{R}$, and any
sufficiently small $\tau \in \mathbb{R}$, we have
\begin{align}
\frac{d}{dt}((it)\cdot m) &= T((it)\cdot m),\\
\frac{d}{d\tau}(\tau\cdot m) &= -I_1T(\tau\cdot m).\label{complexification}
\end{align}

Fix any $m_0\in M$, consider the following differential relations along the $(-I_1T)$-orbit $m_\tau={\tau}\cdot m_0$:
\begin{align}\label{ODE system1}
\frac{d}{d\tau}(\rho^2(m_\tau)) &= d(\rho^2)(-I_1T(m_\tau)) = 4x_1(m_\tau),\\
\frac{d}{d\tau}(x_1(m_\tau)) &= dx_1(-I_1T(m_\tau)) = -g_0(I_1T,-I_1T) = |T|^2_{g_0}(m_\tau).\label{ODE system2}
\end{align}

Note that $|I_1T|_{g_0} \leq C|\rho\frac{\partial}{\partial\rho}|_{g_0}$ for some $C>0$, so the vector field $-I_1T$ is complete on $M$, hence the complexified $\mathbb{C}$-action
is well defined.

\begin{lemma}\label{lemma I_1T-orbit behavior}
For any $m\in M$, we have $\inf_{\tau\in\mathbb{R}}x_1(m_\tau) \leq 0$, $\sup_{\tau\in\mathbb{R}}x_1(m_\tau) \geq 0$.
\end{lemma}
\begin{proof}
We will prove by contradiction. Suppose that there exist $m\in M$ and $\delta>0$ such that for any $\tau\in\mathbb{R}$, we have $x_1(m_\tau) > \delta$, then by \eqref{ODE system1},
$\rho^2(m_\tau)\leq\rho^2(m)$ for any $\tau\leq 0$. Note that $g_0(\rho\frac{\partial}{\partial\rho}, I_1T) = -2x_1$, it follows that $|T|(m_\tau)\geq\frac{2\delta}{\rho(m)}$ for
any $\tau\leq 0$. By \eqref{ODE system2}, we have $\lim_{\tau\rightarrow -\infty}x_1(m_\tau) = -\infty$, a contradiction. This proves the first assertion. The second can be proved
in a similar way.
\end{proof}

Define $\Phi_a: M\rightarrow M$ by the formula
\begin{align}\label{Phia}
\Phi_a(m) = (a^2x_1(m))\cdot m.
\end{align}
Now we can state the following proposition:
\begin{proposition}\label{complex sturcture}
The map $\Phi_a:(M,I_1^a)\rightarrow
(M,I_1)$ is an $\mathbb{R}$-equivariant biholomorphism.
\end{proposition}

\begin{proof}
Since $[T,I_1T]=0$, it follows that $\Phi_1$ is $\mathbb{R}$-equivariant.

Denote by $L_{\tau}$ the map $m\mapsto \tau\cdot m$, we know that $L_{\tau}$ is $I_1$-holomorphic. By the definition of $\Phi_a$, we have
\begin{align}\label{dPhia}
d\Phi_a|_m = (dL_{a^2x_1(m)})|_m - a^2I_1T|_{\Phi_a(m)}\otimes (dx_1)|_m.
\end{align}
For any $(1,0)-$covector $\alpha$ with respect to $I_1$ at $\Phi_a(m)$, the pull-back $L_{a^2x_1(m)}^*\alpha$ is a $(1,0)-$covector at $m$, hence by Lemma \ref{deformed (1,0) form},
$$\Phi_a^*\alpha = L_{a^2x_1(m)}^*\alpha - ia^2(L_{a^2x_1(m)}^*\alpha)(T)dx_1$$ is $(1,0)$ with respect to $I_1^a$. This shows that $\Phi_a$ is holomorphic.

For any $m\in M$, we want to show that there exists a unique $\tau_m\in \mathbb{R}$ such that $\Phi_a(\tau_m\cdot m) = m$, or equivalently $a^2x_1(\tau_m\cdot m) + \tau_m = 0$. Now
the function $a^2x_1(\tau\cdot m) + \tau$ is strictly increasing and tends to $\pm\infty$ as $\tau\rightarrow \pm\infty$, so $\tau_m$ exists and is unique. This shows that $\Phi_a$
is a bijection.
\end{proof}

\begin{lemma}\label{I1T preserves omega23}
The complex symplectic structure of $(M, I_1, \omega_2 + i\omega_3)$ is preserved by $I_1T$. More precisely, we have $L_{I_1T}I_1=0$, $L_{I_1T}\omega_2 = 0$ and $L_{I_1T}\omega_3 =
0$.
\end{lemma}
\begin{proof}
That $I_1T$ preserves $I_1$ follows from the fact that the complexification of holomorphic action is also holomorphic. For the complex symplectic form, we will show that
$L_{I_1T}\omega_3 = 0$, and a similar argument works for $\omega_2$.

Since $\omega_3$ is closed, we have $L_{I_1T}\omega_3 = d(\iota_{I_1T}\omega_3)$. For any vector $Y$ we have
\begin{align}
(\iota_{I_1T}\omega_3)(Y) = \omega_3(I_1T,Y) = g_0(I_3I_1T,Y) = g_0(I_2T,Y) = \omega_2(T,Y) = -dx_2(Y),
\end{align}
so $\iota_{I_1T}\omega_2=dx_2$, hence $L_{I_1T}\omega_3 = d(-dx_2) = 0.$
\end{proof}

Finally we state the main result of this subsection:
\begin{proposition}\label{complex symplectic structure}
The map $\Phi_a: (M, I_1^a ,\omega_2^a + i\omega_3^a) \rightarrow (M, I_1, \omega_2 + i\omega_3)$ is a complex symplectomorphism.
\end{proposition}
\begin{proof}
By Proposition \ref{complex sturcture}, it suffices to show that $\Phi_a^*\omega_2 = \omega_2^a$.

Take any tangent vector $X,Y\in T_mM$, write $X = X_0 + a_0T + a_1I_1T, Y = Y_0 + b_0T + b_1I_1T$, where $X_0,Y_0$ are $g_0$-orthogonal to $T,I_1T$. Since $[T,I_1T] = 0$, we have
$(dL_{a^2{x_1}})|_m(T|_m) = T|_{\Phi_a(m)}$ and $(dL_{a^2{x_1}})|_m(I_1T|_m) = I_1T|_{\Phi_a(m)}$. Recall that $dx_1(T) = 0$, $dx_1(I_1T) = -|T|^2$, we have
\begin{align}
d\Phi_a(X) &= (dL_{a^2{x_1}})|_m(X_0) + a_0T + a_1(a^2|T|^2 + 1)I_1T,\\
d\Phi_a(Y) &= (dL_{a^2{x_1}})|_m(Y_0) + b_0T + b_1(a^2|T|^2 + 1)I_1T.
\end{align}
In the above formula $|T|^2$ is valued at $m$ while $T,I_1T$ are valued at $\Phi_a(m)$.

Along the curve $c(\tau) = \tau\cdot m$ for $\tau$ between $0$ and $a^2x_1(m)$, we have the following vector fields: $(dL_{\tau})(X_0),(dL_{\tau})(Y_0), T, I_1T$, each of them is
invariant by $I_1T$. And we also have a 2-form $\omega_2$, and by Lemma \ref{I1T preserves omega23} it is also invariant by $I_1T$. It follows that
$\omega_2((dL_{\tau})(X_0),(dL_{\tau})(Y_0)) = \omega_2(X_0,Y_0)$ for any tangent vector $X_0,Y_0$ at $m$. So by comparing $\omega_2(X, Y)$ and $\omega_2(d\Phi_a(X), d\Phi_a(Y))$,
we deduce that
\begin{align}
(\Phi_a^*\omega_2)(X,Y) = \omega_2(X,Y) + a^2|T|^2\omega_2(X_0,b_1I_1T) + a^2|T|^2\omega_2(a_1I_1T,Y_0).
\end{align}
We note that $dx_1(X) = -a_1|T|^2$, $dx_1(Y) = -b_1|T|^2$, $\omega_2(X_0,I_1T) = -dx_3(X)$, $\omega_2(I_1T,Y_0) = dx_3(Y)$, and it follows that
\begin{align}
\Phi_a^*\omega_2 = \omega_2 + a^2dx_3\wedge dx_1 = \omega_2^a,
\end{align}
here the last equality follows from Proposition \ref{deformation}.
\end{proof}

\subsection{A K\"ahler potential of the deformed metric}

In this subsection we give an explicit formula of the K\"ahler potential of $g_a$ with respect to $I_1^a$.
\begin{proposition}\label{deformed kahler potential}
Under the same assumption of Proposition \ref{deformation}. Fix $i=1,2,3$, suppose that $K_i$ is an $\mathbb{R}$-invariant K\"ahler potential of
$g$ with respect to $I_i$, that is to say $\frac{1}{2}dd^c_{I_i}K_i=\omega_i$, and suppose that the moment map $x_j$ is chosen so that
\begin{align}
d^c_{I_j}K_j(T)=2x_j.
\end{align}
Then $K_i^a=K_i+\frac{a^2}{2}(x_i^2+|x|^2)$ is an
$\mathbb{R}$-invariant K\"ahler potential of $g_a$ with respect to $I_i^a$. In fact for $(i,j,k)$ a cyclic permutation of $(1,2,3)$, we have
\begin{align}
d^c_{I_i^a}K_i^a = d^c_{I_i}K_i + a^2(x_jdx_k - x_kdx_j).
\end{align}
\end{proposition}
Here we use the notations $|x|^2=x_1^2+x_2^2+x_3^2$ and $d^c_I=I\circ d=i(\bar{\partial}_{I}-\partial)$, so that $i\partial\bar{\partial}_I=\frac{1}{2}dd^c_I$.

\begin{proof}
By symmetry, we will prove the formula for $i=1$.

It is clear that $K_1^a$ is a $T$-invariant function since the moment maps $x_1,x_2,x_3$ are invariant.

By formula of $I_1^a$ in Proposition \ref{deformation} and assumption on $x_j$, we have
\begin{align}
d^c_{I_1^a}K_1 &= d^c_{I_1}K_1 - \frac{a^2V}{V + a^2}2x_1\theta,\\
d^c_{I_1^a}x_1 &= \frac{V}{V + a^2}\theta,
\end{align}
and $d^c_{I_1^a}x_2 = d^c_{I_1}x_2 = dx_3$, $d^c_{I_1^a}x_3 = d^c_{I_1}x_3 = -dx_2$. It follows that
\begin{align}
d^c_{I_1^a}K_1^a = d^c_{I_1}K_1 + a^2(x_2dx_3 - x_3dx_2).
\end{align}
So, by formula of $\omega_1^a$ in Proposition \ref{deformation}, we have $\frac{1}{2}dd^c_{I_1^a}K_1^a = \omega_1^a$.
\end{proof}

Applying the above proposition to the case of Taub-NUT deformation of a hyperk\"ahler cone where $x_1,x_2,x_3$ are defined by \eqref{xj}, we have
\begin{proposition}\label{K_1^1}
The function $K_1^a=\frac{1}{2}\rho^2+a^2x_1^2+\frac{a^2}{2}(x_2^2+x_3^2)$ is an $\mathbb{R}$-invariant K\"ahler potential of $g_a$ with respect to $I_1^a$.
\end{proposition}

\begin{remark}
The above formula for the potential can be seen as a generalization of the formula of the potential of Taub-NUT metric given in \cite{lebrun1991complete}.
\end{remark}

\begin{corollary}\label{corollary K_ALF^a}
The function $K_{\mathrm{ALF}}^a = (\Phi_a^{-1})^*K_1^a$ is an $I_1-$potential of the K\"ahler metric $(M,(\Phi_a^{-1})^*g_a)$.
\end{corollary}
In summary, by Proposition \ref{complex symplectic structure}, if one cares only about the complex symplectic structure of $M=C(S)$, then the Taub-NUT deformation leaves $(M,I_1,\omega_2 + i\omega_3)$ unchanged. Moreover, by Corollary \ref{corollary K_ALF^a} if one also wants to keep track of the K\"ahler form associated to $I_1$, then the Taub-NUT deformation can be viewed as replacing its potential $\frac{1}{2}\rho^2$ by $K_{\mathrm{ALF}}^a$. Note that a hyperk\"aher structure is completely determined by its three K\"ahler forms, so we have an equivalent description of the Taub-NUT deformation of hyperk\"ahler cones.

Then we will prove an estimate of $dK_1^a\wedge d^c_{I_1^a}K_1^a$ which turns out to be useful.
\begin{lemma}\label{lemma 4x_1^2V}
For $i = 1,2,3$, we have $4x_i^2 \leq \rho^2|T|^2$, or equivalently $4x_1^2V\leq \rho^2$ when $T\neq 0$.
\end{lemma}
\begin{proof}
One verifies that $2x_i = -g_0(\rho\frac{\partial}{\partial \rho}, I_iT)$, so the result follows from the Cauchy-Schwartz inequality.
\end{proof}

\begin{proposition}\label{proposition dK wedge d^cK}
There exists $C > 0$ such that $dK_1^a\wedge d^c_{I_1^a}K_1^a \leq C K_1^add^c_{I_1^a}K_1^a$.
\end{proposition}
\begin{proof}
It suffices to show that for any tangent vector $X$ of $M$, we have
\begin{align}\label{aim dK wedge d^cK}
(dK_1^a(X))^2 + (d^c_{I_1^a}K_1^a(X))^2 \leq C K_1^a g_a(X,X).
\end{align}
Decompose $X = \bar{X} + V(a_0T - a_1T_1 - a_2T_2 - a_3T_3)$, where $T_i = I_iT$, $\bar{X}$ is orthogonal to $T$ and $T_i$ with respect to $g_0$, then $\theta(X) = a_0$, $dx_i(X) = a_i$ for $i = 1,2,3$.
Recall that $d(\rho^2)(T) = 0$ and $d(\rho^2)(T_i) = -4x_i$, we have
\begin{align}
dK_1^a(X) = \frac{1}{2}d(\rho^2)(\bar{X}) + 2(V + a^2)x_1a_1 + (2V + a^2)(x_2a_2 + x_3a_3).
\end{align}
Similarly we have
\begin{align}
d^c_{I_1^a}K_1^a(X) = \frac{1}{2}d^c_{I_1}(\rho^2)(\bar{X}) + 2Vx_1a_0 + (2V + a^2)(x_2a_3 - x_3a_2).
\end{align}
By mean value inequality $(\frac{x+y+z}{3})^2 \leq \frac{x^2 + y^2 +z^2}{3}$, we have
\begin{align}
&(dK_1^a(X))^2 + (d^c_{I_1^a}K_1^a(X))^2 \\
&\leq 3 \left((\frac{1}{2}d(\rho^2)(\bar{X}))^2 + (\frac{1}{2}d^c_{I_1^a}(\rho^2)(\bar{X}))^2 + 4V^2x_1^2a_0^2 \right) + \\
& + 3 \left( 4(V + a^2)^2x_1^2a_1^2 + (2V + a^2)^2(x_2^2 + x_3^2)(a_2^2 + a_3^2)  \right).
\end{align}
On the other hand, recalling the formula of $g_a$ in Proposition \ref{deformation}, we get
\begin{align}
g_a(X,X) = g_0(\bar{X},\bar{X}) + (V + a^2)(a_1^2 + a_2^2 + a_3^2) + \frac{V^2}{V + a^2}a_0^2.
\end{align}
Now by Lemma \ref{lemma 4x_1^2V} the inequality \eqref{aim dK wedge d^cK} is implied by the following estimates:
\begin{align}
(\frac{1}{2}d(\rho^2)(\bar{X}))^2 + (\frac{1}{2}d^c_{I_1^a}(\rho^2)(\bar{X}))^2 \leq \rho^2g_0(\bar{X},\bar{X}), \\
V^2x_1^2a_0^2 \leq C(\rho^2 + a^2(2x_1^2 + x_2^2 + x_3^2))\frac{V^2}{V + a^2}a_0^2, \\
4(V + a^2)^2x_1^2a_1^2 \leq C(\rho^2 + a^2(2x_1^2 + x_2^2 + x_3^2)) (V + a^2)a_1^2,
\end{align}
\begin{align}
(2V + a^2)^2(x_2^2 + x_3^2)(a_2^2 + a_3^2) \leq C(\rho^2 + a^2(2x_1^2 + x_2^2 + x_3^2)) (V + a^2)(a_2^2 + a_3^2).\qed\qedhere
\end{align}
\end{proof}
\begin{remark}
Note that the proof shows that $C$ can be chosen uniformly for all $a>0$.
\end{remark}

\begin{corollary}\label{corollary dd^cK^alpha>0}
There exists $0 < \alpha_0 < 1$ such that for all $\alpha > \alpha_0$, we have $dd^c_{I_1^a}(K_1^a)^\alpha > 0$.
\end{corollary}
\begin{proof}
In general, consider smooth functions $f:M\rightarrow\mathbb{R}$ and $\psi: \mathbb{R} \rightarrow \mathbb{R}$, we have
\begin{align}\label{dd^c of composition}
dd^c(\psi\circ f) = \psi^{\p\p}(f) df\wedge d^cf + \psi^\p(f)dd^cf.
\end{align}
Now let $\psi(s) = s^\alpha$ and $f = K_1^a$, applying Proposition \ref{proposition dK wedge d^cK} we have
\begin{align}
dd^c_{I_1^a}(K_1^a)^\alpha = \alpha(K_1^a)^{\alpha-2}\left[(\alpha - 1)dK_1^a\wedge d^c_{I_1^a}K_1^a + K_1^add^c_{I_1^a}K_1^a \right] > 0
\end{align}
by choosing $\alpha_0$ sufficiently close to $1$.
\end{proof}
\begin{remark}
On the K\"ahler cone $M$, $\frac{1}{2}\rho^2$ is a potential and it is known that for any $\alpha > 0$, we have $dd^c(\frac{1}{2}\rho^{2})^{\alpha} > 0$. So we can think of Corollary \ref{corollary dd^cK^alpha>0} as a generalization of this property to the Taub-NUT deformation.
\end{remark}

In general, let $(M,g,I_1,I_2,I_3,\omega_1,\omega_2,\omega_3)$ be a hyperk\"ahler manifold, then there is a family of complex structures parameterized
by the unit sphere $\mathbb{S}^2$: $\{x_1I_1+x_2I_2+x_3I_3 \mid x_1^2+x_2^2+x_3^2=1\}$. In this and the previous subsection we only focus on the K\"ahler structure of a
particular complex structure of this family, so it is natural to consider the same problem with respect to other complex structures. It turns out
that all the K\"ahler structures are isomorphic, due to the $\Sp(1)$-equivariance of the construction, see the next subsection.

\subsection{The $\Sp(1)$-equivariance}

Recall from Subsection \ref{subsection hk cone with S1 action} that there is an $\Sp(1)$-action on $M$ and its effect on $x_1,x_2,x_3$ is given by the group morphism $\phi:\Sp(1)
\rightarrow \SO(3)$.

\begin{proposition}\label{Sp(1) acts by isometry}
The group $\Sp(1)$ acts isometrically with respect to $(M,g_a)$.
\end{proposition}

\begin{proof}
We apply the formula of $g_a$ in Proposition \ref{deformation}.
\end{proof}

Fix $x=(x_1,x_2,x_3)\in \mathbb{S}^2\subset\mathbb{R}^3$, that is to say $x_1^2+x_2^2+x_3^2=1$. Denote by $I_x=x_1I_1+x_2I_2+x_3I_3$ the complex structure associated with $x$. Choose
$A_x\in \SO(3)$ such that the first column of $A_x$ is $(x_1,x_2,x_3)^t$ and then choose $q_x\in \mathbb{H}$ such that $\phi(q_x)=A_x$. Note that the choices of $A_x$ and $q_x$ are
not unique.

\begin{proposition}\label{Inner left Sp(1) action}
The action of $q_x\in \Sp(1)$ gives an isometric biholomorphism $(M,I_1,g_0)\rightarrow (M,I_x,g_0)$. Furthermore, denote by $I_x^a=x_1I_1^a+x_2I_2^a+x_3I_3^a$ the deformed complex
structure associated to $x$, then the action by $q_x$ gives an isometric biholomorphism $(M,I_1^a,g_a)\rightarrow (M,I_x^a,g_a)$.
\end{proposition}
\begin{proof}
First, we will use $q_x$ to denote the action of $q_x$ on $M$ and $dq_x$ its differential. For the first assertion, it suffices to prove that $dq_xI_1=(x_1I_1+x_2I_2+x_3I_3)dq_x$.
But this follows from Proposition \ref{Sp(1)-equivariance on hk sructures}.

And the second assertion follows from the first, Proposition \ref{Sp(1)-equivariance on hk sructures} and the formulas for $I_1^a,I_2^a,I_3^a$ in Proposition \ref{deformation}.
\end{proof}
\begin{remark}
It follows that the K\"ahler structures $(M,I_x^a,g_a)$ are all isomorphic for $x\in \mathbb{S}^2$.
\end{remark}
\begin{remark}
Note that the choice of $q_x$ has a degree of freedom of 1. More precisely, if $q_x$ is a choice, then so is $q_xe^{i\theta}$ for $e^{i\theta} \in \mathbb{S}^1$. If one makes a choice
of the complex symplectic form of $(M,I_x^a,g_a)$ and requires $q_x$ to be a complex symplectomorphism, then this degree of freedom will be eliminated.
\end{remark}

\section{Hyperk\"ahler cone with a locally free $\mathbb{S}^1$-symmetry} \label{section hk metrics with locally free S1 action}
In this section we consider the situation of a hyperk\"ahler cone admitting a locally free $\mathbb{S}^1$-symmetry. More precisely, let $S$ be a 3-Sasakian compact connected manifold of dimension $4n - 1$ that admits a locally free $\mathbb{S}^1$-action of automorphisms. Let $M = C(S)$ be the hyperk\"ahler cone over $S$ equipped with $(g_0, I_i, \omega_i)$. The
$\mathbb{S}^1$-action extends to a locally free action on $M$, let $T$ denote its generator, then there exists $C>1$ such that $C^{-1}\rho^2 \leq |T|^2_{g_0} \leq C\rho^2$. Since an
$\mathbb{S}^1$-action is a special case of an $\mathbb{R}$-action, discussion in Section \ref{section Basic properties of the main construction} applies here.

\subsection{The twistor space of $S$} \label{subsection the twistor space of S}
In the beginning of this subsection we only assume that $S$ admits an $\mathbb{S}^1$-symmetry, which may not be locally free.

By \cite{boyer2008sasakian} any 3-Sasakian manifold is quasi-regular, so the Reeb field $\xi_1 = I_1\rho\frac{\partial}{\partial \rho}$ generates a locally free
$\mathbb{S}^1$-action. Recall that $K_1 = \frac{1}{2}\rho^2$ is a K\"ahler potential of $\omega_1$ with respect to $I_1$, which is invariant by $\xi_1$ (see Proposition
\ref{Sp(1)-equivariance, generalized}). Hence, by Proposition \ref{moment map of cone}, $\frac{1}{2}d^c_{I_1}K_1(\xi_1) = \frac{1}{2}\rho^2$ is a moment map of the
$\mathbb{S}^1_{\xi_1}$-action with respect to $\omega_1$. Here we use the notation $\mathbb{S}^1_{\xi_1}$ to indicate the $\mathbb{S}^1$-action generated by $\xi_1$.

The K\"ahler quotient $Z = S/\mathbb{S}^1_{\xi_1}$ of $M$ with respect to $\mathbb{S}^1_{\xi_1}$ at the moment $\frac{1}{2}\rho^2 = \frac{1}{2}$ is known to be the twistor space of
the 3-Sasakian manifold $S$. It is compact connected as $S$ is. It is also a complex orbifold equipped with a Fano K\"ahler-Einstein metric $(Z, I_Z, \omega_Z)$. Since
$\dim_\mathbb{R}S = 4n - 1$, we know that $\dim_\mathbb{C}Z = 2n - 1$.

Let $\pi: S\rightarrow Z$ be the projection map. For $s\in S$ we can identify $T_{\pi(s)}Z$ with $\{\xi_1\}^\perp\subset T_sS$ by $d\pi_s$, where $\{\xi_1\}^\perp$ represents the
$g_0$-orthogonal complement of $\xi_1$ in $T_sS$. Then $I_Z$ is the restriction of $I_1$ on $\{\xi_1\}^\perp$, and $\pi^*\omega_Z = i_S^*\omega_1$ where $i_S:S\rightarrow M$ is the
inclusion of $S$ in $M$.

Since $L_{\xi_1}x_1 = 0$, the moment map $x_1$ descends to a function $\bar{x}_1$ on $Z$. And since $L_{\xi_1}T = 0$, the $\mathbb{S}^1_T$-action descends to an
$\mathbb{S}^1$-action on $Z$ generated by
\begin{align}\label{Tbar}
\bar{T} = d\pi(T - \frac{g_0(T,\xi_1)}{g_0(\xi_1,\xi_1)}\xi_1) = d\pi(T - 2x_1\xi_1).
\end{align}

Since $x_1$ is the moment map of $T$ on $M$, one verifies easily that $\bar{x}_1$ is the moment map of $\bar{T}$ on $Z$, and the gradient of $\bar{x}_1$ is given by
\begin{align}
\op{grad}_Z(\bar{x}_1) = -I_Z\bar{T} = d\pi(-I_1T - 2x_1\rho\frac{\partial}{\partial \rho}) = d\pi(\op{grad}_S(x_1)).
\end{align}
Then by \cite[Lemma 5.3]{ET1997hamiltoniantorus}, $\bar{x}_1$ is a Morse-Bott function on $Z$ and its critical suborbifolds are K\"ahler with even index. Denote by $(\phi_t =
\Phi_t^{\op{grad}_Z(\bar{x}_1)})_{t\in\mathbb{R}}$ the gradient flow of $\bar{x}_1$. For each connected component $D_i$ of the critical set $\op{crit}_Z(\bar{x}_1)$, let
$\lambda_Z(D_i)$ be its index and let
\begin{align}
W^+_Z(D_i) &= \{z\in Z \mid \lim_{t\rightarrow+\infty}\phi_t(z) \in D_i\},\\
W^-_Z(D_i) &= \{z\in Z \mid \lim_{t\rightarrow-\infty}\phi_t(z) \in D_i\}.
\end{align}
be the stable and unstable set of $D_i$, then $W^+_Z(D_i)$ (respectively, $W^-_Z(D_i)$) is an orbi-bundle over $D_i$ of $\lambda_Z(D_i)$-dimensional (respectively ($4n-2-\dim_\mathbb{R}D_i -\lambda_Z(D_i)$)-dimensional) fibers.
\begin{lemma}
The function $\bar{x}_1$ has unique local minima and maxima, and the corresponding level sets are connected. Moreover, let $D_i$ be a connected component of $\op{crit}_Z(\bar{x}_1)$
such that $\bar{x}_1(D_i) > 0$, then $\op{codim}_\mathbb{C}(W^-_Z(D_i)) \geq 1$.
\end{lemma}
\begin{proof}
Let $D_1,\dots,D_N$ be connected components of $\op{crit}_Z(\bar{x}_1)$, then we can order them to have the decomposition of $Z$ as follows:
$$Z = W_Z^-(D_1)\sqcup\dots\sqcup W_Z^-(D_s)\sqcup W_Z^-(D_{s+1})\dots\sqcup W_Z^-(D_N),$$
where $D_1,\dots,D_s$ have zero index and for $i>s$, we have $\lambda_Z(D_i)\geq 2$. Let $a_i = \bar{x}_1(D_i)$, then $a_1,\dots,a_s$ are local minimas of $\bar{x}_1$. For $i>s$, since $\op{codim}_\mathbb{R}(W_Z^-(D_i))\geq 2$, the set $M\setminus\sqcup_{i>s}W_Z^-(D_i)$ is connected. Hence $s=1$ and $a_1$ is the unique minima, moreover its level set $D_1$ is connected. Conversely, one can consider the decomposition of $Z$ by $W_Z^+(D_i)$ to prove the assertion for local maximas.

The second assertion follows from the first, since if
$\bar{x}_1(D_i) > 0$ then by a rotation of $\Sp(1)$ sending $(x_1,x_2,x_3)$ to $(-x_1,-x_2,x_3)$, there exists a point on $Z$ such that $\bar{x}_1 < 0$, so $D_i$ cannot be the level
set of a local minimum.
\end{proof}

Similarly, recall $\mu_c = x_2 + ix_3: M\rightarrow\mathbb{C}$ is $I_1-$holomorphic and $L_{\rho\frac{\partial}{\partial \rho}}\mu_c = 2\mu_c$, so $\mu_c^{-1}(0)\subset M$ defines a
K\"ahler submanifold of $M$ which is also a metric cone. Assuming $n\geq 2$, it follows that $S^\p = S\cap \mu_c^{-1}(0)$ is a Sasakian manifold. Since $S$ is quasi-regular, so is
$S^\p$, then $Z^\p = S^\p/\mathbb{S}^1_{\xi_1}$ is a K\"ahler orbifold $(Z^\p,I_{Z^\p},\omega_{Z^\p})$, which can be seen as a suborbifold of $Z$. Note that $\dim_\mathbb{R}S^\p =
4n - 3$ so $\dim_\mathbb{C}Z^\p = 2n-2$.

By a similar argument, the $\mathbb{S}^1_T$-action descends to $Z^\p$ with generator $\bar{T}$ and moment map $\bar{x}_1$. So $\bar{x}_1$ is a Morse-Bott function on $Z^\p$ and its
critical suborbifolds are K\"ahler with even index. By almost the same proof, we have
\begin{lemma}\label{lemma codim in Zp}
If $S^\p$ is connected, then $\bar{x}_1$ has unique local minima and maxima, and the corresponding level sets are connected. Moreover, for a connected component $D_i$ of
$\op{crit}_{Z^\p}(\bar{x}_1)$, let
\begin{align}
W^-_{Z^\p}(D_i) = \{z^\p \in Z^\p\mid \lim_{t\rightarrow -\infty}\phi_t(z^\p)\in D_i\}
\end{align}
be its unstable set. If $\bar{x}_1(D_i) > 0$, then $\op{codim}_\mathbb{C}(W^-_{Z^\p}(D_i)) \geq 1$.
\end{lemma}
We have the following proposition which relates the critical sets in $Z$ and $Z^\p$.
\begin{lemma}
We have $\op{crit}_Z(\bar{x}_1) = \op{crit}_{Z^\p}(\bar{x}_1)$.
\end{lemma}
\begin{proof}
Since on $M$ we have $dx_1(I_2T) = dx_1(I_3T) = 0$, it follows that $\op{crit}_{Z^\p}(\bar{x}_1)$ is contained in $\op{crit}_Z(\bar{x}_1)$. Conversely, suppose $z\in\op{crit}_Z(\bar{x}_1)$,
and let $s\in S$ such that $\pi(s) = z$. We have $T(s) - 2x_1(s)\xi_1(s) = 0$ and, in particular, $T(s)$ is parallel to $\xi_1(s)$. On the other hand, by \eqref{xj} we have
$g_0(T,\xi_2) = 2x_2$, $g_0(T,\xi_3) = 2x_3$, so we must have $x_2(s) = x_3(s) = 0$. This shows that $s\in S^\p$ and hence $z\in Z^\p$.
\end{proof}

From now on we assume that $T$ is non-vanishing on $S$, that is to say, the $\mathbb{S}^1_T$-action is locally free.
\begin{proposition}\label{S prime plus is connected}
Assume that the $\mathbb{S}^1$-action is locally free. Let $S^\p_{>0} = S^\p \cap \{x_1 > 0\}$, then $S^\p_{>0}$ is connected.
\end{proposition}
\begin{proof}
First, we note that since $\mu^{-1}(0)\cap S$ has codimension 3 in $S$, we know that $S\setminus\mu^{-1}(0)$ is connected. Then we observe that $S\setminus\mu^{-1}(0) = \Sp(1)\cdot
S^\p_{>0}$, since $\Sp(1)$ rotates $x_1,x_2,x_3$ by Proposition \ref{Sp(1)-equivariance, generalized}. Now, if $S^\p_{>0}$ has two connected components $A, B$, we claim that
$\Sp(1)\cdot A$ and $\Sp(1)\cdot B$ are disjoint. Suppose not, then there is $q\in \Sp(1),a\in A, b\in B$ such that $b = q\cdot a$. Since $a,b\in S^\p_{>0}$, $q$ must be of the form
$e^{i\theta}$ for some $\theta\in \mathbb{R}$. Then $c(t) = e^{it\theta}\cdot a$ is a curve in $S^\p_{>0}$ joining $a,b$, so we get a contradiction which proves the claim. It
follows that $\Sp(1)\cdot A$ and $\Sp(1)\cdot B$ are two connected components of $S\setminus\mu^{-1}(0)$, which contradicts the fact that $S\setminus\mu^{-1}(0)$ is connected.
\end{proof}

\begin{proposition}\label{S prime is connected}
Under the same assumption as above, the Sasakian manifold $S^\p$ is connected.
\end{proposition}
\begin{proof}
Since $g_0(I_1T,\rho\frac{\partial}{\partial \rho}) = -2x_1$ and $T$ is non-vanishing on $S$, we know that $0$ is not a critical value of $x_1$ as a function on $S^\p$. Using the
inverse gradient flow of $x_1$, one can show that there exists small $\varepsilon > 0$ such that $S^\p\cap \{-\varepsilon < x_1<\varepsilon\}$ is diffeomorphic to $S^\p\cap\{x_1 = 0\} \times
(-\varepsilon, \varepsilon)$. Now $S^\p_{>0}$ is connected by the previous lemma. Let $q\in \Sp(1)$ such that $\phi(q)$ is the diagonal matrix with coefficients $-1, -1, 1$ on the
diagonal, where $\phi$ is the 2-fold covering as in Proposition \ref{2-fold covering}. Then $S^\p_{<0} = S^\p \cap \{x_1<0\} = q\cdot S^\p_{>0}$ is also connected. So $S^\p$ is
connected as the union of $S^\p_{>0}$, $S^\p \cap \{-\varepsilon < x_1<\varepsilon\}$ and $S^\p_{<0}$.
\end{proof}
It follows that the assumption of Lemma \ref{lemma codim in Zp} is satisfied if the $\mathbb{S}^1$-action is locally free.
\begin{corollary}\label{corollary S_0 is connected}
The 3-Sasakian quotient $S_0 = \mu_S^{-1}(0)/\mathbb{S}^1$ of $S$ with respect to a locally free $\mathbb{S}^1$-action is connected. Here, $\mu_S$ denotes the restriction of $\mu = (x_1,x_2,x_3)$ to $S$.
\end{corollary}
For the notion of the 3-Sasakian quotient, we refer to \cite{boyer2008sasakian}.
\begin{proof}
By Proposition \ref{S prime is connected}, we know that $Z^\p$ is connected. So by \cite[Lemma 5.1]{ET1997hamiltoniantorus}, we know that $\bar{x}_1^{-1}(0)$ is connected in $Z^\p$, so $S^\p\cap\{x_1 = 0\} = \mu_S^{-1}(0)$ is connected.
\end{proof}
\begin{remark}
I think of the above corollary as a 3-Sasakian version of the Atiyah connectedness theorem \cite[Theorem 1]{atiyah1982convexity}.
\end{remark}

\subsection{The hyperk\"ahler quotient of $M$ by $\mathbb{S}^1$}\label{subsection the hk quotients of M}
From now on we assume that the $\mathbb{S}^1_T$-action is locally free. Since $S$ is compact, there exists $C>1$ such that $C^{-1}\rho^2\leq |T|^2_{g_0} \leq C\rho^2$. In this case,
all the hyperk\"ahler quotients of $M$ with respect to $\mathbb{S}^1_T$ are orbifolds.

It is known that the hyperk\"ahler quotient $M_0 = \mu^{-1}(0)/\mathbb{S}^1$ is also a hyperk\"ahler cone. Fix $x_1>0$, denote by $M_{x_1,0,0} = \mu^{-1}(x_1,0,0)/\mathbb{S}^1$ the
hyperk\"ahler quotient of $M$ with respect to $\mathbb{S}^1$ at $(x_1,0,0)$. Note that $\dim_\mathbb{R}\mu^{-1}(x_1,0,0) = 4n - 3$ hence $\dim_{\mathbb{R}}M_{x_1,0,0} = 4n-4$, so we are mainly interested in the case where $n\geq 2$.

Let $p_{x_1,0,0}: \mu^{-1}(x_1,0,0) \rightarrow M_{x_1,0,0}$ be the quotient map. Let $m\in \mu^{-1}(x_1,0,0)$, one may identify $T_{p_{x_1,0,0}(m)}M_{x_1,0,0}$ with $\{T\}^\perp\subset T_m\mu^{-1}(x_1,0,0)$ by $dp_{x_1,0,0}$, then the complex structures on $M_{x_1,0,0}$ are restrictions of $I_j$ to $\{T\}^\perp$. Denote by $\omega^\p_{1,x_1,0,0},\omega^\p_{2,x_1,0,0},\omega^\p_{3,x_1,0,0}$ the quotient K\"ahler forms on $M_{x_1,0,0}$, then we can get $p_{x_1,0,0}^*\omega^\p_{j,x_1,0,0} = i^*_{\mu^{-1}(x_1,0,0)}\omega_j$, where $i_{\mu^{-1}(x_1,0,0)}$ is the inclusion map.

The aim of this subsection is to show that $M_{x_1,0,0}$ is a resolution of $M_0$, here by resolution we
mean that the conic singularity of $M_0$ at $0$ is resolved to orbifold singularity.

Recall Subsection \ref{subsection of complex structure} that one can always complexify the action generated by $T$. Now $T$ generates an $\mathbb{S}^1$-action, so the
$I_1-$complexification gives a $\mathbb{C}^*$-action on $M$, characterized by
\begin{align}
\frac{d}{dt}(e^{it}\cdot m) &= T(e^{it}\cdot m),\\
\frac{d}{d\tau}(e^\tau\cdot m) &= -I_1T(e^\tau\cdot m).\label{S1 complexification}
\end{align}
Here $m\in M$, $t,\tau\in\mathbb{R}$ so $e^{it}\in\mathbb{S}^1\subset \mathbb{C}^*$ and $e^\tau\in \mathbb{R}_+\subset \mathbb{C}^*$.
Under this new notation, \eqref{ODE system1} and \eqref{ODE system2} becomes
\begin{align}\label{ODE new system1}
\frac{d}{d\tau}(\rho^2(e^\tau\cdot m)) &= 4x_1(e^\tau\cdot m),\\
\frac{d}{d\tau}(x_1(e^\tau\cdot m)) &= |T|^2_{g_0}(e^\tau\cdot m).\label{ODE new system2}
\end{align}
From the above differential equations and $C^{-1}\rho^2\leq |T|^2_{g_0} \leq C\rho^2$, it follows easily that
\begin{lemma}\label{lemma ODE infinity}
Fix $m\in M$, then the function $x_1(e^\tau\cdot m)$ is strictly increasing for $\tau\in \mathbb{R}$. If $x_1(m) \geq 0$, then $\rho^2(e^\tau\cdot m)$ is strictly increasing for
$\tau > 0$, $\lim_{\tau\rightarrow +\infty}x_1(e^\tau\cdot m) = +\infty$ and $\lim_{\tau\rightarrow +\infty}\rho^2(e^\tau\cdot m) = +\infty$. Similarly, if $x_1(m)\leq 0$, then
$\rho^2(e^\tau\cdot m)$ is strictly decreasing for $\tau < 0$, $\lim_{\tau\rightarrow -\infty}x_1(e^\tau\cdot m) = -\infty$ and $\lim_{\tau\rightarrow -\infty}\rho^2(e^\tau\cdot m)
= +\infty$.
\end{lemma}
It follows that for any $m_0\in \mu^{-1}(0)$, there exists a unique $\tau(m_0) > 0$ such that $e^{\tau(m_0)}\cdot m_0 \in \mu^{-1}(x_1,0,0)$. And we define $\psi_{x_1,0,0}:
\mu^{-1}(0) \rightarrow \mu^{-1}(x_1,0,0)$ to be the map sending $m_0$ to $e^{\tau(m_0)}\cdot m_0$.

Since $L_{I_1T}T = 0$ and $L_{I_1T}I_1 = 0$, $\psi_{x_1,0,0}$ is $\mathbb{S}^1$-equivariant and hence descends to an $I_1-$holomorphic map $\psi_{x_1,0,0}: M_0\rightarrow
M_{x_1,0,0}$. By Lemma \ref{I1T preserves
omega23} we have $\psi_{x_1,0,0}^*(\omega^\p_{c,x_1,0,0}) = \omega^\p_{c,0,0,0}$, where $\omega^\p_{c,x_1,0,0} = \omega^\p_{2,x_1,0,0} + i\omega^\p_{3,x_1,0,0}$, so $\psi_{x_1,0,0}$
preserves the complex holomorphic form.

Observe that according to Lemma \ref{lemma ODE infinity}, $\psi_{x_1,0,0}: \mu^{-1}(0)\rightarrow \mu^{-1}(x_1,0,0)$ is injective. Define $ED_{x_1,0,0} = \mu^{-1}(x_1,0,0)\setminus
\op{Im}(\psi_{x_1,0,0})$ as the complement of the image, and define $\phi_{x_1,0,0}: \mu^{-1}(x_1,0,0)\rightarrow\mu^{-1}(0)\cup\{0\}$ to be an extension of $\psi_{x_1,0,0}^{-1}$
by sending $ED_{x_1,0,0}$ to $0$, where $0$ is the vertex of the cone, so $\mu^{-1}(0)\cup\{0\}$ is the metric completion of $\mu^{-1}(0)$. The definition of $ED_{x_1,0,0}$ and
Lemma \ref{lemma ODE infinity} implies the following description of $ED_{x_1,0,0}$:
\begin{proposition}\label{proposition ED flows to 0}
For any $m\in \mu^{-1}(x_1,0,0)$, $m\in ED_{x_1,0,0}$ if and only if $\lim_{\tau\rightarrow -\infty}e^\tau\cdot m = 0$.
\end{proposition}
\begin{proof}
By Lemma \ref{lemma ODE infinity}, we know that $m\in ED_{x_1,0,0}$ if and only if $x_1(e^\tau\cdot m) > 0$ for any $\tau\in\mathbb{R}$. But by Lemma \ref{lemma I_1T-orbit behavior}, we know that $\lim_{\tau\rightarrow -\infty}x_1(e^\tau\cdot m) = 0$, so by equation \eqref{ODE new system2} we know that $\lim_{\tau\rightarrow -\infty}\rho^2(e^\tau\cdot m) = 0$, so $\lim_{\tau\rightarrow -\infty}e^\tau\cdot m = 0$. The converse can be proved in a similar way.
\end{proof}

\begin{lemma}\label{phi_x1,0,0 is continuous}
The map $\phi_{x_1,0,0}: \mu^{-1}(x_1,0,0)\rightarrow\mu^{-1}(0)\cup\{0\}$ is continuous.
\end{lemma}
\begin{proof}
It suffices to show the continuity of $\phi_{x_1,0,0}$ on $ED_{x_1,0,0}$.
Fix any $m\in ED_{x_1,0,0}$. For any $\varepsilon > 0$, the condition $\rho < \varepsilon$ defines a neighbourhood of $\phi(m) = 0$ in $\mu^{-1}(0)\cup \{0\}$. By Proposition
\ref{proposition ED flows to 0}, there exists $\tau_\varepsilon < 0$ such that $\rho^2(e^{\tau_\varepsilon}\cdot m) < \varepsilon^2$. Let $x_1^\varepsilon = x_1(e^{\tau_\varepsilon}\cdot m)$, then
$e^{\tau_\varepsilon}\cdot m \in \mu^{-1}(x_1^\varepsilon,0,0)$. Note that by Lemma \ref{lemma ODE infinity}, we have $0<x_1^\varepsilon < x_1$. Now $U_\varepsilon = \{\rho<\varepsilon\}\cap
\mu^{-1}(x_1^\varepsilon,0,0)$ is an open neighbourhood of $e^{\tau_\varepsilon}\cdot m$ in $\mu^{-1}(x_1^\varepsilon,0,0)$. By Lemma \ref{lemma ODE infinity}, for any $m^\p\in U_\varepsilon$,
we have $\rho^2(\phi_{x_1^\varepsilon,0,0}(m^\p)) < \rho^2(m^\p) < \varepsilon^2$.

Define $\Phi: \mu^{-1}(x_1,0,0)\rightarrow \mu^{-1}(x_1^\varepsilon,0,0)$ by $x\mapsto e^{\tau(x)}\cdot x$, where $\tau(x)$ is determined by $e^{\tau(x)}\cdot x \in
\mu^{-1}(x_1^\varepsilon,0,0)$. By Lemma \ref{lemma ODE infinity}, the map $\Phi$ is a well-defined diffeomorphism. And we have $\phi_{x_1,0,0} = \phi_{x_1^\varepsilon,0,0}\circ\Phi$,
$\Phi(m) = e^{\tau_\varepsilon}\cdot m$. So $\Phi^{-1}(U_\varepsilon)$ is a neighbourhood of $m$ in $\mu^{-1}(x_1,0,0)$ and $\phi_{x_1,0,0}(\Phi^{-1}(U_\varepsilon)) \subset \{\rho <
\varepsilon\}\cap (\mu^{-1}(0)\cap(0))$. This shows that $\phi_{x_1,0,0}$ is continuous at $m$.
\end{proof}

By Proposition \ref{Sp(1)-equivariance, generalized}, the $\mathbb{S}^1_{\xi_1}$-action preserves $\mu^{-1}(x_1,0,0)$. Since $[\xi_1,T] = 0$, the $\mathbb{S}^1_{\xi_1}$-action descends to an $\mathbb{S}^1$-action on $M_{x_1,0,0}$, generated by
\begin{align}
\bar{\xi}_1 = dp_{x_1,0,0}(\xi_1 - \frac{2x_1}{|T|^2}T).
\end{align}
Since $L_T(\frac{1}{2}\rho^2) = 0$, the function $\frac{1}{2}\rho^2$ descends to a function on $M_{x_1,0,0}$, denoted by $\frac{1}{2}\bar\rho^2$. Since $\frac{1}{2}\rho^2$ is the moment map of $\xi_1$ with respect to $\omega_1$ on $M$, one easily verifies that $\frac{1}{2}\bar\rho^2$ is a moment map of $\bar{\xi}_1$ with respect to $\omega^\p_{1,x_1,0,0}$ on $M_{x_1,0,0}$, and that $\op{grad}_{M_{x_1,0,0}}(\frac{1}{2}\bar\rho^2) = -I_1\bar\xi_1 = dp_{x_1,0,0}(\rho\frac{\partial}{\partial \rho} + \frac{2x_1}{|T|^2}I_1T) = dp_{x_1,0,0}(\op{grad}_{\mu^{-1}(x_1,0,0)}(\frac{1}{2}\rho^2))$.

By \cite[Lemma 5.3]{ET1997hamiltoniantorus} again, $\frac{1}{2}\bar\rho^2$ is a Morse-Bott function on $M_{x_1,0,0}$, and its critical suborbifolds are K\"ahler with even index.

Let $P: M\rightarrow S$ be defined by $m\mapsto \frac{1}{\rho(m)}m$, the projection along the radius. Then for $X\in T_mM$, we have
\begin{align}\label{dP}
dP(X) = \frac{1}{\rho(m)}X - \frac{1}{\rho^2(m)}\frac{1}{2}d(\rho^2)(X)\frac{\partial}{\partial\rho}|_{P(m)}.
\end{align}
\begin{lemma}\label{lemma P is diffeomorphism}
The map $P:\mu^{-1}(x_1,0,0)\rightarrow S^\p_{>0}$ is a diffeomorphism, moreover we have
\begin{align}
(dP)|_m(\op{grad}_{\mu^{-1}(x_1,0,0)}(\frac{1}{2}\rho^2)) = -\frac{2x_1}{|T|^2(m)}\op{grad}_{S^\p}(x_1).
\end{align}
\end{lemma}
\begin{proof}
Since $L_{\rho\frac{\partial}{\partial \rho}}x_1 = 2x_1$, we know that $P$ is bijective.

Since $g_0(\rho\frac{\partial}{\partial \rho}, I_1T) = -2x_1\neq 0$ along $\mu^{-1}(x_1,0,0)$, we know that $\rho\frac{\partial}{\partial \rho}$ is transversal to $\mu^{-1}(x_1,0,0)$, hence $P$ is a local diffeomorhpism. In conclusion, the map $P$ is a diffeomorphism.

The second assertion follows from \eqref{dP}.
\end{proof}
\begin{remark}
By Proposition \ref{S prime plus is connected}, we know that $S^\p_{>0}$ is connected, so we deduce from the previous proposition that $\mu^{-1}(x_1,0,0)$ is connected. I think of it as a hyperk\"ahler version of the Atiyah connectedness theorem \cite[Theorem 1]{atiyah1982convexity}.
\end{remark}

Note that the coefficient $-\frac{2x_1}{|T|^2}$ is strictly negative along $\mu^{-1}(x_1,0,0)$, so we have a kind of ``duality'' between $\mu^{-1}(x_1,0,0)$ and $S^\p_{>0}$, in the following sense:

Firstly, $P$ induces a bijection $P: \op{crit}_{\mu^{-1}(x_1,0,0)}(\frac{1}{2}\rho^2) \rightarrow \op{crit}_{S^\p_{>0}}(x_1)$.

Secondly, since $T$ is parallel to $\xi_1$ along the critical sets, $P$ descends to a bijection $\bar{P}: \op{crit}_{M_{x_1,0,0}}(\frac{1}{2}\bar\rho^2)\rightarrow \op{crit}_{Z^\p_{>0}}(\bar{x}_1)$, where $Z^\p_{>0} = Z^\p\cap \{\bar{x}_1 >0 \}$. Moreover, $\bar{P}$ is in fact holomorphic.

Thirdly, the above map $\bar{P}$ induces a bijection between the critical suborbifold $C_i$ of $\frac{1}{2}\bar\rho^2$ and the critical suborbifold $D_i$ of $\bar{x}_1$ with $\bar{x}_1>0$, determined by $D_i = \bar{P}(C_i)$. Hence, it induces a bijection between critical values of $\frac{1}{2}\bar\rho^2$ and strictly positive critical values of $\bar{x}_1$, under which local minimum corresponds to local maximum. As a consequence, there are only finitely many critical values and critical suborbifolds of $\frac{1}{2}\bar\rho^2$.

Finally, suppose $\bar{P}(C_i) = D_i$ for critical suborbifolds $C_i$ and $D_i$, we have
\begin{align}
\lambda_{M_{x_1,0,0}}(C_i) = 4n - 4 - \dim_\mathbb{R}D_i - \lambda_{Z^\p}(D_i),
\end{align}
where $\lambda$ stands for the index.

Let $\psi_t$ be the gradient flow of $\frac{1}{2}\bar\rho^2$ on $M_{x_1,0,0}$, then for critical suborbifold $C_i$ of $\frac{1}{2}\bar\rho^2$, define
\begin{align}
W^+(C_i) = \{m\in M_{x_1,0,0}\mid \lim_{t\rightarrow +\infty}\psi_t(m)\in C_i\}
\end{align}
to be the stable set of $C_i$. Then by Lemma \ref{lemma codim in Zp}, we get $\op{codim}_\mathbb{C}W^+(C_i)\geq 1$.
\begin{proposition}
The set $ED_{x_1,0,0}$ is the union of $p_{x_1,0,0}^{-1}(W^+(C_i))$ where $C_i$ ranges over all critical suborbifolds of $\frac{1}{2}\bar\rho^2$. Consequently, $ED_{x_1,0,0}$ is compact.
\end{proposition}
\begin{proof}
Taking $m\in p_{x_1,0,0}^{-1}(W^+(C_i))$, then $m$ flows to $p_{x_1,0,0}^{-1}(C_i)$ along the flow of $\op{grad}_{\mu^{-1}(x_1,0,0)}(\frac{1}{2}\rho^2)$ as $t\rightarrow +\infty$. After projection by $P$, we know that $P(m)$ flows to $\pi^{-1}(D_i)$ along $\op{grad}_{S^\p}(x_1)$ as $t\rightarrow -\infty$, where $\bar{P}(C_i) = D_i$. Since the projection of the $(-I_1T)$-orbit of $m$ as $t\rightarrow -\infty$ to $S$ is exactly the $\op{grad}_{S^\p}(x_1)$-orbit of $P(m)$ as $t\rightarrow -\infty$, by Proposition \ref{proposition ED flows to 0} and Lemma \ref{lemma I_1T-orbit behavior}, we know that $m\in ED_{x_1,0,0}$.

So $\bigcup_{C_i}p_{x_1,0,0}^{-1}(W^+(C_i)) \subset ED_{x_1,0,0}$, and the converse can be proved in a similar way.

Now being a finite union of closed sets, we know that $ED_{x_1,0,0}$ is closed. Also, we know that $\frac{1}{2}\rho^2$ only takes finitely many values on $ED_{x_1,0,0}$, so it is also bounded. It follows that $ED_{x_1,0,0}$ is compact.
\end{proof}
It follows from the above proposition and $\op{codim}_\mathbb{C}W^+(C_i)\geq 1$ that we have $\op{codim}_\mathbb{C}p_{x_1,0,0}(ED_{x_1,0,0}) \geq 1$.

According to \cite[Theorem 1.8]{Conlon-Hein1}, the K\"ahler cone $\mu^{-1}(0)/\mathbb{S}^1$ is isomorphic to the smooth part of a normal algebraic variety $V\subset \mathbb{C}^N$ with one singular point. This gives $M_0\cup\{0\}$ a complex analytic structure. Using the Riemann extension theorem and the fact that $\op{codim}_\mathbb{C}p_{x_1,0,0}(ED_{x_1,0,0}) \geq 1$, we have
\begin{proposition}\label{proposition M_x_1,0,0 resolves M_0,0,0}
The map $\phi_{x_1,0,0}: \mu^{-1}(x_1,0,0) \rightarrow \mu^{-1}(0)\cup \{0\}$ descends to an analytic map $\phi_{x_1,0,0}: M_{x_1,0,0} \rightarrow M_0\cup\{0\}$. It is a resolution in the sense that $\phi_{x_1,0,0}$ is an isomorphism between $M_{x_1,0,0}\setminus \phi_{x_1,0,0}^{-1}(0)$ and $M_0$.
\end{proposition}

\subsection{The hyperk\"ahler metric of $M_{x_1,0,0}$}\label{subsection the hk metric of M_x_1,0,0}
Fix $x_1>0$, denote by $g_{x_1,0,0}$ the pull back by $\psi_{x_1,0,0}$ of the quotient hyperk\"ahler metric $g^\p_{x_1,0,0}$ on $M_{x_1,0,0}$ and denote by
$\omega_{1,x_1,0,0}$ the pull back by $\psi_{x_1,0,0}$ of the quotient hyperk\"ahler form $\omega^\p_{1,x_1,0,0}$. The aim of this subsection is to describe $g_{x_1,0,0}$ and $\omega_{1,x_1,0,0}$ in terms of objects on $M_0$.

For $m\in \mu^{-1}_c(0)$, choose $\tau(m,x_1)\in \mathbb{R}$ such that $e^{\tau(m,x_1)}\cdot m \in \mu^{-1}(x_1,0,0)$. Here, the action is the complexification of the $\mathbb{S}^1_T$-action with respect to $I_1$. Recall Lemma \ref{lemma ODE infinity} that $\tau(m,x_1)$ is well defined.

For $m\in \mu^{-1}(0)$, define
\begin{align}\label{f_x_1,0,0}
f_{x_1,0,0}(m) = \frac{1}{2}\rho^2(e^{\tau(m,x_1)}\cdot m) - 2x_1\tau(m,x_1).
\end{align}
It is clear that $f_{x_1,0,0}$ is $\mathbb{S}^1$-invariant, so we may think of it as a function on $M_0$.
\begin{proposition}\label{proposition omega_1_x_1,0,0}
We have
\begin{align}
\omega_{1,x_1,0,0} = i\partial\bar\partial_{I_1} f_{x_1,0,0} + x_1\Omega_0,
\end{align}
where $\Omega_0$ is the curvature 2-form of the $S^1$-bundle $\mu^{-1}(0)\rightarrow M_0$.
\end{proposition}
\begin{proof}
Let $q:\mu_c^{-1}(0) \rightarrow \mu^{-1}(x_1,0,0)$ be the map $m\mapsto e^{\tau(m,x_1)}\cdot m$ and let $\hat{\omega}_1 = q^*i_{\mu^{-1}(x_1,0,0)}^*\omega_1$, $\hat{K}(m) = \frac{1}{4}\rho^2(e^{\tau(m,x_1)}\cdot m) - x_1\tau_m$, then by \cite[Theorem 7]{biquard2021hyperkahler} we have
$\hat{\omega}_1 = dd^c_{I_1}\hat{K}$. It is clear that $\hat{K}|_{\mu^{-1}(0)} = \frac{1}{2}f_{x_1}$.

For any $m_0 \in \mu^{-1}(0)$, let $Y + aT \in T_{m_0}\mu^{-1}(0)$ be a tangent vector, where $g_0(Y,T) = 0$ and hence $a = \eta(Y + aT)$, we have
\begin{align}
(d^c_{I_1}\hat{K})(Y + aT) = -d\hat{K}(I_1Y + aI_1T) = -d\hat{K}(I_1Y) - a d\hat{K}(I_1T).
\end{align}
For $d\hat{K}(I_1T)$, since $q(m) = e^{\tau(m,x_1)}\cdot m$ is invariant by $I_1T$, and $L_{(-I_1T)}\tau(m,x_1) = -1$, we have $d\hat{K}(I_1T) = -x_1$. Here we treat $x_1$ as a constant when differential is taken. So
\begin{align}\label{first line of proof of potential of Eguchi-Hanson, generalized}
(d^c_{I_1}\hat{K})(Y + aT) = -d\hat{K}(I_1Y) + x_1\eta(Y + aT).
\end{align}
On the other hand, we have $dp_0(Y + aT) = Y$, hence
\begin{align}\label{second line of proof of potential of Eguchi-Hanson, generalized}
p_0^*(d^c_{I_1}(\frac{1}{2}f_{x_1}))(Y + aT) = d^c_{I_1}(\frac{1}{2}f_{x_1})(Y) = d^c_{I_1}(\hat{K}|_{\mu^{-1}(0)})(Y) = -d\hat{K}(I_1Y),
\end{align}
here we note that since $g_0(Y,T) = 0$, we have $I_1Y \in T_{m_0}\mu^{-1}(0)$.

Combining \eqref{first line of proof of potential of Eguchi-Hanson, generalized} and \eqref{second line of proof of potential of Eguchi-Hanson, generalized}, let $i_{\mu^{-1}(0)}:
\mu^{-1}(0) \rightarrow \mu_c^{-1}(0)$ be the inclusion, we have
\begin{align}
i_{\mu^{-1}(0)}^*d_{I_1}^c\hat{K} = p_0^*d^c_{I_1}(\frac{1}{2}f_{x_1}) + x_1 i_{\mu^{-1}(0)}^*\eta.
\end{align}
Taking exterior derivatives, and noting that on $\mu^{-1}(0)$ we have $dx_1 = 0$, one gets
\begin{align}
i^*_{\mu^{-1}(0)}\hat{\omega}_1 = p_0^*i\partial\bar\partial_{I_1}f_{x_1} + x_1 i^*_{\mu^{-1}(0)}d\eta = p_0^*(i\partial\bar\partial_{I_1}f_{x_1} + x_1\Omega_0).
\end{align}
By the definition of K\"ahler quotient, we have
\begin{align}
i_{\mu^{-1}(x_1,0,0)}^*\omega_1 = p_{x_1,0,0}^*\omega^\p_{1,x_1,0,0}.
\end{align}
So we have
\begin{align}
i^*_{\mu^{-1}(0)}\hat{\omega}_1 &= i^*_{\mu^{-1}(0)}q^*i_{\mu^{-1}(x_1,0,0)}^*\omega_1\\
                                &= i^*_{\mu^{-1}(0)}q^*p_{x_1,0,0}^*\omega^\p_{1,x_1,0,0}\\
                                &= p_0^*\psi_{x_1,0,0}^*\omega^\p_{1,x_1,0,0}\\
                                &= p_0^*\omega_{1,x_1,0,0}.
\end{align}
Here in the third equality, we used the fact that $p_{x_1,0,0}\circ q\circ i_{\mu^{-1}(0)} = \psi_{x_1,0,0}\circ p_0$.

It follows that $p_0^*\omega_{1,x_1,0,0} = p_0^*(i\partial\bar\partial_{I_1}f_{x_1} + x_1\Omega_0)$. Since $dp_0$ is surjective, $p_0^*$ is injective, the result follows.
\end{proof}

\begin{remark}
One has $\frac{\partial}{\partial x_1}|_{x_1 = 0}f_{x_1,0,0} = 0$, so we have $\frac{\partial}{\partial x_1}|_{x_1 = 0}\omega_{1,x_1,0,0} =
\Omega_0$. It coincides with the result of \cite{zbMATH03793065}.
\end{remark}

\begin{remark}
Recall $\dim_\mathbb{R}M = 4n$. In the case $n=1$, the hyperk\"ahler quotients $M_{x_1,0,0}$ are zero-dimensional, so there is no need to consider $g_{x_1,0,0}$. In fact, this
case is already treated by the Gibbons-Hawking anasatz.

In the case $n=2$, we know that $\dim_\mathbb{R}\mu^{-1}(0) = 5$ and the orthogonal complement of $T$ in $T\mu^{-1}(0)$ is spanned by $\rho\frac{\partial}{\partial\rho},\xi_1,\xi_2,\xi_3$, which is an integrable distribution by Subsection \ref{subsection hk cone with S1 action}. So in this case $\Omega_0 = 0$ and we get an $I_1$-potential $f_{x_1,0,0}$ of $g_{x_1,0,0}$. In Example \ref{C^2n}, when $n=2$, we get a K\"ahler potential of the Eguchi-Hanson metric, which coincides with the formula given by \cite{calabi1979metriques}.

If $n\geq 3$, then $\Omega_0$ is generally non-vanishing.
\end{remark}

More generally, using the $\Sp(1)$-equivariance, for any $x\neq 0$ we may construct the map $\psi_x : \mu^{-1}(0) \rightarrow \mu^{-1}(x)$ by complexifying the $\mathbb{S}^1$-action
with respect to $I_x=\frac{1}{|x|}(x_1I_1+x_2I_2+x_3I_3)$. It sends $m\in \mu^{-1}(0)$ to $e^{\tau(m,x)}\cdot m$. Its inverse extends to a resolution of the conic singularity by sending $ED_x=\mu^{-1}(x)\setminus \op{Im}(\psi_x)$ to $0$.

Consider the following function defined on $M_0$:
\begin{align}\label{f_x}
f_x(m) = \frac{1}{2}\rho^2(e^{\tau(m,x)}\cdot m) - 2x_1\tau(m,x).
\end{align}
Let $h_x(-,-)$ be the 2-tensor $i\partial\bar\partial_{I_x}f_{x}(-,I_x-)$ and let $H_x$ be the 2-tensor $\Omega_0(-,I_x-)$.
\begin{proposition}
Let $g_x$ be the pull back by $\psi_x$ of the quotient hyperk\"ahler metric on $M_x = \mu^{-1}(x)/\mathbb{S}^1$, then we have
\begin{align}\label{abstract formula of g_x}
g_x = h_x + |x|H_x.
\end{align}
\end{proposition}
\begin{proof}
One simply applies a conjugation by $q_{\frac{x}{|x|}}\in \Sp(1)$, where $q_{\frac{x}{|x|}}$ is chosen as in Proposition \ref{Inner left Sp(1) action}. And it remains to show that
$q_{\frac{x}{|x|}}^*\Omega_0 = \Omega_0$, i.e. $\Omega_0$ is invariant by the $\Sp(1)$-action. Here we note that the Reeb fields $\xi_i$ on $M$ descend to the Reeb fields on
$M_0$ and we will use the same notation $\xi_i$ for it, and consequently the $\Sp(1)$-action on $M$ descends to the $\Sp(1)$-action on $M_0$.

To show $L_{\xi_i}\Omega_0 = 0$, since $d\Omega_0 = 0$, it suffices to show $\iota_{\xi_i}\Omega_0 = 0$. For any local vector fields $X$ on $M_0$, let
$\tilde{X}$ be a horizontal lift of $X$ on $\mu^{-1}(0)$, then $g_0(\tilde{X},T) = 0$. We have $\Omega_0(\xi_i,X) = \eta([\xi_i,\tilde{X}]) = Vg_0([\xi_i,\tilde{X}],T)$.

We take a Lie derivative as follows
\begin{align}
0 = L_{\xi_i}(g_0(\tilde{X}, T)) = (L_{\xi_i}g_0)(\tilde{X}, T) + g_0([\xi_i,\tilde{X}],T) + g_0(\tilde{X},[\xi_i,T]).
\end{align}
Note that $L_{\xi_i}g_0 = 0$ and $[\xi_i,T] = 0$, we conclude that $g_0([\xi_i,\tilde{X}],T) = 0$ hence $\iota_{\xi_i}\Omega_0 = 0$.
\end{proof}

\begin{remark}
A similar argument shows that $\iota_{\rho\frac{\partial}{\partial\rho}}\Omega_0 = 0$ and $L_{\rho\frac{\partial}{\partial\rho}}\Omega_0 = 0$ (this can also be deduced from the fact
that $L_{\rho\frac{\partial}{\partial\rho}}\eta = 0$). So the subspace generated by $\rho\frac{\partial}{\partial\rho},\xi_1,\xi_2,\xi_3$ is contained in the kernel of $\Omega_0$,
hence the kernel of $H_x$ for any $0\neq x \in \mathbb{R}^3$.
\end{remark}

\begin{remark}
In the formula \eqref{f_x}, if we let $|x|\rightarrow 0$, then the potential becomes $\frac{1}{2}\rho^2$, which gives the conic metric $g_{0,0,0}$ on
$M_0$.
\end{remark}

\begin{remark}\label{remark size of ED}
Note that for $\lambda > 0$, we have $ED_{\lambda^2x} = \lambda ED_x$ by homothety, so the size of $ED_{x}$ is proportional to $\sqrt{|x|}$.
\end{remark}

\section{Twist coordinates}\label{section Twist coordinates}

As in Section \ref{section hk metrics with locally free S1 action}, we assume that the $\mathbb{S}^1$-action is locally free on the compact connected 3-Sasakian manifold $S$.

We think of the underlying space of $M$ as a double fibration. Firstly, it is the total space of an $\mathbb{S}^1$-fibration over the
quotient $M/\mathbb{S}^1$. Recall that the three moment maps $x_1,x_2,x_3$ defined by \eqref{xj} are $\mathbb{S}^1$-invariant, so the hyperk\"ahler moment map
$\mu=(x_1,x_2,x_3):M\rightarrow \mathbb{R}^3$ descends to a map $M/\mathbb{S}^1\rightarrow \mathbb{R}^3$. And in this way
we think of the quotient $M/\mathbb{S}^1$ as a fibration over $\mathbb{R}^3$ whose fiber over $q\in \mathbb{R}^3$ is
$\mu^{-1}(q)/\mathbb{S}^1$, which is the hyperk\"ahler quotient $M_q$ of $M$ with respect the $\mathbb{S}^1$-action at moment $q$. As we shall see, these two fibrations give us an
explicit system of coordinates.

More precisely, we introduce the following coordinates. Define $\Psi: \mu^{-1}(0)\times \mathbb{R}^3 \rightarrow M$ by
\begin{align}
\Psi(m,x) = \psi_x(m) = q_x\psi_{|x|,0,0}(q_x^{-1}m) = q_xe^{\tau(m,|x|)}\cdot(q_x^{-1}m).
\end{align}
Here, as before, $q_x\in \Sp(1)$ satisfies $q_x\cdot\mu^{-1}(1,0,0) = \mu^{-1}(\frac{x}{|x|})$.

At this point, we note that the image of $\Psi$ is the complement of the union $\bigcup_{x\in\mathbb{R}^3\setminus\{0\}}ED_x$, so the image of $\Psi$ is dense in $M$. Consequently, most of the
results obtained by these coordinates can be extended to $M$ by continuity.

\begin{remark}
By the definition of $\Psi$ and $\psi_x$, we see that the complex structure we use to complexify the $\mathbb{S}^1$-action depends on $\frac{x}{|x|}\in \mathbb{S}^2$. This is a similar idea for the construction of the twistor space of a hyperk\"ahler manifold. That is why we call this system of coordinates as twistor coordinates. In fact, it can be shown that $\mu^{-1}(\mathbb{S}^2)/\mathbb{S}^1$, when equipped with a good complex structure depending on $x$, is isomorphic to the twistor space of the hyperk\"ahler quotient $M_{1,0,0}$.
\end{remark}

\subsection{Calculation of the metric under twist coordinates}\label{subsection twist coordinates}
In this subsection, we will calculate the metrics $g_0$ and $g_a$ with respect to these coordinates. To simplify the notation, we will do the calculation at $x=(x_1,0,0)$ with $x_1>0$, so we may take $q_x=1$. The general case can be deduced by the $\Sp(1)$-equivariance.

We start with $\frac{\partial\Psi}{\partial x_1}$,
\begin{align}
\frac{\partial\Psi}{\partial x_1} = \frac{\partial\tau(m,x)}{\partial x_1}(-I_1T)=-VI_1T.
\end{align}
In the above expression, if it is not indicated explicitly, then the object takes value at $\Psi(m,x)$, and we will adopt this convention in this subsection.

It follows that $g_0(T,\frac{\partial\Psi}{\partial x_1})=0$, $g_0(\frac{\partial\Psi}{\partial x_1},\frac{\partial\Psi}{\partial x_1}) = V$.

For $\frac{\partial\Psi}{\partial x_2}$, we note that at $(x_1,0,0)$, we have $\frac{\partial |x|}{\partial x_2} = 0$, so $\frac{\partial \tau}{\partial x_2} = 0$. Recall
Proposition \ref{Sp(1)-equivariance, generalized}, the vector field associated with $\frac{\partial q_x}{\partial x_2}$ is $\frac{1}{2x_1}\xi_3$. So we have
\begin{align}
\frac{\partial \Psi}{\partial x_2} = \frac{1}{2x_1}\xi_3 - \frac{1}{2x_1}d\psi_{x_1,0,0}(\xi_3(m)).
\end{align}
Note that $g_0(T,\xi_j) = 2x_j$, so at $(x_1,0,0)$, we
have $g_0(T,\xi_3)=0$. On the other hand, we have $g_0(T,d\psi_{x_1,0,0}(\xi_3(m))) = g_0(T,I_1d\psi_{x_1,0,0}(\xi_2(m))) = -\omega_1(T,d\psi_{x_1,0,0}(\xi_2(m))) =
dx_1(d\psi_{x_1,0,0}(\xi_2(m))) = 0$. Hence $g_0(T,\frac{\partial\Psi}{\partial x_2}) = 0$.

And we also have
\begin{align}
g_0(\frac{\partial\Psi}{\partial x_1},\frac{\partial\Psi}{\partial x_2}) = -Vg_0(I_1T,\frac{\partial\Psi}{\partial x_2}) = -V\omega_1(T,\frac{\partial\Psi}{\partial x_2}) =
Vdx_1(\frac{\partial\Psi}{\partial x_2}) = 0.
\end{align}

Let $Y$ be any tangent vector of $\mu^{-1}(0)$ at $m$ which is orthogonal to $T$. Then we have
\begin{align}
g_0(\frac{\partial\Psi}{\partial x_2},d\Psi(Y)) &= g_0(\frac{\partial\Psi}{\partial x_2},d\psi_{x_1,0,0}(Y)) \\
                                                &= \frac{1}{2x_1}g_0(\xi_3 - d\psi_{x_1,0,0}(\xi_3(m)),d\psi_{x_1,0,0}(Y))\\
                                                &= \frac{1}{2x_1}g_0(\xi_3,d\psi_{x_1,0,0}(Y)) - \frac{1}{2x_1}g_{x_1,0,0}(\xi_3(m),Y).
\end{align}
For the first term $g_0(\xi_3,d\psi_{x_1,0,0}(Y))$, we have
\begin{align}
g_0(\xi_3,d\psi_{x_1,0,0}(Y)) = \omega_3(\rho\frac{\partial}{\partial\rho},d\psi_{x_1,0,0}(Y)).
\end{align}
We claim that this term is invariant by $I_1T$, indeed by Lemma \ref{I1T preserves omega23} we have $L_{I_1T}\omega_3 =0$. We also have $[\rho\frac{\partial}{\partial\rho},I_1T] =
0$ and by the definition of $\psi_{x_1,0,0}$ we have $L_{I_1T}(d\psi_{x_1,0,0}(Y))=0$. So $g_0(\xi_3,d\psi_{x_1,0,0}(Y)) = g_{0,0,0}(\xi_3(m),Y)$. In conclusion we have
\begin{align}
g_0(\frac{\partial\Psi}{\partial x_2},d\Psi(Y))=\frac{1}{2x_1}(g_{0,0,0}-g_{x_1,0,0})(\xi_3(m_0),Y).
\end{align}

Then we calculate that
\begin{align}
g_0(\frac{\partial\Psi}{\partial x_2},\frac{\partial\Psi}{\partial x_2}) &= \frac{1}{4x_1^2}g_0(\xi_3 - d\psi_{x_1,0,0}(\xi_3(m)),\xi_3 - d\psi_{x_1,0,0}(\xi_3(m)))\\
&=\frac{1}{4x_1^2}[\rho^2 -2\rho^2(m) + g_{x_1,0,0}(\xi_3(m),\xi_3(m))].
\end{align}

Similarly we have $\frac{\partial\Psi}{\partial x_3} = -\frac{1}{2x_1}(\xi_2-d\psi_{x_1,0,0}(\xi_2(m)))$, hence $g_0(T,\frac{\partial\Psi}{\partial
x_3})=0$, $g_0(\frac{\partial\Psi}{\partial x_1},\frac{\partial\Psi}{\partial x_3}) = 0$ and $g_0(\frac{\partial\Psi}{\partial x_3},\frac{\partial\Psi}{\partial x_3}) =
g_0(\frac{\partial\Psi}{\partial x_2},\frac{\partial\Psi}{\partial x_2})$. And a similar calculation shows that $g_0(\frac{\partial\Psi}{\partial x_2},\frac{\partial\Psi}{\partial
x_3}) = 0$.

In conclusion, let $\beta_3 = \frac{1}{2x_1}(g_{0,0,0} - g_{x_1,0,0})(\xi_3(m),-)$ and let $\beta_2 = \frac{1}{2x_1}(g_{0,0,0} - g_{x_1,0,0})(\xi_2(m),-)$, then we have
\begin{align}\label{g_0}
\Psi^*g_0 =& V(dx_1)^2 + \frac{1}{4x_1^2}[\rho^2 -2\rho^2(m) + g_{x_1,0,0}(\xi_3(m),\xi_3(m))]((dx_2)^2+(dx_3)^2)+ \\
     &+ \beta_3dx_2 - \beta_2dx_3 + g_{x_1,0,0} + \frac{1}{V}\eta^2.
\end{align}
hence by Proposition \ref{deformation} we have
\begin{align}\label{g_a}
\Psi^*g_a =&  V(dx_1)^2 + \frac{1}{4x_1^2}[\rho^2 -2\rho^2(m) + g_{x_1,0,0}(\xi_3(m),\xi_3(m))]((dx_2)^2+(dx_3)^2)+ \\
     &+ a^2\sum_{j = 1}^3(dx_j)^2 + \beta_3dx_2 - \beta_2dx_3 + g_{x_1,0,0} + \frac{1}{V+a^2}\eta^2.
\end{align}

\subsection{The asymptotic behavior}\label{subsection the asymptotic behavior}
In this subsection, we give several estimates of the objects introduced in the previous subsection. First, we start from $\tau(m,x_1)$, where $x_1>0$.
\begin{lemma}\label{lemma estimation of tau}
There exists a uniform $C>0$ such that $0 < \tau(m,x_1) \leq C\frac{x_1}{\rho^2(m)}$.
\end{lemma}
\begin{proof}
Recall that $\tau(m,x_1)$ is determined by $e^{\tau(m,x_1)}\cdot m \in \mu^{-1}(x_1,0,0)$ for $m\in \mu^{-1}(0)$, where the action is obtained by $I_1-$complexification of $\mathbb{S}^1_T$. By equation \ref{ODE new system2}, \ref{S1 complexification} and Lemma \ref{lemma ODE infinity}, we have
\begin{align}
x_1 = x_1(e^{\tau(m,x_1)}\cdot m)
    &= \int_0^{\tau(m,x_1)}\frac{d}{dt}(x_1(e^t\cdot m))dt
    = \int_0^{\tau(m,x_1)}|T|^2(e^t\cdot m))dt \\
    &\geq C^{-1}\int_0^{\tau(m,x_1)}\rho^2(e^t\cdot m)dt
    \geq C^{-1}\int_0^{\tau(m,x_1)}\rho^2(m)dt \\
    &= C^{-1}\rho^2(m)\tau(m,x_1).
\end{align}
\end{proof}
Next, we give an estimate of $\rho^2(e^{\tau(m,x_1)}\cdot m)$.
\begin{lemma}\label{estimation of rho^2}
There exists a uniform $C>0$ such that $0 < \rho^2(e^{\tau(m,x_1)}\cdot m) - \rho^2(m) \leq C\frac{x_1^2}{\rho^2(m)}$.
\end{lemma}
\begin{proof}
By equation \eqref{ODE new system1}, we have
\begin{align}
\rho^2(e^{\tau(m,x_1)}\cdot m) - \rho^2(m) &= \int_0^{\tau(m,x_1)}\frac{d}{dt}\rho^2(e^t\cdot m)dt = 4\int_0^{\tau(m,x_1)}x_1(e^t\cdot m)dt \\
                                           &\leq 4\int_0^{\tau(m,x_1)}x_1dt = 4x_1\tau(m,x_1).
\end{align}
So the result follows from Lemma \ref{lemma estimation of tau}.
\end{proof}
As a corollary of the previous two lemma, we have
\begin{corollary}
There exists a uniform $C>0$ such that
\begin{align}
|f_{x_1,0,0}(m) - f_{0,0,0}(m)| \leq C\frac{x_1^2}{\rho^2(m)}.
\end{align}
\end{corollary}\label{corollary estimate of f_x_1,0,0}
Since $[I_1T,\rho\frac{\partial}{\partial\rho}] = 0$, the $\mathbb{C}^*$-action commutes with homothety $\nu_\lambda: (\rho, \Theta)\mapsto (\lambda\rho, \Theta)$ in the coordinate system of $C(M_0) = \mathbb{R}_+\times S_0$. It follows that $\tau(\nu_\lambda(m), \lambda^2x_1) = \tau(m,x_1)$. As a consequence, we have $f_{\lambda^2x_1,0,0}(\nu_\lambda(m)) = \lambda^2f_{x_1,0,0}(m)$. In the $(\rho,\Theta)$-coordinate of $M_0$, we have
\begin{align}
f_{x_1,0,0}(\rho,\Theta) = \rho^2 f_{\rho^{-2}x_1,0,0}(1, \Theta).
\end{align}
Since $f_{x_1,0,0}$ depends real analytically on $x_1$, we may develop as follows:
\begin{align}
f_{\zeta,0,0}(1,\Theta) = \sum_{\nu\geq 0}k_\nu(\Theta)\zeta^\nu,
\end{align}
which is valid for $|\zeta| < \varepsilon_0$. Then we have
\begin{align}
f_{x_1,0,0}(\rho,\Theta) = \rho^2f_{\rho^{-2}x_1,0,0}(1,\Theta) = \sum_{\nu\geq 0}k_\nu(\Theta)\frac{x_1^\nu}{\rho^{2\nu-2}}.
\end{align}
Note that $f_{0,0,0} = \frac{1}{2}\rho^2$, so $k_0(\Theta) = \frac{1}{2}$. And by Corollary \ref{corollary estimate of f_x_1,0,0}, we have $k_1(\Theta) = 0$. So we get
\begin{align}\label{developement of f_x_1,0,0}
f_{x_1,0,0}(\rho,\Theta) = \frac{1}{2}\rho^2 + \sum_{\nu\geq 2}k_\nu(\Theta)\frac{x_1^\nu}{\rho^{2\nu-2}},
\end{align}
which is valid for $|\frac{x_1}{\rho^2}|<\varepsilon_0$.

Similarly, the homothety on $M$ gives rise to the relation $\nu_\lambda^*g_{x_1,0,0} = \lambda^2g_{\lambda^{-2}x_1,0,0}$. For $|\zeta|<\varepsilon_0$, we can develop
\begin{align}
g_{\zeta,0,0}(1,\Theta) = \sum_{\nu\geq 0}h_\nu(\Theta)\zeta^\nu,
\end{align}
where $h_\nu \in \Gamma(S_0, T^*M_0\otimes T^*M_0)$. One may think of each $h_\nu$ as a 2-tensor on $M_0$ by $\frac{1}{\rho}\nu_\rho: T_{(1,\Theta)}M_0 \rightarrow T_{(\rho,\Theta)}M_0$. Note that $\frac{1}{\rho}\nu_\rho$ is an $g_{0,0,0}-$isometry, so the norm of $h_\nu$ on $M_0$ is the same as the norm of it on $S_0$. In this way we may think of $g_{\zeta,0,0}(1,\Theta)$ as a metric on $M_0$ and for $|\frac{x_1}{\rho^2}| < \varepsilon_0$ we have
\begin{align}
g_{x_1,0,0}(\rho,\Theta) = g_{\rho^{-2}x_1,0,0}(1,\Theta) = \sum_{\nu\geq 0}h_\nu(\Theta)(\frac{x_1}{\rho^2})^\nu.
\end{align}
Comparing it with equation \eqref{developement of f_x_1,0,0} and \eqref{abstract formula of g_x}, we see that $h_0(\Theta) = g_{0,0,0}$ and $h_1(\Theta) = \rho^2H_{x_1,0,0}$. In conclusion for $|\frac{x_1}{\rho^2}| < \varepsilon_0$ we have
\begin{align}\label{developement of g_x_1,0,0}
g_{x_1,0,0}(\rho,\Theta) = g_{0,0,0} + x_1H_{x_1,0,0} + \sum_{\nu\geq 2}h_\nu(\Theta)(\frac{x_1}{\rho^2})^\nu.
\end{align}
Consequently, we have
\begin{proposition}\label{proposition asymptotic of g_x_1,0,0}
For $0<\varepsilon<\varepsilon_0$, there exists $C = C(\varepsilon) > 0$ such that for $|\frac{x_1}{\rho^2}| < \varepsilon$, we have
\begin{align}
|g_{x_1,0,0} - g_{0,0,0}|_{g_{0,0,0}} \leq C\frac{x_1}{\rho^2},\\
|g_{x_1,0,0} - g_{0,0,0} - x_1H_{x_1,0,0}|_{g_{0,0,0}} \leq C\frac{x_1^2}{\rho^4}.
\end{align}
\begin{remark}
The above argument of developing the metric $g_{x_1,0,0}$ is used in \cite{Kronheimer1} to show the ALE property of certain hyperk\"ahler quotient.
\end{remark}
\end{proposition}
And we can generalize the above trick to get the asymptotic of $g_{x_1,0,0}$.
\begin{proposition}
The resolution $\phi_{x_1,0,0}: M_{x_1,0,0} \rightarrow M_0\cup\{0\}$ makes $M_{x_1,0,0}$ an asymptotically conic Calabi-Yau manifold with tangent cone $M_0$, in the sense of \cite[Definition 1.11]{Conlon-Hein1}.
\end{proposition}
\begin{proof}
Fix $0<\varepsilon<\varepsilon_0$, taking covariant derivatives of equation \eqref{developement of g_x_1,0,0} with respect to $g_{0,0,0}$, one gets the development of $\nabla^kg_{x_1,0,0}$ in terms of $\frac{x_1}{\rho^2}$, then we have for $|\frac{x_1}{\rho^2}|<\varepsilon$, there exists $C_k = C_k(\varepsilon)>0$ such that
\begin{align}
|\nabla^k(g_{x_1,0,0} - g_{0,0,0})|_{g_{0,0,0}} \leq C_k\frac{x_1}{\rho^{2+k}}.\qed\qedhere
\end{align}
\end{proof}
Finally we turn to the estimation of $g_a$.
\begin{proposition}\label{proposition estimation of g_a}
Fix $0<\varepsilon<\varepsilon_0$, there exists $C = C(\varepsilon) > 0$ such that for any $(m,x)\in\mu^{-1}(0)\times\mathbb{R}^3$ with $|\frac{|x|}{\rho^2(m)}|<\varepsilon$, then at $(m,x)$ we have
\begin{align}
&|\Psi^*g_a - (a^2\sum_{j = 1}^3(dx_j)^2 + g_{0,0,0} + |x|H_x + \frac{1}{a^2}\eta^2 )|_{\Psi^*g_a} \\
&\leq C(\varepsilon) \left(\frac{1}{\rho^2(m)} + \frac{|x|}{\rho^3(m)} + \frac{|x|^2}{\rho^4(m)} \right).
\end{align}
\end{proposition}
\begin{proof}
It is a consequence of formula \eqref{g_a}, Lemma \ref{lemma ODE infinity}, Lemma \ref{estimation of rho^2} and Proposition \ref{proposition asymptotic of g_x_1,0,0}.
\end{proof}
As a consequence we have
\begin{corollary}\label{corollary g_a is asymptotic to cone}
Recall $\dim_\mathbb{R}M = 4n$. For sufficiently large $R>0$, on the region $\{\Psi(m,x) \mid m\in \mu^{-1}(0),x\in\mathbb{R}^3,\rho(m)\geq |x|, \rho(m)\geq R\}$ we have
\begin{itemize}
  \item If $n=2$, then $|\Psi^*g_a - (a^2\sum_{j = 1}^3(dx_j)^2 + g_{0,0,0} + \frac{1}{a^2}\eta^2 )|_{\Psi^*g_a} \leq C\frac{1}{\rho_a^2}$,
  \item If $n\geq3$, then $|\Psi^*g_a - (a^2\sum_{j = 1}^3(dx_j)^2 + g_{0,0,0} + \frac{1}{a^2}\eta^2 )|_{\Psi^*g_a} \leq C\frac{1}{\rho_a}$,
\end{itemize}
for some uniform $C>0$. Here we recall that $\rho_a$ is the distance function measured by $g_a$.
\end{corollary}
\begin{proof}
Based on Proposition \ref{proposition estimation of g_a}, it suffices to note that $|\Omega_0|_{g_{0,0,0}} = C\frac{1}{\rho^2(m)}$, and that on the region $\{\Psi(m,x)\mid m\in \mu^{-1}(0),x\in\mathbb{R}^3,\rho(m)\geq |x|, \rho(m)\geq R\}$, the distance function $\rho_a$ is comparable to $\rho(m)$.
\end{proof}

\begin{remark}
We note that for the case $n=1$, almost all the statements in Subsection \ref{subsection the asymptotic behavior} are vacuous. Formally, when $n=1$, the set $\mu^{-1}(0)$ is just a point, so $\rho(m) = 0$ for any $m\in \mu^{-1}(0)$.
\end{remark}

\begin{remark}
It might be confusing to compare Proposition \ref{estimate of K, generalized} with Corollary \ref{corollary g_a is asymptotic to cone}. Here we note that while the estimate in Proposition \ref{estimate of K, generalized} holds globally, the estimates in Corollary \ref{corollary g_a is asymptotic to cone} hold only on the region $\{\Psi(m,x)\mid m\in \mu^{-1}(0),x\in\mathbb{R}^3,\rho(m)\geq |x|, \rho(m)\geq R\}$. For example, in the proof of Proposition \ref{estimate of K, generalized}, we have applied Proposition \ref{estimate of rho1, generalized}, which is not optimal in the region $\{\Psi(m,x)\mid m\in \mu^{-1}(0),x\in\mathbb{R}^3,\rho(m)\geq |x|, \rho(m)\geq R\}$. However, we will use Corollary \ref{corollary g_a is asymptotic to cone} to prove the asymptotic cone of $g_a$. As one can imagine, the difficulty will be the complement of this region, which will be treated in Proposition \ref{proposition asymptotic cone singular direction}.
\end{remark}

In summary, in the generic region $\{\Psi(m,x)\mid \rho(m)\geq |x|\}$, the metric $g_a$ is asymptotic to a product metric $a^2\sum_{j = 1}^3(dx_j)^2 + g_{0,0,0} + \frac{1}{a^2}\eta^2$. In the region $\{\Psi(m,x)\mid \rho(m)\geq |x|^\alpha\}$ for some $\frac{1}{2}<\alpha \leq 1$, the metric $g_a$ is asymptotic to $a^2\sum_{j = 1}^3(dx_j)^2 + g_{x_1,0,0} + \frac{1}{a^2}\eta^2$.
But if one takes into account the special direction, for example the region
$\{\Psi(m,x)\mid |x|\geq \rho(m)\}$, into account, then one loses the control of the asymptotic behavior.

Recall that the metric $g_{x_1,0,0}$ on $M_0$ is singular at $r = 0$, corresponding to the resolution $\phi_{x_1,0,0}$. When we pass to the asymptotic cone of $g_a$, it is expected that the singular metric $g_{x_1,0,0}$ will shrink to the conic metric $g_{0,0,0}$ with one conic singularity, producing a $\{0\}\times\mathbb{R}^3$ locus of singularity inside the asymptotic cone $M_0\times\mathbb{R}^3$. This is indeed the case, as we shall see in Proposition \ref{asymptotic cone} and Example \ref{example EH blown down}.

\subsection{The volume growth of $g_a$}
It is time to show that the volume growth of $g_a$ is of order $4n-1$, here we recall that $\dim_\mathbb{R}M=4n$.
\begin{proposition}\label{volume growth}
If the 3-Sasakian manifold $S$ admits a locally free symmetry of $\mathbb{S}^1$, and let $g_a$ be the Taub-NUT deformation of $M=C(S)$. Then there exists $C>1$ such that $C^{-1}R^{4n-1} \leq Vol(g_a,B(g_a,R)) \leq CR^{4n-1}$ for any $R>1$.
\end{proposition}
\begin{proof}
Under the coordinate system $\Psi:M_0\times \mathbb{R}^3\rightarrow M/\mathbb{S}^1$, in the generic region $\{\Psi(m,x)\mid \rho(m)\geq |x|\}$, we know that $g_a/\mathbb{S}^1$ is asymptotic to the product metric of the Euclidean metric on $\mathbb{R}^3$ and the conic metric on $M_0$ by Corollary \ref{corollary g_a is asymptotic to cone}. So
the volume growth is at least $R^{4n-1}$. On the other hand, since $V + a^2$ is bounded from below, the volume growth of $(M,g_a)$ is of the same order as
$(M/\mathbb{S}^1,g_a-\frac{1}{V + a^2}\eta^2)$. While the latter space is of ($4n-1$)-dimensional, its volume growth is at most $R^{4n-1}$, the result follows.
\end{proof}

\subsection{The asymptotic cone of $g_a$}
Recall that for any Riemannian manifold $(N,g)$, the asymptotic cone is defined to be the Gromov-Hausdorff limit of $\lambda(N,g)=(N,\lambda^2g)$ as $\lambda\rightarrow 0+$. In this
subsection we will show that the asymptotic cone of $(M,g_a)$ is $M_0\times\mathbb{R}^3$, where $M_0=\mu^{-1}(0)/\mathbb{S}^1$ and $\mathbb{R}^3$ is equipped with the Euclidean
metric.

For the moment, we fix $a = 1$. Recall that by the construction of $g_1$, $(M,g_1)=(M\times \mathbb{S}^1\times\mathbb{R}^3)/\!\!/\!\!/\mathbb{S}^1$, here the symbol $/\!\!/\!\!/$ denote a hyperk\"ahler quotient, and the $\mathbb{S}^1$-action is the diagonal action on the product of $M$ and $\mathbb{S}^1\times\mathbb{R}^3$.
So it follows that
\begin{align}
\lambda(M,g_1) = \lambda \left( (M\times \mathbb{S}^1\times\mathbb{R}^3)/\!\!/\!\!/\mathbb{S}^1 \right) = \lambda(M\times \mathbb{S}^1\times\mathbb{R}^3)/\!\!/\!\!/\mathbb{S}^1,
\end{align}
since hyperk\"ahler quotient commutes with rescaling. Note that the metric on $\lambda(M\times \mathbb{S}^1\times\mathbb{R}^3)$ is $\lambda^2g_0 + \lambda^2g_{\mathbb{S}^1} +
\lambda^2g_{\mathbb{R}^3}$, where $g_{\mathbb{S}^1},g_{\mathbb{R}^3}$ are standard Euclidean metrics. So it follows that the moment map is rescaled accordingly
\begin{align}
\mu_{\lambda(M\times \mathbb{S}^1\times\mathbb{R}^3)} = \lambda^2\mu_{M\times \mathbb{S}^1\times\mathbb{R}^3}.
\end{align}
Consider a change of coordinates $\phi_\lambda: (M\times\mathbb{S}^1_\lambda\times\mathbb{R}^3,g_0+\lambda^2g_{\mathbb{S}^1}+g_{\mathbb{R}^3}) \rightarrow \lambda(M\times
\mathbb{S}^1\times\mathbb{R}^3)$ defined by $(m,e^{i\theta},q)\mapsto (\lambda^{-1}m,e^{i\theta},\lambda^{-1}q)$. Here $\mathbb{S}^1_\lambda$ is the notation for $\mathbb{S}^1$
equipped with the metric $\lambda^2g_{\mathbb{S}^1}$. Then $\phi_\lambda$ is an isometry. Denote by $\mu_\lambda$ the pullback $\phi_\lambda^*\mu_{\lambda(M\times
\mathbb{S}^1\times\mathbb{R}^3)}$, then
\begin{align}
\mu_\lambda(m,e^{i\theta},q) = \mu(m) - \lambda q,
\end{align}
where $\mu$ is the moment map $(x_1,x_2,x_3)$ of $M$ (see Subsection \ref{subsection hk cone with S1 action}). In conclusion we have
\begin{align}
\lambda(M,g_1) = \mu_\lambda^{-1}(0)/\mathbb{S}^1.
\end{align}
Here $\mu_\lambda^{-1}(0)$ is equipped with the restriction of the metric of $(M\times\mathbb{S}^1_\lambda\times\mathbb{R}^3,g_0+\lambda^2g_{\mathbb{S}^1}+g_{\mathbb{R}^3})$.

Now we want to show that the Gromov-Hausdorff limit of $\mu_\lambda^{-1}(0)/\mathbb{S}^1$ as $\lambda\rightarrow 0+$ is $(\mu^{-1}(0)/\mathbb{S}^1)\times\mathbb{R}^3$. Consider the
map from $M_0\times \mathbb{R}^3$ to $(M,\lambda^2 g_1)$ defined in twist coordinates by $(m,q)\mapsto \Psi(\lambda^{-1}m,\lambda^{-1}q)$, note that this map is only well-defined up
to $\mathbb{S}^1$-action on $M$, but since the length of $\mathbb{S}^1$-orbit is controlled by $\lambda$, it does not affect the asymptotic cone. We will show that this map is a
Gromov-Hausdorff map. More precisely, we want to show that for any $\varepsilon>0, R>0$, there exists $\lambda_0 > 0$ such that for any $m_1,m_2 \in M_0, q_1,q_2\in \mathbb{R}^3$
with $\rho(m_1),\rho(m_2),|q_1|,|q_2|<R$ and $\lambda < \lambda_0$, we have
\begin{align}
\Bigg\lvert \lambda d_{g_1}(\Psi(\lambda^{-1}m_1,\lambda^{-1}q_1),\Psi(\lambda^{-1}m_2,\lambda^{-1}q_2)) - \sqrt{d_{M_0}(m_1,m_2)^2 + |q_1-q_2|^2} \Bigg\rvert < \varepsilon.
\end{align}

If we suppose that $\rho(m_1),\rho(m_2) > d_0$ for some $d_0 > 0$, then the above inequality follows from Corollary \ref{corollary g_a is asymptotic to cone}. To treat the case where
$m_1,m_2$ are close to the vertex, we will prove the corresponding inequality before taking the quotient by $\mathbb{S}^1$, taking into account the identification between
$(M,\lambda^2g_1)$ and $\mu_\lambda^{-1}(0)/\mathbb{S}^1$.

\begin{proposition}\label{proposition asymptotic cone singular direction}
For any $\varepsilon > 0, R_0 > 0$, there exist $d_0 >0,\lambda_0>0$ such that for any $m_1,m_2\in M_0, q_1,q_2\in \mathbb{R}^3, \theta_1,\theta_2\in \mathbb{R}$ with $\rho(m_i)<d_0, |q_i|< R_0,\lambda<\lambda_0$, the following inequality holds:
\begin{align}
\Bigg|d_{\mu_\lambda^{-1}(0)}((\Psi(m_1, \lambda q_1),e^{i\theta_1},q_1),(\Psi(m_2, \lambda q_2),e^{i\theta_2},q_2)) - |q_1-q_2|\Bigg| < \varepsilon.
\end{align}
\end{proposition}
\begin{proof}
As we shall see in Lemma \ref{lemma H(m,x)}, the distance of $m\in M_0$  measured by the metric $g_x$ to $0$ is bounded by $H(m,|x|)$ and
$H(m,|x|)\leq C\rho(m)$. Recall Remark \ref{remark size of ED} that the size of $ED_x$ is controlled by $C\sqrt{|x|}$ for some uniform $C>0$. Now, let $x_i = \lambda q_i\in \mathbb{R}^3$ be the
hyperk\"ahler moment of $\Psi(m_i, \lambda q_i)$, then $|x_i|\leq \lambda_0R_0$.

Now we join $(\Psi(m_1, x_1),e^{i\theta_1},q_1),(\Psi(m_2, x_2),e^{i\theta_2},q_2)$ in three steps. Without loss of generality, we may assume that $|x_1|\geq
|x_2|$

Fix $\varepsilon^\p > 0$ sufficiently small. Firstly, we move $m_1$ along the radius of $M_0$ to $m_1^\p$ such that $\frac{|x_1|}{\rho^2(m_1^\p)} = \varepsilon^\p$ while keeping $x_1$ unchanged, its length is controlled by $H(m_1, |x_1|) + H(m_1^\p, |x_1|)\leq C(\rho(m_1) + \rho(m_1^\p)) = C(\rho(m_1) + \sqrt\frac{|x_1|}{\varepsilon^\p})$, which is bounded by $C(d_0 + \frac{1}{\sqrt{\varepsilon^\p}}\sqrt{\lambda_0R_0})$.

Secondly, we keep $m_1^\p$ unchanged and move $x_1$ to $x_2$ along the line segment. In the factor $M$, by formula \eqref{g_0} and Proposition \ref{proposition asymptotic of g_x_1,0,0}, for $|\frac{x_1}{\rho^2(m)}| \leq \varepsilon^\p$, the tensor $g_0$ in the $x-$direction at $\Psi(m,x)$ is bounded by
$\frac{C(\varepsilon^\p)}{\rho^2(m)}\sum_{i = 1}^3(dx_i)^2$. So the length of the segment from $\Psi(m_1^\p, x_1)$ to $\Psi(m_1^\p, x_2)$ in $M$ is bounded by $\frac{\sqrt{C(\varepsilon^\p)}}{\rho(m_1^\p)}|x_1-x_2| \leq 2\frac{\sqrt{C(\varepsilon^\p)}}{\rho(m_1^\p)}|x_1| = 2\sqrt{C(\varepsilon^\p) \varepsilon^\p |x_1|} \leq 2\sqrt{C(\varepsilon^\p)\varepsilon^\p}\sqrt{\lambda_0R_0}$.
Here we note that one must apply the formula of $g_0$ instead of $g_1$, since we have not yet taken the quotient by $\mathbb{S}^1$. In the factor of $\mathbb{R}^3$, the length is $|q_1-q_2|$, since we chose the line segment between them.

Finally, we keep $x_2$ unchanged and move $m_1^\p$ to $m_2$, the length is bounded by $H(m_2, |x_2|) + H(m_1^\p, |x_2|) + C\sqrt{|x_2|}\leq C(d_0 + (1 + \frac{1}{\sqrt{\varepsilon^\p}})\sqrt{\lambda_0R_0})$.

In addition, there might be a difference between $\theta_1$ and $\theta_2$, but in the factor $\mathbb{S}_\lambda^1$, its length is controlled by $2\pi\lambda < 2\pi\lambda_0$.

So the result follows by choosing $d_0$ and $\lambda_0$ sufficiently small such that $C(d_0 + \frac{1}{\sqrt{\varepsilon^\p}}\sqrt{\lambda_0R_0}) + 2\sqrt{C(\varepsilon^\p)\varepsilon^\p}\sqrt{\lambda_0R_0} + C(d_0 + (1 + \frac{1}{\sqrt{\varepsilon^\p}})\sqrt{\lambda_0R_0}) + 2\pi\lambda_0 < \varepsilon$.
\end{proof}

\begin{lemma}\label{lemma H(m,x)}
For $x\neq 0,x\in\mathbb{R}^3$, let $H(m, |x|)$ be the length of the segment from the vertex to $m\in M_0$ with respect to $g_x$. Then we have
\begin{align}
H(m, |x|) = \sqrt{|x|}H(\frac{1}{\sqrt{|x|}}m, 1),
\end{align}
and there exists a uniform $C>0$ such that $H(m, |x|) \leq C\rho(m)$.
\end{lemma}
\begin{proof}
First, we note that by the $\Sp(1)$-action, the distance in question depends only on $|x|$, and we may assume that $x = (x_1,0,0)$ for $x_1>0$. Next by the relation $\nu_\lambda^*g_{x_1,0,0} = \lambda^2g_{\lambda^{-2}x_1,0,0}$, we have $H(m, |x|) = \sqrt{|x|}H(\frac{1}{\sqrt{|x|}}m, 1)$. Then, by Proposition \ref{proposition asymptotic of g_x_1,0,0}, there exists $C > 0$ such that for $\rho(m) > 1$, $H(m,1)\leq C\rho(m)$. Finally, we note that $g_x(\rho\frac{\partial}{\partial \rho}, \rho\frac{\partial}{\partial \rho}) = \rho^2 - 4|x|^2V$ by formula \eqref{abstract formula of g_x}. This function is bounded and continuous as $\rho(m)\rightarrow 0$, in fact, it converges to $0$. This shows that $H(m,1)$ can be bounded by $C\rho(m)$ by some $C > 0$ when $\rho(m)$ is small, which finishes the proof.
\end{proof}

It follows that the Gromov-Hausdorff limit of $\mu_\lambda^{-1}(0)$ as $\lambda\rightarrow 0+$ is $\mu^{-1}(0)\times\mathbb{R}^3$. Then by \cite[Theorem 2.1]{zbMATH04067795}, we may
take the quotient by $\mathbb{S}^1$ to get
\begin{align}
\lim_{\lambda\rightarrow 0+}\mu_\lambda^{-1}(0)/\mathbb{S}^1=(\mu^{-1}(0)\times\mathbb{R}^3)/\mathbb{S}^1 = M_0\times \mathbb{R}^3.
\end{align}
Recall that $\lambda(M,g_1) = \mu_\lambda^{-1}(0)/\mathbb{S}^1$, so we have proved that
\begin{proposition}\label{asymptotic cone}
The asymptotic cone of $(M,g_1)$ is $M_0\times \mathbb{R}^3$.
\end{proposition}

We fix $a = 1$ in the above discussion; here the point is that we need the Taub-NUT deformation $g_1$ to be obtained by an $\mathbb{S}^1-$hyperk\"ahler quotient. In fact, this holds for arbitrary $a>0$: it suffices to note that in the proof of Proposition \ref{deformation}, if $T$ generates an $\mathbb{S}^1$-action, then the vector field $Z = aT + \frac{\partial}{\partial q_0}$ on $M\times \mathbb{H}$ descends to $M\times (\mathbb{R}/(2\pi a^{-1}\mathbb{Z})\times \mathbb{R}^3)$ and generates an $\mathbb{S}^1$-action. So all the results obtained for $g_1$ still hold for $g_a$. In particular, we have shown that

\begin{proposition}\label{asymptotic cone of g_a}
Suppose that the 3-Sasakian manifold $S$ admits a locally free symmetry of $\mathbb{S}^1$. For any $a>0$, let $g_a$ be the Taub-NUT deformation of $M=C(S)$. Then the asymptotic cone of $(M,g_a)$ is $M_0\times \mathbb{R}^3$.
\end{proposition}

\begin{remark}
Recall that Proposition \ref{estimate of K, generalized} states that the sectional curvature of $g_a$ is bounded by $\frac{C}{\rho_a}$, Proposition \ref{volume growth} states that the
volume growth of $g_a$ is of order $(4n-1)$ and Proposition \ref{asymptotic cone of g_a} implies that the asymptotic cone of $g_a$ is ($4n-1$)-dimensional. It is in this sense that I call $g_a$ an ALF metric.
\end{remark}

\section{ALF Calabi-Yau metrics on crepant resolutions}\label{section ALF CY metrics on crepant resolutions}
In this section, we will prove the existence of ALF Calabi-Yau metrics on certain crepant resolutions of hyperk\"ahler cones. We first explain the construction of the approximately Ricci-flat metric, then we will show how to get the Calabi-Yau metric by applying a Tian-Yau's method.

The main theorem of this section is the following result:
\begin{theorem}\label{theorem main theorem}
(i) Let $S$ be a compact connected 3-Sasakian manifold of dimension $4n-1$ ($n\geq 1$) admitting a locally free $\mathbb{S}^1$-symmetry. Let $M$ be the hyperk\"ahler cone over $S$, then for any $a > 0$, the Taub-NUT deformation $(g_a, I_i^a,\omega_i^a)$ with respect to this $\mathbb{S}^1$-symmetry is an ALF hyperk\"ahler metric on $M$ in the sense that the sectional curvature $K_{g_a}$ is bounded by $\frac{C}{\rho_a}$ for some $C>0$, the volume growth of $g_a$ is of order $4n-1$ and its asymptotic cone is a metric cone of dimension $4n-1$. Here $\rho_a$ is the distance measured by $g_a$.

(ii) Assume furthermore that $n\geq 2$ and there is a finite group $\Gamma$ acting on $S$ whose extension to $M$ satisfies:
\begin{itemize}
  \item The group $\Gamma$ acts $I_1$-holomorphically;
  \item The group $\Gamma$ preserves the deformed $I_1-$potential $K_{\mathrm{ALF}}^a$ (see Cor \ref{corollary K_ALF^a}).
\end{itemize}
Let $\pi: Y\rightarrow M/\Gamma$ be an $I_1-$holomorphic crepant resolution, then for any compactly supported K\"ahler class of $Y$ and any $c > 0$, there exists an ALF Calabi-Yau metric $\omega$ in this class which is asymptotic to $c\omega_1^a$ near the infinity. More precisely, we have
\begin{align}
|\nabla^k(\omega - c\pi^*(i\partial\bar\partial K_{\mathrm{ALF}}^a))|_\omega \leq C(k,\varepsilon)(1 + \rho_\omega)^{-4n + 3 + \varepsilon},
\end{align}
where $\varepsilon > 0$ is sufficiently small, $\rho_\omega$ is the distance from a point in $Y$ measured by $\omega$ and $k\geq 0$.
\end{theorem}
Note that the first part of the theorem is nothing more than a brief summary of the results of the previous sections.

\subsection{Asymptotic ALF Calabi-Yau metric}\label{subsection asymptotic ALF CY metric}
In this subsection we will construct an asymptotic ALF Calabi-Yau metric by gluing. First we note that by the assumption of Theorem \ref{theorem main theorem}, the composition $\Phi_a^{-1}\circ \pi: Y \rightarrow (M/\Gamma,I_1^a)$ is also a crepant resolution.

\begin{lemma}\label{lemma h_alpha}
Assumption as in Theorem \ref{theorem main theorem}, for $\alpha > \alpha_0$ (where $\alpha_0$ is defined in Corollary \ref{corollary dd^cK^alpha>0}), there exists a smooth plurisubharmonic function $h_\alpha$ on $Y$ which is strictly plurisubharmonic and equal to $(K_{\mathrm{ALF}})^\alpha$ outside a compact set, where $K_{\mathrm{ALF}} = K_{\mathrm{ALF}}^a\circ \pi = K_1^a\circ \Phi_a^{-1}\circ \pi$ for a fixed $a>0$.
\end{lemma}
\begin{proof}
Let $\psi:\mathbb{R}\rightarrow\mathbb{R}$ be a smooth function such that $\psi^\p,\psi^{\p\p}\geq 0$, $\psi(t) = 3$ if $t<2$ and $\psi(t) = t$ if $t > 4$. Define $h_\alpha = \psi((K_{\mathrm{ALF}})^\alpha)$, then the lemma follows from \eqref{dd^c of composition} and Corollary \ref{corollary dd^cK^alpha>0}.
\end{proof}

\begin{lemma}\label{lemma hat omega}
Assumption as in Theorem \ref{theorem main theorem}. Fix $a>0$, for any compactly supported K\"ahler class $[\omega_Y]\in H^2_c(Y,\mathbb{R})$ represented by the K\"ahler form $\omega_Y$ and any $c>0$, there exists a K\"ahler form $\hat{\omega}$ on $Y$ in the same class of $\omega_Y$ and $\hat{\omega} = c(\Phi_a^{-1}\circ \pi)^*\omega_1^a$ outside a compact set.
\end{lemma}
\begin{proof}
Since $[\omega_Y]\in H^2_c(Y,\mathbb{R})$, by \cite[Corollary A.3]{Conlon-Hein1}, there exists a smooth real function $v$ such that $\omega_Y = -i\partial\bar\partial v$ when $K_{\mathrm{ALF}} > R$ for $R$ sufficiently large. Here we note that by Lemma \ref{lemma h_alpha}, $Y$ is clearly 1-convex. Also $\dim_\mathbb{C}Y>2$ by assumption that $\dim_\mathbb{R}(S) = 4n-1$ and $n\geq 2$. So indeed we can apply \cite[Corollary A.3]{Conlon-Hein1}.

Fixing $\alpha \in (\alpha_0, 1)$, we can assume (by enlarging $R$ if necessary) that $h_\alpha = (K_{\mathrm{ALF}})^\alpha$ and $h_1 = K_{\mathrm{ALF}}$ when $K_{\mathrm{ALF}}>R$ where $h_\alpha,h_1$ are defined in Lemma \ref{lemma h_alpha}.

Fix a cutoff function $\chi: \mathbb{R}\rightarrow [0,1]$ satisfying $\chi(s) = 0$ if $s<2R$ and $\chi(s) = 1$ if $s>3R$. Define $\zeta: Y \rightarrow \mathbb{R}$ by $\zeta(y) = \chi(K_{\mathrm{ALF}}(y))$. For $S>2$, define $\zeta_S(y) = \chi(\frac{K_{\mathrm{ALF}}(y)}{S})$. Note that $0<2R<3R<2SR<3SR<+\infty$.

Given $c>0$, we construct
\begin{align}
\hat{\omega} = \omega_Y + i\partial\bar\partial(\zeta v) + C i\partial\bar\partial((1-\zeta_S)h_\alpha) + ci\partial\bar\partial h_1,
\end{align}
with $C$ and $S$ to be determined. It is clear that $\hat{\omega}$ is in the same cohomology class as $\omega_Y$.

On the region $\{K_{\mathrm{ALF}} < 2R \}$, $\hat{\omega} = \omega_Y + Ci\partial\bar\partial h_\alpha + ci\partial\bar\partial h_1 \geq \omega_Y > 0$.

On the region $\{3R < K_{\mathrm{ALF}} < 2SR\}$, $\hat{\omega} = Ci\partial\bar\partial h_\alpha + ci\partial\bar\partial h_1 > 0$.

On the region $\{3SR < K_{\mathrm{ALF}}\}$, $\hat{\omega} = ci\partial\bar\partial h_1 = c(\Phi_a^{-1}\circ \pi)^*\omega_1^a > 0$.

On the region $\{2R \leq K_{\mathrm{ALF}} \leq 3R \}$, $\hat{\omega} = \omega_Y + i\partial\bar\partial(\zeta v) + Ci\partial\bar\partial h_\alpha + ci\partial\bar\partial h_1 > 0$ if $C$ is made large enough.

Finally on the region $\{2SR\leq K_{\mathrm{ALF}} \leq 3SR \}$, $\hat{\omega} = C i\partial\bar\partial((1-\zeta_S)h_\alpha) + ci\partial\bar\partial h_1$. By Proposition \ref{proposition dK wedge d^cK}, we know that
\begin{align}
|dK_{\mathrm{ALF}}|_{i\partial\bar\partial K_{\mathrm{ALF}}}, |d^cK_{\mathrm{ALF}}|_{i\partial\bar\partial K_{\mathrm{ALF}}} \leq C^\p(K_{\mathrm{ALF}})^{\frac{1}{2}}.
\end{align}
After some simple derivation, it follows that $|i\partial\bar\partial((1-\zeta_S)h_\alpha)|_{i\partial\bar\partial h_1} \leq C^{\p\p}S^{-(1-\alpha)}$ on this region, so $\hat{\omega} > 0$ if $S$ is made large enough.
\end{proof}

\subsection{Hein's package}

To obtain the Calabi-Yau metric of Theorem \ref{theorem main theorem}, we want to apply the Tian-Yau method \cite{zbMATH04186562, zbMATH00059791}, as precised by the thesis of Hans-Joachim Hein \cite{Heinthesis}. In order to state the result we need, we will begin with some definitions given in \cite{Heinthesis}.

First we recall the definition of \text{SOB} property.
\begin{definition}\label{definition of SOB}
A complete noncompact Riemannian manifold $(N,g)$ of dimension $m>2$ is called $\op{SOB}(\beta)$, ($\beta\in\mathbb{R_+}$) if there exists $x_0\in N$ and $C\geq 1$ such that
$A(x_0,s,t)=\{s\leq r(x)\leq t\}$ is connected for all $t>s>C$, $Vol(g,B(x_0,s))\leq Cs^\beta$ for all $s\geq C$, and $Vol(g,B(x,(1-\frac{1}{C})r(x)))\geq \frac{1}{C}r(x)^\beta$ and
$\op{Ric}(x)\geq -Cr(x)^2$ if $r(x)=d(x_0,x)\geq C$.
\end{definition}
\begin{remark}\label{remark connectivity condition in SOB}
It is observed by \cite[Section 7]{zbMATH07131296} that to apply Proposition \ref{hein's thesis} it suffices to check the following property instead of the connectivity of $A(x_0,s,t)$: For sufficiently large $D$, any two points $m_1,m_2\in N$ with $r(m_i)=D$ can be joined by a curve of length at most $CD$, lying in the annulus $A(x_0,C^{-1}D,CD)$, for a uniform constant $C$.
\end{remark}

\begin{proposition}\label{proposition SOB(4n-1)}
The approximately Calabi-Yau metric $(Y, \hat\omega)$ constructed in Lemma \ref{lemma hat omega} is $\op{SOB}(4n-1)$ in the sense of Remark \ref{remark connectivity condition in SOB}.
\end{proposition}
\begin{proof}
By the gluing construction of $\hat\omega$, it suffices to show the $\op{SOB}(4n-1)$ property for the Taub-NUT deformation $(M,g_a)$,

One inspects the conditions in Definition \ref{definition of SOB}. Proposition \ref{proposition estimation of g_a} guarantees the properties concerning the volume growth of order $4n-1$, since $g_a$ is hyperk\"ahler, the property concerning Ricci curvature is satisfied. Finally we need to show the connectivity condition in Remark \ref{remark connectivity condition in SOB}.

We note that by Corollary \ref{corollary S_0 is connected}, the link $S_0$ of $M_0$ is connected, so the link of $M_0\times \mathbb{R}^3$ is also connected.

For any two points $\Psi(m_i,x_i)\in M$ ($m_i\in \mu^{-1}(0),x_i\in\mathbb{R}^3$) with distance $\rho_a(\Psi(m_i,x_i))=D$, if both are in the generic region $\{\Psi(m,x)\mid m\in \mu^{-1}(0),x\in \mathbb{R}^3,\rho(m)\geq |x|\}$, then since the metric $g_a$ is asymptotic to a cone with connected link, we know that they can be joined by a curve in the annulus $B(0,CD)\setminus B(0,C^{-1}D)$ of length at most $CD$, for a sufficiently large uniform $C>0$. Now if one point $\Psi(m_i,x_i)$ lies in the special region $\{\Psi(m,x)\mid m\in \mu^{-1}(0),x\in \mathbb{R}^3,\rho(m)\leq |x|\}$, then it suffices to join it to the generic region by moving $m_i$ to $m_i^\p$ along the radius of $\mu^{-1}(0)$ until $\rho(m_i^\p) = |x_i|$ while keeping $x_i$ fixed. The length of this curve is bounded by $H(m_i^\p, |x_i|) \leq C|x_i|$. Note that for $m\in \{\Psi(m,x)\mid m\in \mu^{-1}(0),x\in \mathbb{R}^3,\rho(m)\leq |x|\}$, by \eqref{g_a}, we have $a|x| \leq \rho_a(\Psi(m,x)) \leq C|x|$ for some $C>0$, so the result follows.
\end{proof}

Next we recall the definition of Quasi-atlas.
\begin{definition}
Let $(N,\omega_0,g_0)$ be a complete K\"ahler manifold. A $C^{k,\alpha}$ quasi-atlas for $(N,\omega_0,g_0)$ is a collection $\{\Phi_x\mid x\in A\}$, $A\subset N$, of holomorphic local
diffeomorphisms $\Phi_x: B\rightarrow N$, $\Phi_x(0)=x$, from $B=B(0,1)\subset \mathbb{C}^m$ into $N$ which extend smoothly to the closure $\bar{B}$, and such that there exists
$C\geq 1$ with $\op{inj}_{\Phi_x^*g_0}\geq \frac{1}{C}$, $\frac{1}{C}g_{\mathbb{C}^m}\leq \Phi_x^*g_0\leq Cg_{\mathbb{C}^m}$, and
$||\Phi_x^*g_0||_{C^{k,\alpha}(B,g_{\mathbb{C}^m})}\leq C$ for all $x\in A,$ and such that for all $y\in N$ there exists $x\in A$ with $y\in \Phi_x(B)$ and
$d_{g_0}(y,\partial\Phi_x(B))\geq \frac{1}{C}$.
\end{definition}
\begin{proposition}\label{proposition quasi-atlas}
The K\"ahler metric $(Y,\hat\omega)$ admits a $C^{k,\alpha}$ quasi-atlas for any $k,\alpha$.
\end{proposition}
\begin{proof}
It suffices to prove the result for the model metric $(M,g_a,\omega_1^a)$ near infinity. This is a consequence of Proposition \ref{estimate of K, generalized}, \cite[Lemma 4.3]{Heinthesis} and the fact that $g_a$ is Ricci-flat K\"ahler.
\end{proof}

The following result of \cite{Heinthesis} will be used to show the existence of a Calabi-Yau metric $\omega$ on $Y$ modeled on $c\omega_1^a$.

\begin{proposition}\label{hein's thesis}
Let $(N,\omega_0,g_0)$ be a complete noncompact K\"ahler manifold of complex dimension $m$ with a $C^{3,\alpha}$ quasi-atlas which satisfies $\op{SOB}(\beta)$ for some $\beta>2$. Let
$f\in C^{2,\alpha}(N)$ satisfies $|f|\leq Cr^{-\mu}$ on $\{r>1\}$ for some $\beta>\mu>2$. Then there exist
$\bar{\alpha}\in (0,\alpha]$ and $u\in C^{4,\bar{\alpha}}$ such that $(\omega_0+i\partial\bar\partial u)^m=e^f\omega_0^m$ and that $\omega_0+i\partial\bar\partial u$ is a K\"ahler form uniformly equivalent to $\omega_0$. If in addition $f\in C^{k,\bar\alpha}_{loc}(N)$ for some $k\geq 3$, then all such solutions $u$ belong to
$C^{k+2,\bar\alpha}_{loc}(N)$.

Moreover, if there is a function $\tilde\rho$ on $N$ comparable to $1+d_{g_0}(x_0,-)$ for some $x_0\in N$, and $\tilde\rho$ satisfies $|\nabla\tilde\rho| +
\tilde\rho|dd^c\tilde\rho|\leq C$ for some $C>0$, then we have the decay estimate $|u|\leq C(\varepsilon)r^{2-\mu+\varepsilon}$ for any sufficiently small $\varepsilon>0$.
\end{proposition}

\begin{lemma}\label{lemma bound on laplacian of sqrt K}
Let $w = (K_1^a)^\frac{1}{2}$, then $|\nabla w| + w|dd^c_{I_1^a} w|$ is bounded on $(M,g_a)$.
\end{lemma}
\begin{proof}
First we calculate $dw = \frac{1}{2}(K_1^a)^{-\frac{1}{2}}dK_1^a$, so by Proposition \ref{proposition dK wedge d^cK} we have $|dw|\leq C$.
Next for $dd^cw$, applying formula \eqref{dd^c of composition} we have
\begin{align}
dd^c_{I_1^a}(K_1^a)^\frac{1}{2} = \frac{1}{2}(K_1^a)^{-\frac{3}{2}}\left[-\frac{1}{2}dK_1^a\wedge d^c_{I_1^a}K_1^a + K_1^add^c_{I_1^a}K_1^a \right].
\end{align}
So it follows that $w|dd^c_{I_1^a}w|$ is also bounded.
\end{proof}

\begin{lemma}\label{lemma sqrt K is comparable to radius}
The function $(K_1^a)^{\frac{1}{2}}$ is comparable to $\rho_a$ outside of a compact set, where $\rho_a$ is the radius measured by $g_a$.
\end{lemma}
\begin{proof}
Recall that $K_1^a = \frac{1}{2}\rho^2 + a^2x_1^2 + \frac{a^2}{2}(x_2^2 + x_3^2)$. So on the generic region $\{\Psi(m,x)\mid m\in \mu^{-1}(0),x\in \mathbb{R}^3,\rho(m)\geq |x|\}$, the function $K_1^a$ is comparable to $\rho_a$ outside a compact set by Corollary \ref{corollary g_a is asymptotic to cone}.

As for the special region $\{\Psi(m,x)\mid m\in \mu^{-1}(0),x\in \mathbb{R}^3,\rho(m)\leq |x|\}$, by the proof of Proposition \ref{proposition SOB(4n-1)}, we know that $\rho_a(\Psi(m,x))$ is comparable to $|x|$ outside a compact set and it is clear that $K_1^a(\Psi(m,x))\geq \frac{a^2}{2}|x|^2$, so it suffices to show that $\rho(\Psi(m,x))\leq C|x|$.

Using $\Sp(1)$-equivariance, we may assume that $x = (x_1,0,0)$ with $x_1>0$, so $\Psi(m,x) = e^{\tau(m,x_1)}\cdot m$ (see Subsection \ref{subsection the hk metric of M_x_1,0,0} for notation). By \eqref{ODE new system1} and \eqref{ODE new system2}, for $\tau\in\mathbb{R}$, $m\in \mu^{-1}(0)$, we have
\begin{align}
\frac{d}{d\tau}(\frac{x_1}{\rho}(e^\tau\cdot m)) = \frac{\rho^2|T|^2 - 2x_1^2}{\rho^3}(e^\tau\cdot m).
\end{align}
By Lemma \ref{lemma 4x_1^2V}, we have
\begin{align}
\frac{d}{d\tau}(\frac{x_1}{\rho}(e^\tau\cdot m)) \geq \frac{|T|^2}{2\rho}(e^\tau\cdot m) \geq C \rho(e^\tau\cdot m).
\end{align}
Now $\rho(\Psi(m,x))$ is the $g_0-$length of the radius from $0$ to $\Psi(m,x)$. We may join $0$ and $\Psi(m,x)$ by first join $0$ and $m\in \mu^{-1}(0)$ along radius, then join $m$ and $\Psi(m,x)$ by the curve $e^t\cdot m$ for $t\in[0,\tau(m,x_1)]$. It follows that
\begin{align}
\rho(\Psi(m,x)) \leq \rho(m) + \int_0^{\tau(m,x_1)}|T|(e^t\cdot m)dt.
\end{align}
Since $|T|\leq C\rho$, we find that
\begin{align}
\rho(\Psi(m,x)) &\leq \rho(m) + C\int_0^{\tau(m,x_1)}\frac{d}{dt}(\frac{x_1}{\rho}(e^t\cdot m))dt \\
&= \rho(m) + C\frac{x_1}{\rho(\Psi(m,x))}.
\end{align}
Since we are in the complement a compact set of the special region characterized by $\rho(m)\leq |x|$, so it follows that $\rho(e^{\tau(m,x_1)}\cdot m)\leq Cx_1$ for some $C > 1$. And more generally $\rho(\Psi(m,x))\leq C|x|$.
\end{proof}

\begin{corollary}\label{corollary weight function}
The smooth function $h_{\frac{1}{2}}$ defined in Lemma \ref{lemma h_alpha} is comparable to $1 + \hat\rho$ and $|\nabla h_{\frac{1}{2}}| + h_{\frac{1}{2}}|dd^c h_{\frac{1}{2}}|$ is bounded on $(Y,\hat\omega)$. Here $\hat\rho$ is the distance from some point $y_0\in Y$ measured by $\hat\omega$.
\end{corollary}
\begin{proof}
Note that $\Phi_a^{-1}\circ \pi: (Y,\hat\omega)\rightarrow (M/\Gamma, \omega_1^a)$ is a holomorphic isometry outside a compact set, so Lemma \ref{lemma bound on laplacian of sqrt K} and Lemma \ref{lemma sqrt K is comparable to radius} imply the result.
\end{proof}

Now we can give a proof of the main theorem.
\begin{proof}[Proof of Theorem \ref{theorem main theorem}]
The first part is a brief summary of Proposition \ref{estimate of K, generalized}, \ref{volume growth} and \ref{asymptotic cone}.

Regarding the construction of $\omega$, by Proposition \ref{proposition SOB(4n-1)} and Proposition \ref{proposition quasi-atlas} we can apply Proposition \ref{hein's thesis} to
$(Y,\hat\omega)$ and its Ricci potential $\hat{f}$ with $\beta=4n-1$. Here
\begin{align}
\hat{f} = \log\frac{\Omega_Y\wedge\bar\Omega_Y}{(\hat\omega/c)^{2n}},
\end{align}
and $\Omega_Y = \pi^*\Omega_M$, $\Omega_M = (\omega_2 + i\omega_3)^n$ are the corresponding holomorphic volume forms.

Note that by Proposition \ref{complex symplectic structure}, we have $\Phi_a^*\Omega_M = \Omega_M$. It follows that $\hat f$ is compactly supported on $Y$. In particular, $|\hat f|\leq C\hat{\rho}^{-\mu}$ for any $2 < \mu < 4n-1$. By Proposition \ref{hein's thesis}, we get a smooth solution $u\in C^{4,\bar\alpha}(Y)$ of the Monge-Amp\`{e}re equation
\begin{align}
(\hat\omega + i\partial\bar\partial u)^{2n} = e^{\hat f}\hat\omega^{2n}.
\end{align}
Let $\omega = \hat\omega + i\partial\bar\partial u$, then $\omega$ is Calabi-Yau and uniformly equivalent to $\hat\omega$. Moreover, by the second part of Proposition \ref{hein's thesis}. we have $|u|\leq C(\varepsilon)\rho_\omega^{-4n+3+\varepsilon}$ for $\varepsilon>0$ sufficiently small.

If we think of $\omega$ as given, then the Monge-Amp\`{e}re equation can be written as
\begin{align}
(e^{\hat f} - 1)\hat\omega^{2n} = i\partial\bar\partial u \wedge \sum_{k = 1}^{2n-1}\omega^k\wedge \hat\omega^{2n-1-k},
\end{align}
and it can be thought as an elliptic equation of $u$. Outside a compact set, the left hand side is zero so by Schauder estimates on each chart of the Quasi-atlas outside this compact set, we find that $|\nabla^k u|_\omega \leq C(k,\varepsilon) \rho_\omega^{-4n+3+\varepsilon}$ for $k\geq 0$.
\end{proof}

\section{Examples of higher-dimensional ALF Calabi-Yau metrics}\label{section Examples}
In this section we will apply Theorem \ref{theorem main theorem} to get many ALF Calabi-Yau metrics. First we start with the Taubian-Calabi metric on $\mathbb{C}^{2n}$ and an ALF metric on the total space of the canonical bundle $\mathcal{K}_{\mathbb{CP}^{2n-1}}$ of $\mathbb{CP}^{2n-1}$ asymptotic to it. Then we will give two types of generalizations. One concerns the crepant resolutions of isolated singularity $\mathbb{C}^{2n}/\Gamma$. The other is about the homogeneous 3-Sasakian manifolds. More precisely, we will show that for homogeneous 3-Sasakian manifold associated to group $\Sp(n), (n\geq 1)$, $\SU(m), (m\geq 3)$, $\SO(l), (l\geq 5)$ and $G_2$, it admits locally free $\mathbb{S}^1$-symmetry, so there exist ALF hyperk\"ahler metrics on its corresponding hyperk\"ahler cone, and Calabi-Yau metrics on the canonical bundle of its twistor space. Some non-homogeneous examples are also discussed briefly.

\subsection{The Taubian-Calabi metric}
Recall that in Example \ref{C^2n} we have a $\mathbb{S}^1$-symmetry of $\mathbb{C}^{2n}$, its Taub-NUT deformation is known as the Taubian-Calabi metric. Let $\Gamma=\mathbb{Z}_{2n}$ be the cyclic group acting on $\mathbb{C}^{2n}$ generated by $(z,w)\mapsto (\zeta_{2n}z,\zeta_{2n}w)$, where $\zeta_{2n} = e^{\frac{\pi i}{n}}$, $z,w\in \mathbb{C}^n$. One verifies that $\Gamma$ commutes with the $\mathbb{S}^1$-symmetry so $\Phi_a$ is $\Gamma$-invariant, and that $\Gamma$ preserves $x_2^2 + x_3^2 = |z_1w_1+\dots z_nw_n|^2$. Now $\mathbb{C}^{2n}/\Gamma$ admits a crepant resolution $\pi: \mathcal{K}_{\mathbb{CP}^{2n-1}}\rightarrow \mathbb{C}^{2n}/\Gamma$. Thanks to the Calabi metric constructed in \cite{calabi1979metriques}, $c_1(\mathbb{CP}^{2n-1})$ is a compactly supported K\"ahler class of $\mathcal{K}_{\mathbb{CP}^{2n-1}}$. Applying Theorem \ref{theorem main theorem}, we get
\begin{proposition}\label{proposition Taubian-Calabi}
The Taubian-Calabi metric on $\mathbb{C}^{2n}$ is ALF, and there are ALF Calabi-Yau metrics on $\mathcal{K}_{\mathbb{CP}^{2n-1}}$ asymptotic to the Taubian-Calabi metric.

More precisely, for $n\geq 2$, parameterized by $a>0$, we get a family of Taubian-Calabi metric $g_a$ on $\mathbb{C}^{2n}$.
For the ALF Calabi metrics $\omega$ on $\mathcal{K}_{\mathbb{CP}^{2n-1}}$, they are parameterized by $a>0,c>0$ and $\varepsilon>0$ in the sense that $\omega$ is asymptotic to $c\omega_1^a$ and $[\omega] = \varepsilon c_1(\mathbb{CP}^{2n-1}) \in H^2_c(\mathcal{K}_{\mathbb{CP}^{2n-1}},\mathbb{R})$.
\end{proposition}

\begin{remark}\label{remark unit length}
A special feature of Example \ref{C^2n} is that $|T|^2 = \rho^2$, where $T$ is the generator of the $\mathbb{S}^1$-symmetry. In this case the ODE system \eqref{ODE new system1} \eqref{ODE new system2} can be resolved explicitly. As a consequence, most calculations in Subsection \ref{section Twist coordinates} can be made more explicit.
\end{remark}

\subsection{ALF Calabi-Yau metrics on crepant resolution of isolated singularity}
Here we give a generalization of Proposition \ref{proposition Taubian-Calabi}:
\begin{theorem}\label{theorem isolated singularity}
Let $T$ be the generator of a locally free $\mathbb{S}^1$-symmetry of hyperk\"ahler cone $\mathbb{C}^{2n} = \mathbb{H}^n$ for $n\geq 1$, with hyperk\"ahler moment maps $x_1,x_2,x_3$ defined by formula \eqref{xj}. Given $a>0$, the Taub-NUT deformation $(g_a,I_i^a,\omega_i^a)$ is an ALF hyperk\"ahler metric.

Assume that $n\geq 2$ and there is a finite group $\Gamma\subset \SU(2n)$ such that
\begin{itemize}
  \item For all $\gamma \in \Gamma$, $\gamma T = \pm T$;
  \item The $\Gamma$-action preserves $x_2^2 + x_3^2$;
  \item There exists an $I_1-$holomorphic crepant resolution $\pi: Y\rightarrow \mathbb{C}^{2n}/\Gamma$ of the isolated singularity $\mathbb{C}^{2n}/\Gamma$.
\end{itemize}
Then for any K\"ahler class in $Y$, any $a>0,c>0$, there exists an ALF Calabi-Yau metric in this class asymptotic to $c\omega_1^a$.
\end{theorem}
\begin{proof}
We note that $\gamma T = \pm T$ implies that $x_1$ is either preserved or reversed by $\gamma$, in the latter case, the vector field $-I_1T$ is also reversed, so after all the map $\Phi_a$ is $\Gamma$-invariant. Consequently, this group action satisfies the assumption of Theorem \ref{theorem main theorem}. To apply theorem \ref{theorem main theorem}, it suffices to note that crepant resolution $Y$ of isolated singularity $\mathbb{C}^{2n}/\Gamma$ satisfies $H^2_c(Y,\mathbb{R}) = H^2(Y,\mathbb{R})$.
\end{proof}
As an application of Theorem \ref{theorem isolated singularity}, we have
\begin{example}\label{example C^2n generalized}
For $a_1,\dots,a_n\in\mathbb{Z}$, consider the following $\mathbb{S}^1$-action on $\mathbb{S}^{4n-1}$:
\begin{align}\label{S1 action on Hn}
e^{it}\cdot(u_1,\dots,u_n) = (u_1e^{ia_1t},\dots, u_ne^{ia_nt}).
\end{align}
In the above formula, we think of $\mathbb{S}^{4n-1}\subset\mathbb{H}^n$ and let $(u_1,\dots,u_n)\in \mathbb{H}^n$ be coordinates. It is clear that this action is locally free if and only if $a_\alpha \neq 0$ for all $\alpha = 1,\dots,n$. Applying the first part of Theorem \ref{theorem isolated singularity}, we get infinitely many different ALF metrics on $\mathbb{C}^{2n}$. To see that they are potentially different, we look at the concrete case where $n=2$, $a_1 = 1$, $a_2 = k$ for $k\geq 1$. Then $M_0 = \mu^{-1}(0)/\mathbb{S}^1 = \mathbb{C}^2/\mathbb{Z}_{k+1}$, so by Proposition \ref{asymptotic cone of g_a} the asymptotic cone of the ALF metric is $\mathbb{C}^2/\mathbb{Z}_{k+1} \times \mathbb{R}^3$. Here we mention that the orbifold $M_0$ for different choices of $n\geq 3$ and $a_1,\dots,a_n$ is very carefully studied in \cite{BGM1}.
\end{example}

Now we discuss a special case of the previous example.

\begin{example}\label{example EH blown down}
In the previous example, if we fix $n=2$, $a_1 = 1$, $a_2 = 1$, then $M_0 = \mu^{-1}(0)/\mathbb{S}^1 = \mathbb{C}^2/\mathbb{Z}_{2}$ and in fact for $x_1>0$ we know that $M_{x_1,0,0} = \mu^{-1}(x_1,0,0)/\mathbb{S}^1$ is the Eguchi-Hanson space. To see this, one may apply Proposition \ref{proposition omega_1_x_1,0,0} and equation \eqref{f_x_1,0,0} to find its K\"ahler potential (pushed to $M_0$) is given by
\begin{align}
    f_{x_1, 0, 0}=\sqrt{r^4+x_1^2}+2x_1\op{log}r - x_1\op{log}(\sqrt{r^4+x_1^2}+x_1).
\end{align}
here $r$ is the function on $M_0$ defined by $2r^2 = \rho^2$ and one recognizes that this is the potential of the Eguchi-Hanson metric on $T^*\mathbb{CP}^1$. The formula can be given explicitly because $|T|^2_{g_0} = \rho^2$, see Remark \ref{remark unit length}. In this case, $ED_{x_1,0,0}/\mathbb{S}^1$ is the zero section of $T^*\mathbb{CP}^1$ and is of size $\sqrt{x_1}$. Since $g_{x_1,0,0}$ is isometric to $x_1g_{1,0,0}$, the $C^0$-norm of its sectional curvature is of the order $\frac{1}{x_1}$. Using explicit formula, one can show that the metric $\Psi^*g_a$ is globally dominated by
\begin{align}
    a^2\sum_{j = 1}^3(dx_j)^2 + g_x + \frac{1}{a^2}\eta^2.
\end{align}
Along the singular direction $\{\Psi(m,x)\mid m=0\in\mu^{-1}(0),x\in \mathbb{R}^3\}$, $\rho_a$ is of order $|x|$, so the sectional curvature of $g_x$ is of order $\frac{1}{\rho_a}$, which intuitively explains the $O(\frac{1}{\rho_a})$ decay of the curvature. Also, when we blow down the metric $\lambda g$ by letting $\lambda\rightarrow 0$, the Eguchi-Hanson metric along the fiber over singular direction is blown down to $\mathbb{C}^2/\mathbb{Z}_2$ (see also Proposition \ref{proposition M_x_1,0,0 resolves M_0,0,0}), forming the singular locus $\{0\}\times \mathbb{R}^3$ in the asymptotic cone $(\mathbb{C}^2/\mathbb{Z}_2)\times \mathbb{R}^3$. In summary, both the $O(\frac{1}{\rho_a})$ decay of the curvature and the singularity of asymptotic cone are consequences of the slow growth rate of $g_x$.
\end{example}

Then we briefly discuss the case of $n = 1$.
\begin{example}\label{example n=1}
According to Kronheimer \cite{Kronheimer1}, it is known that any isolated singularity of type $\mathbb{C}^2/\Gamma$ admits an ALF hyperk\"ahler crepant resolution. It can be verified directly that when $\Gamma$ is the cyclic group $\mathbb{Z}_k$ for $k\geq 1$ or the binary dihedral group $D_k$ of order $4(k-2)$ ($k\geq 3$), then $\Gamma$ satisfies the assumption of Theorem \ref{theorem isolated singularity}. However we cannot apply the theorem because we assumed $n\geq 2$ in the theroem. Checking the proof of the theorem, we find that the only place we use the assumption $n\geq 2$ is Lemma \ref{lemma hat omega}, and we use it to produce $\hat{\omega}$. For this kind of ``Kummer construction'' of 4-dimensional ALF metrics, we refer to \cite{biquard2011kummer}. In their works, the construction of approximately Ricci-flat metric is more delicate.
\end{example}

Finally we turn to higher-dimensional crepant resolutions.
\begin{example}\label{example K_CP^2n-1 generalized}
We have already seen in Proposition \ref{proposition Taubian-Calabi} that $\mathcal{K}_{\mathbb{CP}^{2n-1}}$ is a crepant resolution of $\mathbb{C}^{2n}/\mathbb{Z}_{2n}$ admitting K\"ahler metrics. So for any choices $a_1,\dots,a_n\in \mathbb{Z}\setminus\{0\}$, the $\mathbb{S}^1$-symmetry defined by \eqref{S1 action on Hn} leads to new ALF Calabi-Yau metrics on $\mathcal{K}_{\mathbb{CP}^{2n-1}}$.
\end{example}

In general, it is difficult to determine whether there is a crepant resolution of $\mathbb{C}^{2n}/\Gamma$ for $n\geq 2$.
\begin{example}
In \cite[Theorem 3.1]{crepantHIS}, four types of finite subgroup $G\subset \SU(4)$ are defined, and it is claimed that $\mathbb{C}^4/G$ admits crepant resolution. Let $\mathbb{S}^1$
acts on $\mathbb{C}^4=\mathbb{H}^2$ by \eqref{S1 action on Hn} with $a_1,a_2\neq 0$, then it can be shown that the group $G$ of type $(i)$ does not preserve $x_2^2+x_3^2$, but the group of types $(ii),(iii),(iv)$ satisfies the assumption of Theorem \ref{theorem isolated singularity}.
It follows that there exist ALF Calabi-Yau metrics in every K\"ahler class (if there is any) of the corresponding crepant resolution.
\end{example}

\subsection{ALF Calabi-Yau metrics on the canonical bundle of the twistor space of the regular 3-Sasakian manifold}
We give another generalization of Proposition \ref{proposition Taubian-Calabi}.
\begin{theorem}\label{theorem twistor space}
Assumption as in Theorem \ref{theorem main theorem}(i). Assume furthermore that $S$ is regular and let $\mathcal{K}_Z$ be the canonical bundle of the twistor space $Z$ of $S$. Then for any $a>0,c>0, \varepsilon>0$, there exists an ALF Calabi-Yau metric defined on $\mathcal{K}_Z$ in the class $\varepsilon c_1(Z)$ which is asymptotic to $c\omega_1^a$.
\end{theorem}
\begin{proof}
For simplicity, we first assume that $S$ is simply connected. Let $I(Z)$ be the Fano index of $Z$. It is known (see \cite{boyer2008sasakian}) that we can identify $C(S)/\mathbb{Z}_{I(Z)}$ with $\mathcal{K}_Z^\times$, the complement of its zero section in the total space of the canonical bundle $\mathcal{K}_Z$, as a smooth manifold. Here $\mathbb{Z}_{I(Z)}$ acts on $C(S)$ as a cyclic subgroup of the Reeb $\mathbb{S}^1$-action generated by $\xi_1$, and $\mathcal{K}_Z \rightarrow \mathcal{K}_Z^\times = M/\mathbb{Z}_{I(Z)}$ is a crepant resolution. Next one verifies that $\xi_1$ commutes with $T$ and preserves $x_2^2 + x_3^2$ (see Proposition \ref{Sp(1)-equivariance, generalized}), so is $\mathbb{Z}_{I(Z)}$. Since the class of the Calabi metric is compactly supported, we may apply Theorem \ref{theorem main theorem}(ii) to $\Gamma = \mathbb{Z}_{I(Z)}$ and $\pi:\mathcal{K}_Z\rightarrow M/\mathbb{Z}_{I(Z)}$.

For more general cases where $S$ is not simply connected, it suffices to replace $I(Z)$ by a divisor of $I(Z)$.
\end{proof}
Note that homogeneous 3-Sasakian manifolds are regular. (In fact, it is conjectured that the converse is also true.) So we may apply Theorem \ref{theorem twistor space} to homogeneous 3-Sasakian manifolds. To begin with, in the case of $S = \mathbb{S}^{4n-1} = \Sp(n)/\Sp(n-1)$, we have $M = C(S) = \mathbb{H}^n$ and $Z = \mathbb{CP}^{2n-1}$, so we have already discussed it in Example \ref{example K_CP^2n-1 generalized}.

Next we consider the case of $S = \SU(m)/\op{S}(\U(m-2)\times \U(1))$ for $m\geq 3$.
\begin{example}
For $m\geq 3$, the 3-Sasakian quotient of $\mathbb{S}^{4m-1}$ with respect to the action \eqref{S1 action on Hn} with $a_1=\dots = a_m =1$ is known to be the homogeneous 3-Sasakian manifold $\SU(m)/\op{S}(\U(m-2)\times \U(1))$. We will explain it briefly following \cite[Section 6]{BGM1}.

The moment maps $\mu_1,\mu_2,\mu_3$ of this action are given by \eqref{x_1 of C2n} and \eqref{x_2 x_3 of C2n}. Let $N(\mathbb{C}) = \mu^{-1}(0)\cap \mathbb{S}^{4n-1}$, then
\begin{align}
N(\mathbb{C}) = \Big\{(z,w)\in \mathbb{C}^m\times\mathbb{C}^m \mathrel{\Big|} |z|^2 = |w|^2 = \frac{1}{2}, \sum_{\alpha = 1}^mz_\alpha w_\alpha = 0 \Big\}.
\end{align}
From this we see that $N(\mathbb{C})$ can be identified with the complex Stiefel manifold $V^\mathbb{C}_{m,2}$. Thinking of the $\mathbb{S}^1$-action as a subgroup of the $\Sp(m)$-action, then by \cite[Lemma 6.13]{BGM1}, its centralizer in $\Sp(m)$ is $\U(m)$. The restriction of the $\U(m)$-action on $N(\mathbb{C})$ gives $N(\mathbb{C}) = \U(m)/\U(m-2)$.

The 3-Sasakian quotient of $\mathbb{S}^{4m-1}$ with respect to the $\mathbb{S}^1$-action is the smooth manifold $N(\mathbb{C})/\mathbb{S}^1 = \U(m)/(\U(m-2)\times \U(1)) = \SU(m)/\op{S}(\U(m-2)\times \U(1))$. The $\U(m)$-action on $N(\mathbb{C})$ descends to the quotient, making $\SU(m)/\op{S}(\U(m-2)\times \U(1))$ a homogeneous 3-Sasakian manifold.

Now consider an $\mathbb{S}^1-$subaction of this $\U(m)$-action, after a conjugation in $\U(m)$, we may assume that the $\mathbb{S}^1$-action is given by the diagonal matrix
\begin{align}\label{S1 action on N(C)}
\op{diag}(e^{ia_1t},\dots, e^{ia_mt}),
\end{align}
here $a_\alpha\in\mathbb{Z}$.
\end{example}

\begin{proposition}\label{proposition locally freeness of SU(m)}
The $\mathbb{S}^1$-action given by \eqref{S1 action on N(C)} on $\SU(m)/\op{S}(\U(m-2)\times \U(1))$ is locally free if and only if all $a_\alpha$ are distinct.
\end{proposition}
\begin{proof}
Denote by $T^\p$ the generator of the $\mathbb{S}^1$-action given by \eqref{S1 action on N(C)} on $\mathbb{H}^m$, then $T^\p = (a_1u_1i,\dots, a_mu_mi)$. And define $T = ui$. Then $T^\p$ descends to a non-vanishing vector field on $N(\mathbb{C})/\mathbb{S}^1$ if and only if $T^\p$ is not parallel to $T$ on $N(\mathbb{C})$, which is equivalent to
\begin{align}
|g_0(T^\p,T)|^2 < |T^\p|^2|T|^2.
\end{align}
Now $|T|^2 = \sum_{\alpha = 1}^m|u_\alpha|^2$, $|T^\p|^2 = \sum_{\alpha = 1}^ma_\alpha^2|u_\alpha|^2$, and $g_0(T^\p,T) = \sum_{\alpha=1}^ma_\alpha|u_\alpha|^2$.
Direct calculation shows that
\begin{align}\label{Cauchy inequality difference}
|T^\p|^2|T|^2 - |g_0(T,T^\p)|^2 = \sum_{1\leq \alpha < \beta \leq m}(a_\alpha - a_\beta)^2|u_\alpha|^2|u_\beta|^2.
\end{align}
If there exists $\alpha\neq \beta$ such that $a_\alpha = a_\beta$, then there exists $u_\alpha,u_\beta\in\mathbb{H}$ such that $|u_\alpha|^2 + |u_\beta|^2 = 1$, $z_\alpha w_\alpha + z_\beta w_\beta = 0$, $\frac{1}{2}(|z_\alpha|^2 - |w_\alpha|^2) + \frac{1}{2}(|z_\beta|^2 - |w_\beta|^2) = 0$, so $(0,\dots,0,u_\alpha,0,\dots, 0, u_\beta, 0, \dots, 0)\in N(\mathbb{C})$ and $|g_0(T^\p,T)|^2 = |T^\p|^2|T|^2$ at this point. Here we identify $u_\alpha$ with $z_\alpha + w_\alpha j$

Conversely, if all the $a_\alpha$ are distinct, then we claim that $|g_0(T^\p,T)|^2 < |T^\p|^2|T|^2$. For if $|g_0(T^\p,T)|^2 = |T^\p|^2|T|^2$, then by \eqref{Cauchy inequality difference}, there is at most one nonzero $u_\alpha$, and $(0,\dots,0,u_\alpha,0,\dots,0)\in N(\mathbb{C})$ implies that $|u_\alpha|^2 = 1$, $z_\alpha w_\alpha = 0$ and $|z_\alpha|^2 = |w_\alpha|^2$ so $u_\alpha = 0$, which is impossible.
\end{proof}
Applying Theorem \ref{theorem twistor space}, we get
\begin{proposition}
There exist ALF hyperk\"ahler metrics on the cone over $S = \SU(m)/\op{S}(\U(m-2)\times \U(1))$ and ALF Calabi-Yau metrics on the canonical bundle of its twistor space.
\end{proposition}

Then we consider the case of $S = \SO(l)/(\SO(l-4)\times \Sp(1))$ for $l\geq 5$.
\begin{example}
For $l\geq 5$, it is known that the 3-Sasakian quotient of $\mathbb{S}^{4l-1}$ with respect to the right multiplication by $\Sp(1)$ is the homogeneous 3-Sasakian manifold $S = \SO(l)/(\SO(l-4)\times \Sp(1))$. Again we shall follow \cite[Section 6]{BGM1}.

For $u=(u_1,\dots,u_l)\in \mathbb{H}^l$, write $u_\alpha = u_\alpha^0 + u_\alpha^1i + u_\alpha^2j + u_\alpha^3k$. For $p = 0,1,2,3$, write $u^p = (u_1^p,\dots,u_k^p)^t\in \mathbb{R}^l$. Let $\mu:\mathbb{H}^l\rightarrow \mathfrak{sp}(1)^*\otimes \mathbb{R}^3$ be the hyperk\"ahler moment map, and let $N(\mathbb{H}) = \mu^{-1}(0)\cap \mathbb{S}^{4l-1}$. Then we have
\begin{align}\label{N_H}
N(\mathbb{H}) = \Big\{u\in \mathbb{H}^l \mathrel{\Big|}|u^0|^2 = |u^1|^2 = |u^2|^2 = |u^3|^2 = \frac{1}{4}, g_0(u^p,u^q) = \delta_{pq}\Big\}.
\end{align}
From this one sees that $N(\mathbb{H})$ can be identified with the real Stiefel manifold $V^\mathbb{R}_{l,4}$. Thinking of the $\Sp(1)$-action as the diagonal subgroup of the $\Sp(l)$-action by right multiplication, according to \cite[Lemma 6.13]{BGM1}, its centralizer in $\Sp(l)$ is $\op{O}(l)$, and the restriction of the $\op{O}(l)$-action on $N(\mathbb{H})$ gives $N(\mathbb{H}) = \op{O}(l)/\op{O}(l-4) = \SO(l)/\SO(l-4)$.

The 3-Sasakian quotient of $\mathbb{S}^{4l-1}$ with respect to the $\Sp(1)$-action is $N(\mathbb{H})/\Sp(1) = \SO(l)/(\SO(l-4)\times \Sp(1))$. The $\op{O}(l)$-action on $N(\mathbb{H})$ descends to the quotient, making the quotient a homogeneous 3-Sasakian manifold.

Now consider an $\mathbb{S}^1$-subaction on $\SO(l)/(\SO(l-4)\times \Sp(1))$, after a conjugation in $\op{O}(l)$, we may assume that the infinitesimal generator of the action is
\begin{align}\label{S1 action on N(H)}
\left(
\begin{matrix}
\left(
\begin{matrix}
    & b_1  \\
   -b_1 &
\end{matrix}
\right) &            & & \\
        &  \ddots    & & \\
        &            &  \left(
                        \begin{matrix}
                            & b_s  \\
                        -b_s &
                        \end{matrix}
                        \right) & \\
        &           &  & 0
\end{matrix}
\right)
\end{align}
if $l=2s + 1$ is odd. While if $l = 2s$ is even, then there is no $0$ in the last line. Here $b_\beta\in \mathbb{Z}$ for $\beta = 1,\dots,s$.
\end{example}
\begin{proposition}\label{proposition locally freeness of SO(l)}
The $\mathbb{S}^1$-action generated by \eqref{S1 action on N(H)} is locally free on the 3-Sasakian manifold $\SO(l)/(\SO(l-4)\times \Sp(1))$ if and only if all the $|b_\beta|$ are distinct.
\end{proposition}
\begin{proof}
Denote by $T^\p$ the generator of the action given by \eqref{S1 action on N(H)} on $\mathbb{H}^l$. For $q\in \op{Im}(\mathbb{H})$, $q\neq 0$, denote by $T_q$ the vector field $uq$ on $\mathbb{H}^l$, then $T^\p$ descends to a non-vanishing vector field on $\SO(l)/(\SO(l-4)\times \Sp(1))$ if and only if $T^\p$ is not parallel to any $T_q$ on $N(\mathbb{H})$, equivalently, for any $q$ as above, we need
\begin{align}
|g_0(T^\p,T_q)|^2 < |T^\p|^2|T_q|^2.
\end{align}
For the left hand side, we calculate that
\begin{align}
g_0(T^\p,T_q) = \sum_{\beta = 1}^s2b_\beta\op{Re}(u_{2\beta}q{\bar u_{2\beta-1}}),
\end{align}
so we have
\begin{align}
|g_0(T^\p,T_q)|^2 \leq 4|q|^2\left(\sum_{\beta=1}^s|b_\beta||u_{2\beta}||u_{2\beta-1}| \right)^2.
\end{align}
For the right hand side, we have
\begin{align}
|T^\p|^2 &= \sum_{\beta = 1}^s b_\beta^2(|u_{2\beta-1}|^2 + |u_{2\beta}|^2), \\
|T_q|^2  &= |q|^2\sum_{\alpha = 1}^l|u_\alpha|^2 \geq |q|^2 \sum_{\beta=1}^s(|u_{2\beta-1}|^2 + |u_{2\beta}|^2).
\end{align}
It follows that
\begin{align}
|T^\p|^2|T_q|^2 &\geq |q|^2\left(\sum_{\beta = 1}^s |b_\beta|(|u_{2\beta-1}|^2 + |u_{2\beta}|^2) \right)^2 + R \\
                &\geq 4|q|^2\left(\sum_{\beta=1}^s|b_\beta||u_{2\beta}||u_{2\beta-1}| \right)^2 + R \\
                &\geq |g_0(T^\p,T_q)|^2 + R,
\end{align}
where
\begin{align}
R = |q|^2 \sum_{1\leq \beta<\gamma \leq s}(|b_\beta| - |b_\gamma|)^2 (|u_{2\beta-1}|^2 + |u_{2\beta}|^2 )^2 (|u_{2\gamma-1}|^2 + |u_{2\gamma}|^2 )^2.
\end{align}
If $|b_\beta|$ are all distinct for $\beta = 1,\dots,s$, we will show by contradiction that $R>0$ on $N(\mathbb{H})$. Suppose $R=0$, then there is at most one $\beta$ such that $|u_{2\beta-1}|^2 + |u_{2\beta}|^2 > 0$, so there are at most three non-zero $u_\alpha$: $u_{2\beta-1},u_{2\beta},u_l$. In this case the $l$ by $4$ matrix $(u^0, u^1, u^2, u^3)$ is of rank at most $3$. On the other hand, by the description \eqref{N_H} of $N(\mathbb{H})$, $u^0,\dots,u^3$ are non-zero and mutually orthogonal, so the rank of the matrix is $4$, a contradiction. So we have shown that if $|b_\beta|$ are distinct, then the action generated by \eqref{S1 action on N(H)} is locally free on $\SO(l)/(\SO(l-4)\times \Sp(1))$.

Conversely, suppose that $|b_\beta|$ are not all distinct, then without loss of generality, we may assume that $|b_1| = |b_2|$. If $b_1=b_2\geq 0$, then taking $q=i$, $u_1=\frac{1}{2}$, $u_2=-\frac{1}{2}i$, $u_3 =\frac{1}{2}j$, $u_4=\frac{1}{2}k$ and $u_5 = \dots = u_l = 0$. One verifies that at $u\in N(\mathbb{H})$, the vector $T^\p$ is parallel to $T_q$. The other possibilities $b_1 = -b_2 \geq 0$, $b_1 = -b_2\leq 0$ and $b_1=b_2\leq 0$ can be treated in a similar way.
\end{proof}
Applying Theorem \ref{theorem twistor space}, we get
\begin{proposition}
There exist ALF hyperk\"ahler metrics on the cone over $S = \SO(l)/(\SO(l-4)\times \Sp(1))$ and ALF Calabi-Yau metrics on the canonical bundle of its twistor space.
\end{proposition}
\begin{remark}
This is our first ALF metric modeled on a Taub-NUT deformation of a non-toric hyperk\"ahler cone.
\end{remark}

Finally, we consider the case of $S=G_2/\Sp(1)$.
\begin{example}
According to \cite[Corollary 13.6.1]{boyer2008sasakian}, the 3-Sasakian manifold $\mathbb{S}^{27}$ admits an action of $\U(1)\times \Sp(1)$ such that the 3-Sasakian quotient is the compact 3-Sasakian orbifold $\mathbb{Z}_3\setminus G_2/\Sp(1)$.

More precisely, the $\Sp(1)$ acts on $\mathbb{H}^7$ by right multiplication, while $\U(1)$ acts on $\mathbb{H}^7$ by right multiplication with
\begin{align}\label{f(t)}
f(t) =
\left(
\begin{matrix}
A(t) &            & & \\
        &  A(t)    & & \\
        &            &  A(t) & \\
        &           &  & 1
\end{matrix}
\right),
\end{align}
where for $t\in [0,2\pi)$,
\begin{align}
A(t) =
\left(
\begin{matrix}
   \cos t & \sin t  \\
   -\sin t & \cos t
\end{matrix}
\right).
\end{align}
Its infinitesimal generator is given by \eqref{S1 action on N(H)} with $b_1 = b_2 = b_3 = 1$. This $\U(1)$-action has hyperk\"ahler moment map $\nu: \mathbb{H}^7\rightarrow \op{Im}(\mathbb{H}) = \mathbb{R}^3$ defined by
\begin{align}
\nu(u) = \sum_{\beta = 1}^3 (u_{2\beta - 1}\bar{u}_{2\beta} - u_{2\beta}\bar{u}_{2\beta-1}).
\end{align}
Let $N_\nu = N(\mathbb{H})\cap \nu^{-1}(0)$, then the 3-Sasakian quotient of $\mathbb{S}^{27}$ with respect to $\U(1)\times \Sp(1)$ is $N_{\nu}/(\U(1)\times \Sp(1)) = \mathbb{Z}_3\setminus G_2/\Sp(1)$.

For $b_1,b_2,b_3\in \mathbb{Z}$, consider the $\mathbb{S}^1$-action on $\mathbb{H}^7$ generated by \eqref{S1 action on N(H)}. It is clear that this action commutes with $\U(1)\times \Sp(1)$, so it descends to an $\mathbb{S}^1$-action on $\mathbb{Z}_3\setminus G_2/\Sp(1)$.
\end{example}
\begin{proposition}\label{proposition locally freeness of G_2}
If $b_1,b_2,b_3$ are all distinct, then the $\mathbb{S}^1$-action generated by \eqref{S1 action on N(H)} is locally free on $\mathbb{Z}_3\setminus G_2/\Sp(1)$.
\end{proposition}
\begin{proof}
Let $T^\p$ denote the generator of \eqref{S1 action on N(H)} on $\mathbb{H}^7$, let $T$ denote the generator of \eqref{f(t)} on $\mathbb{H}^7$, and $T_q(u) = uq$ the generator of right $\Sp(1)$ multiplication on $\mathbb{H}^7$.

If suffices to show that for any $\lambda \in \mathbb{R}$ and $0\neq q\in \op{Im}(\mathbb{H})$, the vector $T^\p-\lambda T$ is not parallel to $T_q$ on $N_\nu$.

If $|b_1-\lambda|$, $|b_2 - \lambda|$ and $|b_3-\lambda|$ are all distinct, then by the proof of Proposition \ref{proposition locally freeness of SO(l)}, $T^\p-\lambda T$ is not parallel to any $T_q$ on $N$, hence on $N_\nu$.

Since $b_1,b_2,b_3$ are distinct, it leaves the following possibility: $\lambda$ is the average of a pair $b_{\beta_1}, b_{\beta_2}$ so $b_{\beta_1} - \lambda = -(b_{\beta_2} - \lambda)$. Without loss of generality, assume that this pair is $b_1,b_2$ and assume $b_1<b_2$. We will show that in this case $T^\p-\lambda T$ is not parallel to any $T_q$ on $N_\nu$.

Examining the proof of Proposition \ref{proposition locally freeness of SO(l)}, if $T^\p-\lambda T$ is parallel to $T_q$ at $u$, then $u_5 = u_6 = u_7 = 0$, $u_2q\bar{u}_1$ and $u_4q\bar{u}_3$ are real with opposite signs.

First we consider a particular point $u_1 = \frac{1}{2}$, $u_2 = \frac{1}{2}i$, $u_3 = \frac{1}{2}k$, $u_4 = \frac{1}{2}j$, then $u_1\bar{u}_2 + u_3\bar{u}_4 = 0$ so $u\in N_\nu$. Since we want $u_2q\bar{u}_1$ and $u_4q\bar{u}_3$ to be real, the only choice of $q$ is $q = q_1i$ for some $q_1\in \mathbb{R}$, then $u_2q\bar{u}_1 = -q_1$, $u_4q\bar{u}_3 = -q_1$ and they have the same sign. So there is no $0\neq q \in \op{Im}(\mathbb{H})$ such that $T^\p-\lambda T$ is parallel to $T_q$ at $u$.

Next we turn to other possibilities of $u_1,u_2,u_3,u_4$. By description \eqref{N_H}, $(u_1,u_2,u_3,u_4)$ differs from $\frac{1}{2}(1,i,k,j)$ by a rotation of $\SO(4) = \Sp(1)_+\Sp(1)_-$.

If $u_\alpha^\p = q^\p u_\alpha$ for some $q^\p\in \Sp(1)$, then $u_{2\beta}^\p q \bar{u}_{2\beta-1}^\p = q^\p u_{2\beta} q \bar{u}_{2\beta-1} \bar{q}^\p = u_{2\beta} q \bar{u}_{2\beta-1}$ if $u_{2\beta} q \bar{u}_{2\beta-1}$ is real.

If $u_\alpha^{\p\p} = u_\alpha q^{\p\p}$ for some $q^{\p\p} \in \Sp(1)$, then $u_{2\beta}^{\p\p} q \bar{u}_{2\beta-1}^{\p\p} = u_{2\beta} (q^{\p\p}q \bar{q}^{\p\p}) \bar{u}_{2\beta-1}$.

So by the result at $\frac{1}{2}(1,i,k,j,0,0,0)$ and the discussions above, the result for general $u$ follows.
\end{proof}
Applying Theorem \ref{theorem twistor space}, we get
\begin{proposition}
There exist ALF hyperk\"ahler metrics on the cone over $S = G_2/\Sp(1)$ and ALF Calabi-Yau metrics on the canonical bundle of its twistor space.
\end{proposition}

\subsection{Some other examples}
If one only wants to apply the first part of Theorem \ref{theorem main theorem} then there are many situations where a 3-Sasakian manifold admits a locally free $\mathbb{S}^1$-action. In this subsection we will give some examples.

\begin{example}
Choosing $a = (a_1,\dots, a_n)\in (\mathbb{Z}\setminus{0})^n$ with $\gcd(a_\alpha,a_\beta) = 1$ for any $\alpha\neq \beta$, then the 3-Sasakian quotient $S(a)$ of $\mathbb{H}^n$ with respect to the $\mathbb{S}^1$-action given by \eqref{S1 action on Hn} will be a smooth 3-Sasakian manifold (see \cite[Theorem 13.7.6]{boyer2008sasakian}). And a proof similar to Proposition \ref{proposition locally freeness of SU(m)} shows that the $\mathbb{S}^1$-action $ue^{it}$ descends to a locally free symmetry of $S(a)$. So, there is an ALF hyperk\"ahler metric on the cone over $S(a)$ obtained as a Taub-NUT deformation.
\end{example}
This example suggests that there are potentially many examples of locally free $\mathbb{S}^1$-symmetry on toric hyperk\"ahler manifolds. As for a non-toric example, we have
\begin{example}
Take $0 < b_1 < b_2 < b_3$, with $b_i\in \mathbb{Z}$ such that $\gcd(b_{\alpha}, b_\beta) = 1$ for $\alpha \neq \beta$ and $\gcd(b_1 \pm b_2, b_1 \pm b_3) = 1$. Denote by $\U(1)_b$ the $\U(1)$-action generated by \eqref{S1 action on N(H)} on $\mathbb{H}^7$, here $b = (b_1,b_2,b_3)$. Then the 3-Sasakian quotient $\mathbb{S}^{27}/\!\!/\!\!/\ (\U(1)_b\times \Sp(1))$ is a smooth 3-Sasakian manifold (see \cite[Theorem 13.9.7]{boyer2008sasakian}). A proof similar to the proof of Proposition \ref{proposition locally freeness of G_2} shows that $f(t)$ defined by \eqref{f(t)} generates a locally free $\mathbb{S}^1$-action on $\mathbb{S}^{27}/\!\!/\!\!/\ (\U(1)_b\times \Sp(1))$. Applying Theorem \ref{theorem main theorem}, its Taub-NUT deformation is an ALF hyperk\"ahler metric on the cone over $\mathbb{S}^{27}/\!\!/\!\!/\ (\U(1)_b\times \Sp(1))$.
\end{example}

\bibliographystyle{abbrv}
\bibliography{bibnoteAconstructionofan8dimensionalALFmetricrevision}

\end{document}